%% file: chexnew.tex
\documentclass[1p,preprint]{elsarticle}

\usepackage{amsmath}
\usepackage{amsthm}
\usepackage{amsfonts}
\usepackage{color}
\usepackage{amssymb}
\usepackage{hyperref}
\usepackage[all,2cell]{xy}
\UseAllTwocells
\xyoption{v2}
\NoResizing
\UseResizing
\DeclareMathAlphabet{\mathpzc}{OT1}{pzc}{m}{it}
\usepackage{enumitem}   

\theoremstyle{theorem}
\newtheorem{theorem}{Theorem}[section]
\newtheorem*{theorem*}{Theorem}
\newtheorem{corollary}[theorem]{Corollary}
\newtheorem*{corollary*}{Corollary}
\newtheorem{proposition}[theorem]{Proposition}
\newtheorem{lemma}[theorem]{Lemma}
\theoremstyle{remark}
\newtheorem{remark}[theorem]{Remark}
\newtheorem{example}[theorem]{Examples}
\theoremstyle{definition}
\newtheorem{definition}[theorem]{Definition}

\include{macros}

\begin{document}
\begin{frontmatter}

\title{Generalized existential completions and their regular and exact completions}
\author{Maria Emilia Maietti\fnref{milly}}
 
\ead{maietti@math.unipd.it}

\address[milly]{Department of Mathematics, Via Trieste, 63, 35131 Padova, Italy}
\author{Davide Trotta\fnref{davide}}


\ead{trottadavide92@gmail.com}
\address[davide]{University of Pisa, Largo Bruno Pontecorvo, 3, 56127 Pisa, Italy}
\cortext[davide]{The research of the second author is partially supported by the MIUR PRIN 2017FTXR "IT-MaTTerS".
}
\begin{abstract}

%
%
%

This paper aims to apply the tool of generalized existential completions of conjunctive doctrines, concerning a class  $\Lambda$ of morphisms of their base category,
 to deepen the study of  regular and exact completions of existential  elementary Lawvere's doctrines.
 
 After providing a characterization of generalized existential completions,  we observe that   both the subobjects doctrine  $\Sub_{\mC}$ and the weak subobjects doctrine $\Psi_{\mC}$  of a category
 $\mC$  with finite limits
 are generalized existential completions of the constant true doctrine,  the first  along the class of all the monomorphisms of  $\mC$ while the latter 
 along all the morphisms of $\mC$.  We then  name {\it full existential completion} a generalized completion of a conjunctive doctrine  along
 the class of all the morphisms of its base.
 
 From this we immediately deduce that  both the regular and  the exact completion of a finite limit category are regular and exact completions
 of full existential doctrines since it is known  that both  the  regular completion $(\mD)_{\reg/\lex}$ and  the exact completion $(\mD)_{\ex/\lex}$  of a finite limit category $\mD$ are respectively the regular completion
 $\mathsf{Reg}(\Psi_{\mD})$ and the exact completion  ${\cal T}_{\Psi_{\mD} }$  (as an instance of the tripos-to-topos construction)
 of the weak subobjects doctrine $\Psi_{\mD}$ of $\mD$.

 Here we  prove  that  the condition of being a generic  full  existential completion is also sufficient to produce a regular/exact completion equivalent to a regular/exact completion of a  finite limit category.
 
 Then,  we show more specialized characterizations from which we derive known results  as well as  remarkable examples of  exact completions of full existential completions, including all realizability toposes and supercoherent localic toposes.

%
%
%
%
%
%
\end{abstract}


\end{frontmatter}

\tableofcontents

\section{Introduction}
The process of completing a category with quotients by producing an exact or regular category   introduced in \cite{SFEC,REC} has been widely studied in the literature of  category theory, with applications both
to mathematics and computer science, see \cite{Lack1999,phdMenni,MENNI2002187,QCFF}.

In particular, the notion of exact completion of a finite limit (or lex) category, for short ex/lex completion, as well as that of exact completion of a regular category, for short ex/reg completion,  have been proved to be an instance of a more general  exact completion  ${\cal T}_P$  relative to an  elementary existential Lawvere's doctrine $P$, see \cite{UEC}. The construction of the exact category $\tripostotopos{P}$ is, essentially, provided by the tripos-to-topos construction  of J.M.E. Hyland, P.T. Johnstone and A.M. Pitts, see \cite{TT,TTT}. In particular
the ex/lex completion   $\exactcomplex{\mD}$ of a category $\mD$ is the exact completion ${\cal T}_{\Psi_{\mD} }$     of the weak subobjects doctrine $\Psi_{\mD}$ of $\mD$ while the ex/reg completion  $\exactcomp{\mC}$ of a regular category $\mC$
is the exact completion ${\cal T}_{\Sub_{\mC} }$ of the subobjects doctrine $\Sub_{\mC}$ of $\mC$.

Similarly, also  the regular completion $\regularcomp{\mD}$ of a  lex category $\mD$  introduced in \cite{FECLEO} can be seen as an instance  of the  regular completion $\regularcompdoctrine{P}$ of an elementary existential doctrine $P$ introduced in \cite{TECH}, namely the regular completion  $\regularcompdoctrine{\Psi_{\mD}}$ of the weak subobjects doctrine $\Psi_{\mD}$ of $\mD$. 
 
Moreover, in \cite{UEC} it has been shown  that   the exact completion of an existential  doctrine $P$  can be decomposed as the ex/reg completion of the regular completion of $P$.
Such a decomposition
generalizes the well-known  decomposition shown in \cite{REC} stating that   the  ex/lex completion of a  lex category $\mC$  can be seen as the ex/reg completion of the regular completion of $\mC$.

Then, this analysis has been pushed further  in \cite{TECH} by observing that  the exact completion of an existential elementary doctrine $P$ can be also  decomposed as the  regular completion $\mathsf{Reg}(Q_P)$ of  the elementary quotient completion   $Q_P$ of $P$ introduced in \cite{QCFF,EQC}.

This paper contributes to the problem of establishing when the exact or  regular completion of an elementary existential doctrine happens to coincide respectively with the  ex/lex or reg/lex completion
of a lex category.

To reach our purposes we employ  the construction of the
 generalized existential completion of a conjunctive doctrine $P$, concerning a suitable class $\Lambda$ of morphisms of their base category,  originally introduced in  \cite{ECRT}.

We start by providing  a characterization of generalized existential completions based on an   algebraic description of the concept of {\it existential free formulas} by means of choice principles. By employing such a characterization, we show that examples of generalized existential completions include:
 every subobjects doctrine of a lex category, every $\mM$-subobjects doctrine relative to a $\mM$-category \cite{RCICPM,DDC}, every weak subobjects doctrine of a lex category,  every realizability tripos and, among localic triposes, exactly those  associated to a supercoherent locale in the sense of \cite{BANASCHEWSKI199145}.
In particular  both the subobjects doctrine  $\Sub_{\mC}$ and the weak subobjects doctrine $\Psi_{\mC}$  of a category
 $\mC$  with finite limits
 are generalized existential completions of the constant true doctrine,  the first  along the class of all the monomorphisms of  $\mC$ while the latter 
 along all the morphisms of $\mC$.  We then  name {\it full existential completion} a generalized completion of a conjunctive doctrine  along
 the class of all the morphisms of its base.

From this we immediately deduce  that  both the regular completion  $\regularcomp{\mD}$  and  the exact completion $\exactcomplex{\mD}$ of a finite limit category are regular and exact completions
 of  full existential completions  thanks to their above mentioned doctrinal  presentation through the weak subobjects doctrine $\Psi_{\mD}$.

Here we  prove  that  the condition of being a generic  full  existential completion is also sufficient to produce a regular/exact completion equivalent to a regular/exact completion of a  finite limit category.

%
%
For  a better exposure of  this  and related results, we introduce the notion of \emph{regular Morita-equivalence} and  \emph{exact Morita-equivalence} for elementary existential doctrines, motivated by the notion of Morita equivalence between regular theories in  \cite{SAE}. In more detail, we declare  that two elementary and existential doctrines $P$ and $R$ are  {\it regular Morita-equivalent}/ {\it exact Morita-equivalent} when their regular/exact completions are equivalent.

Employing these notions we show that 
 the regular completion   of an elementary existential doctrine $P$ is equivalent to the reg/lex completion of a lex category
  if and only if $P$ is regular Morita-equivalent to a full existential completion, or equivalently  $P$ is  regular Morita-equivalent to the weak subobjects doctrine $\Psi_{\mD}$ of  a lex category $\mD$ if and only if
  $P$ is regular Morita-equivalent to a full existential completion.
 
Furthermore, from this result and the decomposition of exact completions presented in \cite{UEC},  we also derive that the exact completion  $\tripostotopos{P}$ of an elementary existential doctrine $P$ is equivalent  to the ex/lex completion of  a lex category
 if and only if   $P$ is exact Morita-equivalent to a full existential completion, or equivalently,  $P$ is  exact Morita-equivalent to the weak subobjects doctrine $\Psi_{\mD}$ of a  lex category $\mD$ if and only if
  $P$ is exact Morita-equivalent to a full existential completion.

More in detail,  these  characterizations have been derived by   two other characterizations  identifying  regular/exact completions of an elementary existential doctrine which coincide respectively with regular /exact 
 completions of specific categories.

Indeed, we show that  the regular completion of an elementary existential doctrine $P$ on a lex category $\cal C$  coincides with the regular completion   $(\mG_{P'})_{\mathrm{ex/lex}}$
 of  the Grothendieck category of a primary subdoctrine $P'$ of $P$  (via an equivalence preserving the embedding of $P'$ in $P$ and in the weak subobjects doctrine of
 $\mG_{P'}$)  if and only if $P$
is the full existential completion of $P'$.
Combining this with  the decomposition of exact completions in \cite{UEC} we also derive that the exact completion of an existential doctrine $P$ on a lex category $\cal C$ 
coincides with the exact completion  $(\mG_{P'})_{\mathrm{ex/lex}}$  of    the Grothendieck category of a primary subdoctrine $P'$ of $P$ (via an equivalence preserving
 the embedding of $P'$ in $P$ and in the weak subobjects doctrine of
 $\mG_{P'}$)  if and only if $P$
is the   full existential completion of $P'$.
%

Relevant applications of the above results to tripos-to-topos constructions   include all  realizability toposes and localic toposes associated to a supercoherent locale \cite{BANASCHEWSKI199145}.

 Furthermore,  we show similar characterizations  for the case in which $P$ is a {\it pure existential completion} of a primary doctrine $P'$,  namely when
$P$ is  a generalized existential completion of $P'$  only with respect to the class of projections of a finite product category, a concept  introduced under the name of  ``existential completion" in  \cite{ECRT}.

More in detail,  we show that  the regular  completion of an elementary existential doctrine $P$ coincides with the regular completion  $(\Pred{P'})_{\mathrm{reg/lex}}$   of   the category of predicates of  a primary subdoctrine $P'$ of $P$ (via an equivalence preserving the embedding of $P'$ in $P$ and in the weak subobjects doctrine of
 $\Pred{P'}$) if and only if $P$
is a  pure existential completion of $P'$.
Combining again the decomposition of exact completions in \cite{UEC} with our result, we  deduce that the  exact completion of an elementary  existential doctrine $P$
coincides with the exact completion  $(\Pred{P'})_{\mathrm{ex/lex}}$   of   the category of predicates of  a primary subdoctrine $P'$ of $P$  (via an equivalence preserving the embedding of $P'$ in $P$ and in the weak subobjects doctrine of
 $\Pred{P'}$)
 if and only if $P$
is a pure existential completion of $P'$.

Finally we show that the the above characterization generalizes the characterization  in \cite{TECH} asserting that a necessary and sufficient condition for a tripos $P$ to produce  a topos $\tripostotopos{P}$ coinciding with  the exact completion of the lex category $\Pred{P}$ (via an equivalence given by the canonical exact functor from $(\Pred{P})_{\mathrm{ex/lex}}$ to $\tripostotopos{P}$) is to be equipped with Hilbert's $\epsilon$-operators.  
The crucial fact linking the two characterizations is that   $P$ is a pure existential completion of itself if and only if $P$ is equipped with Hilbert's $\epsilon$-operators.

We conclude by underlying that a  preliminary version of our characterization of generalized existential completion was presented in \url{https://www.youtube.com/watch?v=sAMVU_5RQJQ&t=688s} in 2020 and a similar notion was independently presented in \cite{Frey2020}. 
Furthermore, such a characterization  together with the notion of existential free object have already been fruitfully employed in recent works  \cite{hofstra2011,TSPGF,trottapasquale2020,dialecticalogicalprinciples} to  give a categorical version of G\"odel's dialectica interpretation \cite{goedel1986} in terms of quantifier-completions.
Instead  another related study of triposes via different completions can be found in   \cite{PHDHofstra,Hofstra06}.

%
%

\section{Preliminary notions of doctrines}\label{section notions of doctrines}
The notion of hyperdoctrine was introduced by F.W. Lawvere in a series
of seminal papers \cite{AF,EHCSAF}. We recall from \textit{loc. cit.}
some definitions which will be useful in the following. The reader can
find all the details about the theory of elementary and existential
doctrine also in \cite{QCFF,EQC,UEC,TECH,EMMENEGGER2020}.

In the following  we adopt the notation  $fg$ to mean the composition of  a morphism  $f: Y\rightarrow Z$ with another $g: X\rightarrow Y$ within a category.

We indicate with  $\set$  the category of sets formalizable within the classical axiomatic set theory  ZFC.

We introduce the following basic notion  of ``\conjdoctrine "~as a generalization of that of ``primary doctrine" in \cite{QCFF}.
\begin{definition}
A \bemph{\conjdoctrine}~is a functor
$\doctrine{\mC}{P}$ from the opposite of the category $\mC$ to the category of inf-semilattices.
\end{definition}

\begin{definition}
A  conjunctive doctrine $\doctrine{\mC}{P}$  
 is a  \bemph{primary doctrine} if the  category $\mC$ has finite products.
\end{definition}

We add the definition of \bemph{fibred subdoctrine} for conjunctive subdoctrines of conjunctive doctrines on the same base category:
\begin{definition}
A conjunctive doctrine $\doctrine{\mC}{P'}$  is said a \bemph{a fibred subdoctrine} of a conjunctive doctrine $\doctrine{\mC}{P}$ 
if each fibre of $P'(A)$ is a full subposet of $P(A)$ for every object $A$.
\end{definition}

\begin{definition}
A primary doctrine $\doctrine{\mC}{P}$   is \bemph{elementary} if for every $A$ in $\mC$ there exists an object $\delta_A$ in $P(A\times A)$ such that
\begin{enumerate}
\item the assignment
\[\Einv_{\angbr{\id_A}{\id_A}}(\alpha):=P_{\pr_1}(\alpha)\wedge \delta_A\]
for an element $\alpha$ of $P(A)$ determines a left adjoint to $\freccia{P(A\times A)}{P_{\angbr{\id_A}{\id_A}}}{PA}$;
\item for every morphism $e$ of the form $\freccia{X\times A}{\angbr{\pr_1,\pr_2}{\pr_2}}{X\times A\times A}$ in $\mC$, the assignment
\[ \Einv_{e}(\alpha):= P_{\angbr{\pr_1}{\pr_2}}(\alpha)\wedge P_{\angbr{\pr_2}{\pr_2}}(\delta_A)\]
for $\alpha$ in $P(X\times A)$ determines a left adjoint to $\freccia{P(X\times A \times A)}{P_e}{P(X\times A)}$.
\end{enumerate}
\end{definition}
\begin{definition}\label{def existential doctrine}
A primary doctrine $\doctrine{\mC}{P}$ is \bemph{existential} if, for every object $A_1$ and $A_2$ in $\mC$, for any projection $\freccia{A_1\times A_2}{{\pr_i}}{A_i}$, $i=1,2$, the functor
\[ \freccia{P(A_i)}{{P_{\pr_i}}}{P(A_1\times A_2)}\]
has a left adjoint $\Einv_{\pr_i}$, and these satisfy:
\begin{enumerate}
\item[(BCC)] \bemph{Beck-Chevalley condition:} for any pullback diagram
\[
\pullback{X'}{A'}{X}{A}{{\pr'}}{f'}{f}{{\pr}}
\]

with $\pr$ and $\pr'$ projections, for any $\beta$ in $P(X)$ the canonical arrow 
\[ \Einv_{\pr'}P_{f'}(\beta)\leq P_f \Einv_{\pr}(\beta)\]
is an isomorphism;
\item[(FR)] \bemph{Frobenius reciprocity:} for any projection $\freccia{X}{{\pr}}{A}$, for any object $\alpha$ in $P(A)$ and $\beta$ in $P(X)$, the canonical arrow
\[ \Einv_{\pr}(P_{\pr}(\alpha)\wedge \beta)\leq \alpha \wedge \Einv_{\pr}(\beta)\]
in $P(A)$ is an isomorphism.
\end{enumerate}
\end{definition}

\begin{remark}\label{rem left adjoint elementary existential doc}
In an existential elementary doctrine, for every map $\freccia{A}{f}{B}$ in $\mC$ the functor $P_f$ has a left adjoint $\Einv_f$ that can be computed as
\[ \Einv_{\pr_2}(P_{f\times {\id}_B}(\delta_B)\wedge P_{\pr_1}(\alpha))\]
for $\alpha$ in $P(A)$, where $\pr_1$ and $\pr_2$ are the projections from $A\times B$.
However, observe that such a  definition  guarantees only the validity of the corresponding Frobenius reciprocity condition for $\Einv_f$ but  it does not guarantee that  of Beck-Chevalley conditions with respect to pullbacks along $f$.
\end{remark}

\begin{example}\label{example doctrines}
The following examples are discussed in \cite{AF}.
\begin{enumerate}
\item Let $\mC$ be a category with finite limits. The functor \[\doctrine{\mC}{{\Sub_{\mC}}}\]
assigns to an object $A$ in $\mC$ the poset $\Sub_{\mC}(A)$ of subobjects of $A$ in $\mC$ and, for an arrow $\frecciasopra{B}{f}{A}$ the morphism $\freccia{\Sub_{\mC}(A)}{\Sub_{\mC}(f)}{\Sub_{\mC}(B)}$ is given by pulling a subobject back along $f$. The fiber equalities are the diagonal arrows. This is an existential elementary doctrine if and only if the category $\mC$ is regular. See \cite{FSF}. 
\item Consider a category $\mD$ with finite products and weak pullbacks: the doctrine is given by the functor of weak subobjects (or variations)
\[\doctrine{\mD}{{\Psi_{\mD}}}\]
where $\Psi_{\mD}(A)$ is the poset reflection of the slice category $\mD/A$, whose objects are indicated with $[f]$ for any arrow $\frecciasopra{B}{f}{A}$ in $\mD$, and for an arrow $\frecciasopra{B}{f}{A}$, the homomorphism $\freccia{\Psi_{\mD}(A)}{\Psi_{\mD}([f])}{\Psi_{\mD}(B)}$ is given by  the equivalence class of a weak pullback of an arrow $\frecciasopra{X}{g}{A}$ with $f$. This doctrine is existential, and the existential left adjoint are given by the post-composition.
\item Let $\theory$ be a theory in a first order language $\lang$. We define a primary doctrine 
\[\doctrine{\mC_{\theory}}{LT}\]
where $\mC_{\theory}$ is the category of lists of variables and term substitutions:
\begin{itemize}
\item \bemph{objects} of $\mC_{\theory}$ are finite lists of variables $\vec{x}:=(x_1,\dots,x_n)$, and we include the empty list $()$;
\item a \bemph{morphisms} from $(x_1,\dots,x_n)$ into $(y_1,\dots,y_m)$ is a substitution $[t_1/y_1,\dots, t_m/y_m]$ where the terms $t_i$ are built in $\signature$ on the variable $x_1,\dots, x_n$;
\item the \bemph{composition} of two morphisms $\freccia{\vec{x}}{[\vec{t}/\vec{y}]}{\vec{y}}$ and $\freccia{\vec{y}}{[\vec{s}/\vec{z}]}{\vec{z}}$ is given by the substitution
\[ \freccia{\vec{x}}{[s_1[\vec{t}/\vec{y}]/z_k,\dots, s_k[\vec{t}/\vec{y}]/z_k]}{\vec{z}}.\]
\end{itemize}
The functor $\doctrine{\mC_{\theory}}{LT}$ sends a list $(x_1,\dots,x_n)$ to the partial order  $LT(x_1,\dots,x_n)$  of equivalence classes $[\phi]$ of 
well formed formulas $\phi$  in the context $(x_1,\dots,x_n)$
where  $[\psi]\leq [\phi]$ for $\phi,\psi\in LT(x_1,\dots,x_n)$ if $\psi\vdash_{\theory}\phi$ and two formulas are equivalent if they are equiprovable in the theory.
Given a morphism of $\mC_{\theory}$ 
\[\freccia{(x_1,\dots,x_n)}{[t_1/y_1,\dots,t_m/y_m]}{(y_1,\dots,y_m)}\]
the functor $LT_{[\vec{t}/\vec{y}]}$ acts as the substitution
$LT_{[\vec{t}/\vec{y}]}(\psi(y_1,\dots,y_m))=\psi[\vec{t}/\vec{y}]$.

The doctrine $\doctrine{\mC_{\theory}}{LT}$ is elementary exactly when $\theory$ has an equality predicate and it is existential. For all the detail we refer to \cite{QCFF}, and for the case of a many sorted first order theory we refer to \cite{CLP}.
\item Let $\set_{\ast}$  be the category of non-empty sets and let $\xi$ be an ordinal with greatest element, and $\mH$ its frame. Then the doctrine
\[\doctrine{\set_{\ast}}{\mH^{(-)}}\]
is elementary and existential, in particular for every $\alpha\in \mH^{A\times B}$, the left adjoint $\Einv_{\pr_A}$ is defined as
\[\Einv_{\pr_A}(\alpha)(a)=\bigvee_{b\in B} \alpha (a,b)\]
and the equality predicate $\delta (i,j)\in \mH^{A\times A}$ is defined as the top element if $i=j$, and the bottom otherwise. See \cite{TECH} for all the details.
\end{enumerate}
\end{example}
The category of  primary doctrines $\PD$ is a 2-category, where:
\begin{itemize}
\item a \bemph{1-cell} is a pair $(F,b)$
\[\duemorfismo{\mC}{P}{F}{b}{\mD}{R}\]
such that $\freccia{\mC}{F}{\mD}$ is a functor, and $\freccia{P}{b}{R\circ F^{\op}}$ is a natural transformation.
\item a \bemph{2-cell} is a natural transformation $\freccia{F}{\theta}{G}$ from $(F,b)$ to $(G,c)$, namely from  $\freccia{P}{b}{R\circ F^{\op}}$ to $\freccia{P}{c}{R\circ G^{\op}}$,  such that for every $A$ in $\mC$ and every $\alpha$ in $P(A)$, we have 
\[b_A(\alpha)\leq R_{\theta_A}(c_A(\alpha)).\]
\end{itemize}

We denote by $\ED$ the 2-full subcategory of $\PD$ whose elements are existential doctrines, and whose 1-cells are those 1-cells of $\PD$ which preserve the existential structure.

We conclude this section by recalling some choice principles from \cite{TECH}.

\begin{definition}
Let $\doctrine{\mC}{P}$ be an elementary existential doctrine. We say that $P$ satisfies the \bemph{Rule of Unique Choice} (RUC) if for every entire functional relation $\phi$ in $P(A\times B)$ there exists an arrow $\freccia{A}{f}{B}$ such that 
\[\top_A\leq P_{\angbr{\id_A}{f}}{(\phi)}\]
\end{definition}
\begin{example}
The subobjects doctrine $\doctrine{\mC}{{\Sub_{\mC}}}$ presented in Example \ref{example doctrines} satisfies the (RUC) as observed in \cite{TECH}.
\end{example}
Now we recall the notion of  Extended Rule of Choice and its particular instance called Rule of Choice.
\begin{definition}\label{exrc}
An existential doctrine $\doctrine{\mC}{P}$  satisfies the \bemph{Extended Rule of Choice} (ERC) if for every $\phi \in P(B)$ and for every $\freccia{B}{g}{A}$ such that 
\[ \top_A \leq \Einv_{g}(\phi)\]
there exists an arrow $\freccia{A}{f}{B}$ in $\mC$ such that $gf=\id_A$ and
\[ \top_A \leq P_{f}(\phi).\]
\end{definition}
\begin{definition}
For an existential doctrine $\doctrine{\mC}{P}$, we say that $P$ satisfies the \bemph{Rule of Choice} (RC) if it satisfies the Extended Rule of Choice only
for projections, namely for every $\phi \in P(A\times B)$ such that 
\[ \top_A \leq \Einv_{\pr_1}(\phi)\]
there exists an arrow $\freccia{A}{f}{B}$ in $\mC$ such that
\[ \top_A \leq P_{\angbr{\id_A}{f}}(\phi).\]
\end{definition}

\begin{example}\label{example wek-sub doctrine}
Recall from \cite{TECH}  that the  doctrine of weak subobjects 
\[\doctrine{\mD}{{\Psi_{\mD}}}\]
presented in Example \ref{example doctrines} (2) satisfies the Extended Rule of Choice. 
\end{example}

Following \cite{TECH} we introduce:
\begin{definition}
Let $\doctrine{\mC}{P}$ be an existential doctrine. An object $B$ of $\mC$ is equipped with Hilbert's  $\epsilon$-\bemph{operators} if, for any object $A$ in $\mC$ and any $\alpha$ in $P(A\times B)$ there exists an arrow $\freccia{A}{\epsilon_{\alpha}}{B}$ such that
\[ \Einv_{\pr_1}(\alpha)=P_{\angbr{\id_A}{\epsilon_{\alpha}}}(\alpha)\]
holds in $P(A)$, where $\freccia{A\times B}{\pr_1}{A}$ is the first projection.
\end{definition}

\begin{definition}
We say that an elementary existential doctrine $\doctrine{\mC}{P}$ is \bemph{equipped with}  Hilbert's $\epsilon$-\bemph{operators} if every object in $\mC$ is equipped with  $\epsilon$-operators.
\end{definition}

\begin{example}\label{example localic doctrine ass. to ordinal}
We recall from \cite{TECH} that the doctrine
\[\doctrine{\set_{\ast}}{\mH^{(-)}}\]
presented in Example \ref{example doctrines} (4) is equipped with $\epsilon$-operators. In particular for every element $\alpha\in \mH^{A\times B}$, and for every $a\in A$ one can consider the (non empty) set $$I_{\alpha}(a)=\{b\in B \;|\; \alpha(a,b)=\bigvee_{c\in B}\alpha (a,c)\}.$$ Then, by the axiom of choice, there exists a function $\freccia{A}{\epsilon_{\alpha}}{B}$ such that $\epsilon_{\alpha}(a)\in I_{\alpha}(a)$. Therefore we have that
\[ \alpha (a,\varepsilon_{\alpha}(a))=\bigvee_{b\in B} \alpha (a,b)= \Einv_{\pr_A}(\alpha)(a)\]
and this prove that $\mH^{(-)}$ satisfies the epsilon rule.
\end{example}
We conclude this section recalling from \cite{AF,EQC,QCFF} the notion of doctrine with (strong) comprehensions. This notion connects an abstract elementary doctrine with that of the subobjects of the base when this has finite limits and provides an abstract algebraic counterpart of the set-theoretic ``comprehension axiom".
\begin{definition}
Let $\doctrine{\mC}{P}$ be a primary doctrine and  $\alpha$ be an object of $P(A)$.  A \bemph{comprehension} of $\alpha$ is an arrow $\freccia{X}{{\comp{\alpha}}}{A}$ such that $P_{\comp{\alpha}}(\alpha)=\top_X$ and, for every $\freccia{Z}{f}{A}$ such that $P_f(\alpha)=\top_Z$, there exists a unique map $\freccia{Z}{g}{X}$ such that $f=\comp{\alpha} \circ g$.
\end{definition}
Intuitively, the domain of the  comprehension morphism of the predicate $\alpha$  represents the set  $\{ x\in A \ \mid \alpha (x)\ \}$  containing  the elements of the object $A$ 
satisfying $\alpha$.
Then, one says that $P$ \bemph{has comprehensions} if every $\alpha$ has a comprehension, and that $P$ \bemph{has full comprehensions} if, moreover, $\alpha\leq \beta$ in $P(A)$ whenever $\comp{\alpha}$ factors through $\comp{\beta}$.

Note that each morphism can be the full comprehension of a unique object:
\begin{lemma}\label{uniqueness}
Let us consider a doctrine
\[\doctrine{\mC}{P}\]
such that $P$ has full comprehensions.
If $\comp{\alpha_1}=\comp{\alpha_2}$ then $\alpha_1=\alpha_2$.
\end{lemma}
\begin{remark}\label{remark compr has pullbacks}
For every $\freccia{A'}{f}{A}$ in $\mC$ the mediating arrow between the comprehensions $\freccia{X}{\comp{\alpha}}{A}$ and $\freccia{X'}{\comp{P_f(\alpha)}}{A'}$ produces a pullback
\[\pullback{X'}{A'}{X}{A.}{\comp{P_f(\alpha)}}{f'}{f}{\comp{\alpha}}\]
Thus comprehensions are stable under pullbacks.
Moreover it is straightforward to verify that if $\freccia{X}{\comp{\alpha}}{A}$ is a comprehension of $\alpha$, then $\comp{\alpha}$ is monic.
\end{remark}
\textbf{Notation:} given an $\alpha\in P(A)$, we will denote by $A_{\alpha}$ the domain of the comprehension $\comp{\alpha}$.
\begin{lemma}\label{remark tops cover the doctrine with full comp}
 Let $\doctrine{\mC}{P}$ be an doctrine with full comprehensions, and suppose that $P$ has left adjoints along comprehensions, satisfying (BCC), i.e. for every pullback  
 \[\pullback{A'_{P_f(\alpha)}}{A'}{A_{\alpha}}{A.}{\comp{P_f(\alpha)}}{f'}{f}{\comp{\alpha}}\]
we have $P_f\exists_{\comp{\alpha}}=\exists_{\comp{P_f(\alpha)}}P_{f'}$. Then we have 
\[
\alpha =\Einv_{\comp{\alpha}}(\top_{A_{\alpha}})
\]
for every element $\alpha$ in $P(A)$.
\end{lemma}
\begin{proof}
Let $\alpha$ be an element of $P(A)$, and let us consider the comprehension $\freccia{A_{\alpha}}{{\comp{\alpha}}}{A}$. First, it is direct to see that $\Einv_{\comp{\alpha}}(\top_{A_{\alpha}})\leq \alpha $ since $\Einv_{\comp{\alpha}}$ is left adjoint to $P_{\comp{\alpha}}$ and $P_{\comp{\alpha} }(\top_A) =\top_{A_\alpha}$.
 Notice that since comprehensions are monomorphisms, and pullbacks along comprehensions always exist by Remark \ref{remark compr has pullbacks}, we have that $P_{\comp{\alpha}}\Einv_{\comp{\alpha}}=\id$ by BCC. In particular we have $P_{\comp{\alpha}}(\Einv_{\comp{\alpha}}(\top_{A_{\alpha}}))=\top_{A_{\alpha}}$. So, by fullness, $\alpha\leq \Einv_{\comp{\alpha}}(\top_{A_{\alpha}})$, and then we can conclude that $\alpha=\Einv_{\comp{\alpha}}(\top_{A_{\alpha}})$. \end{proof}

We introduce a class of doctrines whose fibre equality turns out to be equivalent to the morphism equality of their base morphisms (see proposition 2.2 in \cite{TECH} and the original
notion called ``comprehensive equalizers" in \cite{QCFF}).
\begin{definition}
An elementary doctrine $\doctrine{\mC}{P}$ has \bemph{comprehensive diagonals} if the arrow $\Delta_A$ is the comprehension of the element $\delta_A\in P(A)$ for every objects $A$.
\end{definition}

Recall from \cite{TECH} that:

\begin{definition}
 An elementary doctrine is called \bemph{m-variational} if it has full comprehensions and comprehensive diagonals.
 \end{definition}

We summarize some useful properties and results about existential m-variational doctrines. 

\begin{lemma}\label{lemma rel funzionale iff pr rho mono}
Let $\doctrine{\mC}{P}$ be an existential m-variational doctrine. Then 
\begin{enumerate}
\item an arrow $\freccia{A}{f}{B}$ is monic if and only if $P_{f\times f}(\delta_B)=\delta_A$;
\item an element $\phi\in P(A\times B)$ is functional, i.e. it satisfies
 \[ P_{\angbr{\pr_1}{\pr_2}}(\phi)\wedge P_{\angbr{\pr_1}{\pr_3}}(\phi)\leq P_{\angbr{\pr_2}{\pr_3}}(\delta_B)\]
in $P(A\times B\times B)$ if and only if $\pr_A\comp{\phi}$ is monic, where $\freccia{A\times B}{\pr_A}{A}$ is the first projection.
\end{enumerate}
\end{lemma}
\begin{proof}
\textit{1.} If $f$ is monic then $P_{f\times f}(\delta_B)=\delta_A$ follows by \cite[Cor. 4.8]{EQC}, while the other direction follows from \cite[Rem. 2.14]{TECH}.\\

\noindent
\textit{2.} Suppose that $\phi\in P(A\times B)$ is functional, and let us consider its comprehension $\freccia{E}{\comp{\phi}}{A\times B}$, which is an arrow of the form $\comp{\phi}:=\angbr{f_1}{f_2}$. Then, we have that proving that $\pr_A\comp{\phi}$ is monic is equivalent to prove that $P_{f_1\times f_1}(\delta_A)=\delta_E$ by (1). Recall that every comprehension is monic, in particular we have that 
\begin{equation}\label{eq delta_E=P_fxf delta_AxB}
\delta_E=P_{\comp{\phi}\times \comp{\phi}}(\delta_{A\times B}).
\end{equation}
Hence, since by definition we have that $\delta_{A\times B}=P_{\angbr{\pr_1}{\pr_3}}(\delta_A)\wedge P_{\angbr{\pr_2}{\pr_4}}(\delta_B)$, we have that \eqref{eq delta_E=P_fxf delta_AxB} implies
\begin{equation}\label{eq delta_E=P_fxf delta_A e P_fxf delta_B}
\delta_E=P_{f_1\times f_1}(\delta_A)\wedge P_{f_2\times f_2}(\delta_B)
\end{equation}
Now, since $\phi$ is functional we have that
\[ P_{\angbr{\pr_1}{\pr_2}}(\phi)\wedge P_{\angbr{\pr_1}{\pr_3}}(\phi)\leq P_{\angbr{\pr_2}{\pr_3}}(\delta_B)\]
in $P(A\times B\times B)$, and also
\[ P_{\angbr{\pr_1}{\pr_2}}(\phi)\wedge P_{\angbr{\pr_3}{\pr_4}}(\phi)\wedge P_{\angbr{\pr_1}{\pr_3}}(\delta_A)\leq P_{\angbr{\pr_2}{\pr_4}}(\delta_B)\wedge P_{\angbr{\pr_1}{\pr_3}}(\delta_A)\]
in $P(A\times B\times A\times B)$. 
Finally
\[P_{\comp{\phi}\times \comp{\phi}}(P_{\angbr{\pr_1}{\pr_2}}(\phi)\wedge P_{\angbr{\pr_3}{\pr_4}}(\phi)\wedge P_{\angbr{\pr_1}{\pr_3}}(\delta_A))\leq P_{\comp{\phi}\times \comp{\phi}}(P_{\angbr{\pr_2}{\pr_4}}(\delta_B)\wedge P_{\angbr{\pr_1}{\pr_3}}(\delta_A))=\delta_E.\]
Notice that the left side of the inequality is exactly 
\begin{equation}\label{equazione tecnica P_f_1xf_1}
P_{\comp{\phi}\times \comp{\phi}}(P_{\angbr{\pr_1}{\pr_2}}(\phi)\wedge P_{\angbr{\pr_3}{\pr_4}}(\phi)\wedge P_{\angbr{\pr_1}{\pr_3}}(\delta_A))=P_{f_1\times f_1}(\delta_{A})
\end{equation} 
because $P_{\comp{\phi}\times \comp{\phi}}(P_{\angbr{\pr_1}{\pr_2}}(\phi))=P_{\pr_E}P_{\comp{\phi}}(\phi)=\top$ and similarly  $P_{\comp{\phi}\times \comp{\phi}}(P_{\angbr{\pr_3}{\pr_4}}(\phi))=\top$.
 Hence we have
\[P_{f_1\times f_1}(\delta_{A})\leq \delta_E\]
and then we obtain $P_{f_1\times f_1}(\delta_{A})= \delta_E$
and by (1) we conclude $\pr_A\comp{\phi}$ is monic.
\\

 One can prove the other direction using similar arguments. We just sketch the idea. Observe that if $\pr_1\comp{\phi}$ is monic, then we have that $$P_{f_1 \times f_1}(\delta_A)\leq P_{f_2\times f_2}(\delta_B).$$ This follows from the fact that comprehensions are monic, i.e. $\delta_E=P_{\comp{\phi}\times \comp{\phi}}(\delta_{A\times B})$, from (2)  and that $f_1$ is monic by our assumption, i.e  $\delta_E=P_{f_1\times f_1}(\delta_A)$.

Therefore we have that $\exists_{f_2\times f_2}P_{f_1 \times f_1}(\delta_A)\leq \delta_B$. Now we employ the same argument as before: notice that the equality \eqref{equazione tecnica P_f_1xf_1} still holds, so we have $P_{f_1 \times f_1}(\delta_A)=P_{\comp{\phi}\times \comp{\phi}}(P_{\angbr{\pr_1}{\pr_2}}(\phi)\wedge P_{\angbr{\pr_3}{\pr_4}}(\phi)\wedge P_{\angbr{\pr_1}{\pr_3}}(\delta_A))$. Hence we have that
$$\exists_{f_2\times f_2}P_{f_1 \times f_1}(\delta_A)=\exists_{\angbr{\pr_2}{\pr_4}}\exists_{\comp{\phi}\times \comp{\phi}}P_{\comp{\phi}\times \comp{\phi}}(P_{\angbr{\pr_1}{\pr_2}}(\phi)\wedge P_{\angbr{\pr_3}{\pr_4}}(\phi)\wedge P_{\angbr{\pr_1}{\pr_3}}(\delta_A))$$
and then we can conclude that 
$$\exists_{\angbr{\pr_2}{\pr_4}}(P_{\angbr{\pr_1}{\pr_2}}(\phi)\wedge P_{\angbr{\pr_3}{\pr_4}}(\phi)\wedge P_{\angbr{\pr_1}{\pr_3}}(\delta_A))\leq \delta_B$$
and hence
$$P_{\angbr{\pr_1}{\pr_2}}(\phi)\wedge P_{\angbr{\pr_3}{\pr_4}}(\phi)\wedge P_{\angbr{\pr_1}{\pr_3}}(\delta_A)\leq P_{\angbr{\pr_2}{\pr_4}}(\delta_B).$$
In particular, applying to both left and right side of the previous inequality the functor $P_{\angbr{\pr_1,\pr_2}{\pr_1,\pr_3}}$, where we consider the projections from $A\times B\times B$, we can can conclude that 
$$P_{\angbr{\pr_1}{\pr_2}}(\phi)\wedge P_{\angbr{\pr_1}{\pr_3}}(\phi)\leq P_{\angbr{\pr_2}{\pr_3}}(\delta_B)$$
since $P_{\angbr{\pr_1,\pr_2}{\pr_1,\pr_3}}(P_{\angbr{\pr_1}{\pr_3}}(\delta_A))=P_{\angbr{\pr_1}{\pr_1}}(\delta_A)=\top$.
\end{proof}

As pointed out in \cite{TECH,EQC,CLTT}, both full comprehensions and comprehensive diagonals can be freely added to a given doctrine.
In particular, recall that the Grothendieck category $\mG_P$ of points of the doctrine $\doctrine{\mC}{P}$ provides the free addition of comprehensions. This construction is called \emph{comprehension completion}. 
\begin{definition}\label{def comprehension comp}
Given a primary doctrine $\doctrine{\mC}{P}$ we can define the \bemph{comprehension completion} $\doctrine{\mG_P}{P_c}$ of $P$ as follows:
\begin{itemize}
\item an object of $\mG_P$ is a pair $(A,\alpha)$ where $A$ is a set and $\alpha\in P(A)$;
\item an arrow $\freccia{(A,\alpha)}{f}{(B,\beta)}$ if a arrow $\freccia{A}{f}{B}$ such that $\alpha\leq P_f(\beta)$.
\end{itemize}
The fibres $P_c(A,\alpha)$ are given by those elements $\gamma$ of $P(A)$ such that $\gamma\leq \alpha$. Moreover, the action of $P_c$ on morphism $f: (B,\beta)\ \rightarrow \ (A,\alpha)$ is defined as  $P_c(f)(\gamma)=P(f) (\gamma)\, \wedge\, \beta $
for $\gamma\in P(A)$  such that $\gamma\leq \alpha$.
\end{definition}

Similarly, the construction which freely adds comprehensive diagonal is provided by the \emph{extensional reflection}.

\begin{definition}\label{def extentiona reflection}
Given an elementary doctrine $\doctrine{\mC}{P}$ we can define \bemph{extensional reflection} $\doctrine{\mX_P}{P_x}$ of $P$ as follows: the base category $\mX_P$ is the quotient category of $\mC$ with respect to the equivalence relation where $f\sim g$ when $\top_A\leq P_{\angbr{f}{g}}(\delta_B)$. The equivalence class of a morphism $f$ of $\mC$, i.e. an arrow of $\mX_P$, is denoted by $[f]$.
\end{definition}
Finally, we denoted by $\Pred{P}$ the \bemph{category of predicates} of a doctrine $P$, i.e. the category defined as $\Pred{P}:=\mX_{P_c}$. 

Again we refer to \cite{TECH} for a complete description of these constructions.
\section{Generalized existential completion}\label{section existential completion}
We recall here from \cite{ECRT} how to complete a conjunctive doctrine to a generalized existential doctrine  with respect to a class $\Lambda$ of morphisms of $\mC$  closed under  composition, pullbacks and containing the identities.

\begin{definition}\label{lclass}
A class of morphisms $\Lambda$ of a category $\mC$ is called a \textbf{\lclass} if it satisfies the following conditions:
\begin{itemize}
\item given an arrow $fh$ of $\mC$, if $f \in \Lambda$ and $h\in \Lambda$, then we have $fh\in \Lambda$;
\item for every arrow $f\in \Lambda$ and $g$ of $\mC$, the pullback of $f$ and $g$ 
\[\xymatrix@+1pc{
D\ar[d]_{\pbmorph{f}{g}} \pullbackcorner \ar[r]^{\pbmorph{g}{f}} &A\ar[d]^g\\
C \ar[r]_f & B
}\]
exists and $\pbmorph{g}{f}\in \Lambda$;
\item $\id_A\in \Lambda$ for every object $A$ of $\mC$.
\end{itemize}
\end{definition}

\begin{example}
For any category with finite products the class of product projections is an example of a \lclass.
\end{example}

Actually,   the class $\Lambda$ represents {\it generalized projections} with respect to which we complete a conjunctive doctrine to a generalized existential ones.

\begin{definition} 
A \bemph{left class doctrine} is a pair $(P,\Lambda)$, where $\doctrine{\mC}{P}$ is a \conjdoctrine~and $\Lambda$ is a \lclass.
\end{definition}

\noindent
Now, we generalize the notion of  ``existential doctrine" to a doctrine closed under  left adjoint to  functors $P_f$ for any morphism $f$  in the left class $\Lambda$:
\begin{definition}\label{def lambda existential doctrine}
A \Lamdoctrine~$(P,\Lambda)$ is called a \bemph{generalized existential doctrine} with respect to a \lclass~$\Lambda$ if, for any arrow $\freccia{A}{f}{B}$ of $\Lambda$, the functor
\[ \freccia{P(B)}{{P_{f}}}{P(A)}\]
has a left adjoint $\Einv_{f}$, and these satisfy:
\begin{itemize}
\item[(BCC)] \bemph{Beck-Chevalley condition}: for any pullback
\[\xymatrix@+1pc{
X'\ar[d]_{f'} \pullbackcorner \ar[r]^{g'}& A'\ar[d]^{f}\\
X\ar[r]_{g} &A
}\]
with $g\in \Lambda$ (hence also $g'\in \Lambda$), for any $\beta\in P(X)$ the following equality holds
\[\Einv_{g'}P_{f'}(\beta)=P_f\Einv_{g}(\beta).\]
\item[(FR)] \bemph{Frobenius reciprocity}: for every morphism $\freccia{X}{f}{A}$ of $\Lambda$, for every element $\alpha\in P(A)$ and $\beta\in P(X)$, the following equality holds
\[\Einv_f(P_f(\alpha)\wedge \beta)=\alpha\wedge \Einv_f(\beta).\]
\end{itemize}
\end{definition}
In order to simplify the notation, sometimes we will simply say that $P$ is a \bemph{$\Lambda$-existential doctrine} in \cite{ECRT} to indicate a \Lamdoctrine~$(P,\Lambda)$ that is a generalized existential doctrine.

Moreover,  we adopt  the following specific names  for $\Lambda$-existential doctrines when  $\Lambda$ is the class of all the base morphisms
or the class of projections:
\begin{definition}
Let $\mC$ be a finite limit category. 
A \bemph{full existential doctrine} is a  generalized existential doctrine   $\doctrine{\mC}{P}$  with respect to the class of all  the base morphisms.
\end{definition}

\begin{definition}
Let $\mC$  be a finite product category. 
A  \bemph{pure existential doctrine} is a generalized existential doctrine   $\doctrine{\mC}{P}$  with respect to the class of all the projections of its base category.
\end{definition}
 
\noindent
Now, we define the 2-category $\LamPR$ as follows:
\begin{itemize}
\item \textbf{objects} are \Lamdoctrines~$(P,\Lambda)$;
\item \textbf{1-cells} are pairs $\freccia{(P,\Lambda)}{(F,b)}{(R,\Lambda')}$ where $\freccia{\mC}{F}{\mD}$ is a functor such that for every $f\in \Lambda$, we have $F(f)\in \Lambda'$, $F$ preserves pullbacks  along morphisms of $\Lambda$ and $\freccia{P}{b}{R\circ F^{\op}}$ is a natural transformation.
\item a \bemph{2-cell} is a natural transformation $\freccia{F}{\theta}{G}$  from $(F,b)$ to $(G,c)$ such that for every $A$ in $\mC$ and every $\alpha$ in $P(A)$, we have 
\[b_A(\alpha)\leq R_{\theta_A}(c_A(\alpha)).\]
 
\end{itemize}
Similarly, we denote by $\LamEX$ the 2-full subcategory of $\LamPR$ whose objects are objects $\Lambda$-existential doctrine $(P,\Lambda)$, and a 1-cell of $\LamEX$ is a 1-cell $(F,b)$ of $\LamPR$ such that the natural transformation $b$ preserves left adjoints of the opportune left classes, i.e. given $(P,\Lambda)$ and $(R,\Lambda')$, for every $\freccia{A}{f}{B}$ arrow of $\Lambda$, we have $b_B\exists_f=\exists_{F(f)}b_A$. As in the case of the ordinary existential doctrines, the 2-cell remains the same.

Given this setting, we present the  construction of the  \bemph{generalized existential completion} introduced in \cite{ECRT}  and  consisting of a $\Lambda$-existential doctrine $\doctrine{\mC}{\gencompex{P}{\Lambda}}$ freely generated  from a \Lamdoctrine~$(P,\Lambda)$.
Such a  free construction provides a left 2-adjoint to the forgetful functor $\freccia{\LamEX}{U}{\LamPR}$. 

\medskip

\begin{definition}[\bf Generalized existential completion]
For every object $A$ of $\mC$ consider the following preorder:
\begin{itemize}
\item the objects are pairs $(\frecciasopra{B}{g\in \Lambda}{A},\; \alpha\in PB)$;
\item $(\frecciasopra{B}{h\in \Lambda}{A},\; \alpha\in PB)\leq (\frecciasopra{D}{f\in \Lambda}{A},\; \gamma\in PD)$ if there exists $\freccia{B}{w}{D}$ such that
\[\xymatrix@+1pc{
 B \ar[r]^{w} \ar[dr]_h& D\ar[d]^f\\
&A
}\]

commutes and $\alpha\leq P_w(\gamma)$.
\end{itemize}
It is easy to see that the previous data give a preorder. We denote by
$\gencompex{P}{\Lambda}(A)$ the partial order obtained by identifying two
objects when
$$(\frecciasopra{B}{h\in \Lambda}{A},\; \alpha\in PB)\gtreqless (\frecciasopra{D}{f\in \Lambda}{A},\; \gamma\in PD)$$
in the usual way. With abuse of notation we denote the equivalence
class of an element in the same way.

Given a morphism $\freccia{A}{f}{B}$ in $\mC$, let $\gencompex{P}{\Lambda}_f(\frecciasopra{C}{g\in \Lambda}{B},\; \beta\in PC)$ be the object 
\[(\frecciasopra{D}{\pbmorph{f}{g}}{A},\; P_{\pbmorph{g}{f}}(\beta)\in PD)\]
where  
\[\xymatrix@+1pc{
D\ar[d]_{\pbmorph{f}{g}} \pullbackcorner \ar[r]^{\pbmorph{g}{f}} &C\ar[d]^g\\
A \ar[r]_f & B
}\]
is a pullback because $g\in \Lambda$. 

The assignment $\doctrine{\mC}{\gencompex{P}{\Lambda}}$ is called \bemph{the generalized existential completion} of $P$.
\end{definition}

\begin{theorem}\label{theorem generalized existential doctrine}
For every  \Lamdoctrine~$(P,\Lambda)$ the doctrine $(\gencompex{P}{\Lambda},\Lambda)$ is a generalized existential doctrine with respect to the left class $\Lambda$.
\end{theorem}
\begin{proof} See \cite[Thm 4.3]{ECRT}.

\end{proof}

Again, we fix the notation for the  specific cases in which the \lclass~of the base category is the class of projections or of all the morphisms:
\begin{definition}
Let $\doctrine{\mC}{P}$ be a \conjdoctrine~on a finite limit (lex)  category $\mC$.
A  \bemph{full existential completion} of $P$, denoted with the symbol $\fullcompex{P}$,
is the generalized existential completion of $P$  with respect to 
 the class of all the base morphisms.
 \end{definition}
   
\begin{definition}
Let $\doctrine{\mC}{P}$ be a \conjdoctrine~on a category $\mC$ with finite products.
A \bemph{pure existential completion}, denoted with the symbol $\purecompex{P}$,
is the generalized existential completion  of $P$ with respect to 
 the class of all the projections in its base category.
 \end{definition}


Observe that we  can define a canonical injection of $\Lambda$-doctrines
 $$\freccia{(P,\Lambda)}{(\id_{\mC},\eta_P)}{(\gencompex{P}{\Lambda},\Lambda)}$$ 
where $\freccia{P(A)}{(\eta_P)_A}{\gencompex{P}{\Lambda}(A)}$ acts sending $$\alpha\mapsto(\frecciasopra{A}{\id_A}{A},\alpha\in P(A)).$$
Similarly, if $P$ is $\Lambda$-existential, we can define the 1-cell of $\Lambda$-existential doctrines $$\freccia{(\gencompex{P}{\Lambda},\Lambda)}{(\id_{\mC},\varepsilon_P)}{(P,\Lambda)}$$
which preserves the left-adjoints along morphisms of $\Lambda$,
where $\freccia{\compex{P}(A)}{(\varepsilon_P)_A}{P(A)}$ acts sending $$(\frecciasopra{B}{f\in \Lambda}{A},\beta\in P(B))\mapsto \Einv_{f}(\alpha).$$ 

It is direct to check that, if $(P,\Lambda)$ is $\Lambda$-existential, then $\varepsilon_P\eta_P=\id_P$ and that $\id_{\gencompex{P}{\Lambda}}\leq \eta_P \varepsilon_P$.

\begin{remark}\label{remark prenex normal form general}
Observe that every element $(\frecciasopra{B}{f}{A},\beta)$ of $\gencompex{P}{\Lambda}(A)$ is equal to $\Einv^{\Lambda}_f\eta_B(\beta)$ because $\Einv^{\Lambda}_f$ acts as the post-composition.
\end{remark} 
\begin{definition}[{\bf $\Lambda$-weak subobjects doctrine}]\label{definition lambda-weak sub}
Let $\mC$ be a category and let $\Lambda$ be a \lclass~of $\mC$. We define the  $\Lambda$-\bemph{weak subobjects} doctrine, or the doctrine of \bemph{$\Lambda$-weak subobjects}, as  the functor $$\doctrine{\mC}{\Psi_{\Lambda}}$$ assigning to every object $A$ of $\mC$ the poset $$\Psi_{\Lambda}(A):=\{\frecciasopra{B}{f}{A} \;|\; f\in \Lambda\}$$
with the usual order given by the factorization, i.e. the partial order induced by the usual preorder  on morphisms given by $f\leq g$ if there exists an arrow $h$ such that $f=gh$. The re-indexing functors $\Psi_{\Lambda}(f)$ act by pulling back the elements of the fibres. 
\end{definition}
One can directly check that $\Psi_{\Lambda}$ is a $\Lambda$-existential doctrine, since the left adjoints are given by the post-composition of arrows of $\Lambda$. 

\begin{definition}\label{definition trivial doc}
Given a category $\mC$, we call the \bemph{trivial doctrine} on $\mC$ the functor
\[\doctrine{\mC}{\trdc}\]
assigning  the poset with only an element to every object $A$ of $\mC$ denoted as $\trdc(A)=\{\top\}$.
\end{definition}
By definition of generalized existential completion, we immediately obtain the following theorem.
\begin{theorem}\label{example lambda-weak-sub}
The $\Lambda$-existential doctrine $\doctrine{\mC}{\Psi_{\Lambda}}$  is the generalized existential completion of the trivial doctrine $\doctrine{\mC}{\trdc}$.
\end{theorem}
\begin{proof}
Observe that the top elements \emph{cover} the weak subobjects doctrine, i.e. for every object $A$ and for every element $\freccia{B}{f}{A}$ in the fibre $\Psi_{\mD}(A)$ we have $f=\Einv_f([\top_B])$, since the left adjoints $\Einv_f$ are given by the post-composition and the top element $\top_B$ is the identity morphism $\id_B$.
\end{proof}

\begin{example}\label{theorem weak sub is a existential comp}
Observe that the weak subobjects doctrine defined in example~\ref{example wek-sub doctrine}   is a  $\Lambda$-\bemph{weak subobjects} doctrine where $\Lambda$ is the class $\Lambda_{\mD}$ of all the morphisms of a finite limit base category $\mD$. Hence from Theorem \ref{example lambda-weak-sub} the weak subobjects doctrine  is a  \emph{full existential completion} of the
trivial doctrine $\doctrine{\mD}{\trdc}$.
\end{example}

Notice that the universal properties  of existential completion shown in \cite{ECRT}, can be generalized for the arbitrary case of the generalized existential completion.

%
 
In particular, the proof of \cite[Thm. 4.14]{ECRT} can be adapted to this more general setting, just observing that in the proof the fact the class $\Lambda$ is the class of projections, is not concretely used. The proof depends only on the properties of closure under pullbacks, composition and identities of the class of projections, and hence it can be directly generalized as follow.
\begin{theorem}
The forgetful 2-functor $\freccia{\LamEX}{U}{\LamPR}$ has a left 2-adjoint 2-functor $\freccia{\LamPR}{E}{\LamEX}$, acting on the objects as $(P,\Lambda)\mapsto (\gencompex{P}{\Lambda},\Lambda)$.
\end{theorem}

In the following we are going to define those elements of a $\Lambda$-existential doctrine $P$ which are  \emph{free from the left adjoints} $\Einv_f$ along $\Lambda$
as an algebraic counterpart of the logical notion of {\it existential-free formula}.

To this purpose we first define:
\begin{definition}\label{definition (P,lambda)-atomic0}
Let $\doctrine{\mC}{P}$ be an $\Lambda$-existential doctrine. An object $\alpha$ of the fibre $P(B)$ is said an  $\Lambda$-\bemph{existential splitting} if 
for every morphism 
\[\xymatrix{
C \ar[r]^{g\in \Lambda} &B 
}\]
and for every element $\beta$ of the fibre $P(C)$, whenever $ \alpha \leq \Einv_{g}(\beta)$ holds
then there exists an arrow $\freccia{B}{h}{C}$ such that
\[    \alpha \leq P_{h}(\beta) \]
and $gh=\id$.
\end{definition}

Now we are ready to introduce the notion of $\Lambda$-\bemph{existential-free} object of a doctrine:
\begin{definition}\label{definition (P,lambda)-atomic}
Let $\doctrine{\mC}{P}$ be an $\Lambda$-existential doctrine. An object $\alpha$ of the fibre $P(A)$ is said a  $\Lambda$-\bemph{existential-free} if 
for every morphism 
\[\xymatrix{
B \ar[r]^f &A
}\]
$P_f(\alpha)$ is  a $\Lambda$-existential splitting.
\end{definition}
\begin{remark}
The name $\Lambda$-\bemph{existential-free} object  represents an algebraic version of the syntactic notion of  {\em existential-free formula}
because if a  $\Lambda$-\bemph{existential-free} object $\alpha$ is equal to the existential quantification of an object $\beta$ along an arrow $f$ in $\Lambda$,  then $\alpha$ is equal to  a reindexing of $\beta$.
Indeed, if
$\alpha=\exists_f(\beta)$ with  $\beta\in P(B)$ and $\freccia{B}{f}{A}$, then  from $\alpha\leq \exists_f(\beta)$  it follows that there exists an arrow $\freccia{B}{h}{A}$ such that $\alpha\leq P_h(\beta)$.
Moreover,  since $\exists_f(\beta)\leq \alpha$, then $\beta\leq P_f(\alpha)$  and also $P_h(\beta)\leq P_h(P_f(\alpha))=\alpha$ because $fh=\id_A$. Hence we conclude that  $\alpha=P_h(\beta)$. 
\end{remark}

%
%
Now we are going to observe how the notion of  \qflamb~object is related to well known choice principles.
To this purpose we first introduce a  generalization of  the notion of Rule of Choice by relativizing the existence of a witness to the arrows of the class $\Lambda$.
\begin{definition}
For a $\Lambda$-existential doctrine $\doctrine{\mC}{P}$, we say that $P$ satisfies the $\bemph{Generalized Rule of Choice}$  with respect to the left class $\Lambda$ of morphisms, also
called $\Lambda$-(RC),  if whenever
\[\top_A\leq \Einv_g (\alpha) \]
where $\freccia{B}{g}{A}$ is an arrow of $\Lambda$, then there exists an arrow $\freccia{A}{f}{B}$ such that
\[\top_A\leq P_{f}(\alpha)\]
and $gf=\id$.
\end{definition}

\begin{remark}\label{RC}
A $\Lambda$-existential doctrine $\doctrine{\mC}{P}$ has $\Lambda$-(RC) if and only if for every object $A$ of $\mC$, the top element $\top_A\in P(A)$ is \qflamb~object.
\end{remark}

\begin{definition}
Given a $\Lambda$-existential doctrine $\doctrine{\mC}{P}$, we also say that an element $\alpha $ of the fibre $ P(A)$ is \bemph{$\Lambda$-covered} by an element $\beta\in P(B)$ if $\beta$ is a \qflamb~object and there exists an arrow $\freccia{B}{f}{A}$ of $\Lambda$ such that $\alpha=\Einv_f(\beta)$.
\end{definition}

\begin{definition}
We say that a $\Lambda$-existential doctrine $P$ has \bemph{enough}-$\Lambda$-\bemph{existential-free objects} if for every object $A$ of $\mC$ and for every element $\alpha\in P(A)$ there exists an object $B$, a \qflamb~element $\beta$ and an arrow $\freccia{B}{g}{A}$ in $\Lambda$ such that
\[\alpha=\Einv_{g}(\beta). \]

\end{definition}


Notice that similar notions were introduced independently in  \cite{Frey2020} (see also the talk of the second author \url{https://www.youtube.com/watch?v=sAMVU_5RQJQ&t=688s}).

Here we fix  some notation for the specific cases in which the \lclass~of the base category is the class of projections or of all the morphisms. In particular:
\begin{itemize}
\item when the case $\Lambda$ is the class of all the morphisms we will speak of \bemph{pure-existential splitting}, \bemph{pure-existential-free} and \bemph{enough-pure-existential-free objects};

\item when $\Lambda$ is the class of all the morphisms of the base category we will speak of  \bemph{full-existential splitting}, \bemph{full-existential-free} and \bemph{enough-full-existential-free objects}.

\end{itemize}

Before proving the main result of this section, we present some useful technical lemmas.
\begin{lemma}\label{lemma E_f (a)<E_g(b)}
Let $\doctrine{\mC}{P}$ be a $\Lambda$-existential doctrine, and let us consider a $\Lambda$-existential-free element $\alpha\in P(B)$. If $\Einv_f(\alpha)\leq\Einv_g(\beta)$ where $\freccia{B}{f}{A}$ and $\freccia{C}{g}{A}$ are arrows of the base category $\mC$ and $\beta\in P(C)$, then there exists an arrow $\freccia{B}{m}{D}$ where $D$ is the vertex of the pullback  of $f$ along $g$ such that

\begin{itemize}
\item $(\pbmorph{f}{g})m=\id$;
\item  $\alpha\leq P_{(\pbmorph{g}{f})m}(\beta)$.
\end{itemize}
And hence also $f=g(\pbmorph{g}{f})m$.
\end{lemma}
\begin{proof}
If $\Einv_f(\alpha)\leq\Einv_g(\beta)$, then $\alpha\leq P_f\Einv_g(\beta)$, and applying BCC, we have $\alpha\leq \Einv_{\pbmorph{f}{g}}P_{\pbmorph{g}{f}}(\beta)$, where
\[\pullback{D}{B}{C}{A}{\pbmorph{f}{g}}{\pbmorph{g}{f}}{f}{g}\]
is a pullback. Thus, since $\alpha$ is  $\Lambda_{\mC}$-existential-free, there exists a morphism $\freccia{B}{m}{D}$ such that $\alpha\leq P_{(\pbmorph{g}{f})m}(\beta)$ and $(\pbmorph{f}{g})m=\id$. In particular $f=g(\pbmorph{g}{f})m$. 
\end{proof}

\begin{lemma}\label{split-ex-free}
Let $\doctrine{\mC}{P}$ be a $\Lambda$-doctrine. 
If every element of the form $\eta_A(\alpha)$ for any object $\alpha$ of $P(A)$ with $A$ object of $\cal C$ is $\Lambda$-existential splitting then every element of the form $\eta_A(\alpha)$ is a \qflamb~object. 
\end{lemma}
\begin{proof}
It follows by the naturality of $\eta_A$.
\end{proof}

%

%
\begin{proposition}\label{proposition doctrine P' cover P}
Let $\doctrine{\mC}{P}$ be a $\Lambda$-existential doctrine  equipped with a fibred subdoctrine  $\doctrine{\mC}{P'}$  such that every element of $P'$ is a $\Lambda$-existential-free element in $P$. Then, if every element of $\alpha$ of $P(A)$ is  $\Lambda$-covered by an element of $P'$,  the elements of $P'$ are exactly those elements of $P$ which are $\Lambda$-existential-free.
\end{proposition}
\begin{proof}
Let $\alpha\in P(A)$ be a $\Lambda$-existential-free object. We have to prove that $\alpha\in P'(A)$. Now, since every element of $P$ is covered by an element of $P'$, we have that $\alpha=\exists_f(\beta)$ where $\beta$ is $\Lambda$-existential-free, $\beta\in P'(B)$ and $\freccia{B}{f}{A}$ is in $\Lambda$. Then,   since
$\alpha\in P(A)$ is a $\Lambda$-existential-free object,  from $\alpha\leq \exists_f(\beta)$  it follows that there exists an arrow $\freccia{B}{h}{A}$ such that $\alpha\leq P_h(\beta)$ with $fh=\id_A$. Moreover, since $\exists_f(\beta)\leq \alpha$, then $\beta\leq P_f(\alpha)$  and also $P_h(\beta)\leq P_h(P_f(\alpha))=\alpha$ because $fh=\id_A$. Hence we conclude that  $\alpha=P_h(\beta)$  which is then an object of  $ P'(A)$ because $\beta$ and
its reindexings are in  $P'$.
\end{proof}

\begin{proposition}\label{exist-free-objects}
Let $\doctrine{\mC}{P}$ be a $\Lambda$-doctrine.
Its generalized existential completion  $\gencompex{P}{\Lambda}$ satisfies the following conditions:
\begin{enumerate}
\item 
the elements  $\eta_A(\alpha)$  for any object $\alpha$ of $P(A)$ with $A$ object of $\cal C$ are \qflamb~objects of  $\gencompex{P}{\Lambda}$;
\item
 $\gencompex{P}{\Lambda}$ has \eqflamb\  of the form $\eta_A(\alpha)$ for any object $\alpha$ of $P(A)$ with $A$ object  in $\cal C$.
\item
the \qflamb~objects of  $\gencompex{P}{\Lambda}$ are exactly the elements  $\eta_A(\alpha)$  for any object $\alpha$ of $P(A)$ with $A$  object in $\cal C$.
 \end{enumerate}
\end{proposition}
\begin{proof}
(1) 
We first show that every element of the form $\eta_A(\alpha)$ is a $\Lambda$-existential splitting. Thus, let $\ovln{\beta}:=(\frecciasopra{C}{f}{B},\beta)$ be an object of the fibre $\gencompex{P}{\Lambda}(B)$, and suppose that 
\begin{equation}\label{eq gen top leq Einv_pr phi}
\eta_A(\alpha)\leq \Einv_g^{\Lambda}(\ovln{\beta})
\end{equation}
where $\freccia{B}{g}{A}$.
Recall that, by definition of the doctrine $\gencompex{P}{\Lambda}$, the inequality \eqref{eq gen top leq Einv_pr phi} means that there exists an arrow $\freccia{A}{h}{C}$ such that the following diagram commutes
\[\xymatrix@+2pc{
A\ar[r]^{h}\ar[dr]_{\id_A}& C\ar[d]^{\;\;\;gf} \\
 & A
}\]
and 
\begin{equation}\label{eq 1 alpha leq Pex(beta)}
\alpha\leq P_{h}(\beta)
\end{equation}
We claim that 
 \begin{equation}\label{claim} \eta_A(\alpha)\leq \gencompex{P}{\Lambda}_{fh}(\ovln{\beta}).
 \end{equation}
Thus, let us consider the pullback 
\[
\pullback{D}{C}{A}{B.}{\pbmorph{f}{(fh)}}{\pbmorph{(fh)}{f}}{f}{fh}
\]
Moreover,we have by definition 
\[\gencompex{P}{\Lambda}_{fh}(\ovln{\beta})=(\frecciasopra{D}{\pbmorph{(fh)}{f}}{A},P_{\pbmorph{f}{(fh)}}(\beta)).\] 
Thus, by the universal property of pullbacks, there exists an arrow $\freccia{A}{w}{D}$ such that the following diagram commutes 

\[\xymatrix@+2pc{ A \ar@/_1.5pc/[ddr]_{\id_A}^{(\bullet)} \ar@{-->}[dr]^w\ar@/^1.5pc/[drr]^{h}_{(\bullet\bullet)} \\
&D \pullbackcorner\ar[r]^{\pbmorph{f}{(fh)}} \ar[d]_{\pbmorph{(fh)}{f}} & C\ar[d]^{f}\\
&A \ar[r]_{fh} &B.
}\]
Hence, combining \eqref{eq 1 alpha leq Pex(beta)} with the triangle $(\bullet\bullet)$ we have that 
\[\alpha\leq P_{h}(\beta)=P_w(P_{\pbmorph{f}{(fh)}}(\beta)).\]
From this and the diagram $(\bullet)$ the claim  \eqref{claim} follows.
This ends the proof that an element of the form $\eta_A(\alpha)$  is $\Lambda$-existential splitting. By Lemma~\ref{split-ex-free} we conclude that  all such elements are \qflamb~objects.

\noindent
(2) It follows by Remark \ref{remark prenex normal form general}.

\noindent 
(3) It follows from (1) and (2)  and Proposition \ref{proposition doctrine P' cover P}.


\end{proof}

\begin{corollary}\label{corollary rule of choice}
Let $\doctrine{\mC}{P}$ be a $\Lambda$-doctrine. Then the generalized existential completion $\gencompex{P}{\Lambda}$ of the doctrine $P$ satisfies the Generalized Rule of Choice with respect to $\Lambda$.
\end{corollary}
\begin{proof}
It follows by Theorem \ref{exist-free-objects} since in each fibre of  $\doctrine{\mC}{P}$ the top element $\top_A$ is existential splitting for each object $A$ in $\mC$.
\end{proof}

\begin{example}
Recall  that if $\Lambda$ is the class of all the base morphisms of a \Lamdoctrine~$\doctrine{\mC}{P}$, then $P$ satisfies the Extended Rule of Choice of \ref{exrc}.
Hence, 
from the  above corollary we have a proof alternative to that in  \cite{TECH}  that 
the doctrine $\Psi_{\mD}$ of weak subobjects satisfies  the Extended Rule of Choice. 
\end{example}

\begin{remark}\label{RCfull}
Note that full and pure existential completions both satisfy the rule of choice (RC) as a consequence of remark~\ref{RC} since they are generalized existential completions
with respect to a  class $\Lambda$ containing all the projections of the base category.
\end{remark}

Now let us consider a  \Lamdoctrine~$(P, \Lambda)$ and  the forgetful functor $\freccia{\mG_P}{U}{\mC}$ from  the Grothendieck category $\mG_{P}$  of $P$.
We denote $U^{-1}(\Lambda)$  the class of morphisms in $\mG_{P}$ of the form $\freccia{(A,\alpha)}{f}{(B,\beta)}$ where $f\in \Lambda$.
\\
Notice that, since a \lclass~contains identities, we have that every arrow of the form $\freccia{(A,\alpha)}{\id_A}{(A,\top)}$ is contained in $\mG_P$.
\begin{lemma}
Let $(P,\Lambda)$ a \Lamdoctrine~with $\doctrine{\mC}{P}$.  Then $(P_c, U^{-1}(\Lambda))$ is a  $\Lambda$-doctrine, too.
\end{lemma}
\begin{proof}
Observe that for any morphism $\freccia{(A,\alpha)}{f}{(B,\beta)}$ with $f\in \Lambda$ and for any other morphism $\freccia{(C,\xi)}{h}{(B,\beta)}$ 
by using  the pullback projections $\freccia{D}{h^\ast f}{ C}$   and  $\freccia{D}{f^\ast h}{ A}$   of $f$ along $h$ claimed to exist in $\mC$
because $\Lambda$ is a \lclass, then 
the following  diagram is a pullback diagram in  $\mG_P$ 
\[\xymatrix@+2pc{
{(D, P_{h^\ast f} (\xi) \wedge  P_{f^\ast h} (\beta) )} \ar[d]_{\pbmorph{f}{h}} \pullbackcorner \ar[r]^{\;\;\;\;\;\;\;\pbmorph{h}{f}} &(C, \xi) \ar[d]^h\\
(A, \alpha)  \ar[r]_f & (B, \beta)
}\]
and $\freccia{(D, P_{h^\ast f} (\xi) \wedge  P_{f^\ast h} (\beta) ) }{h^\ast f}{(C,\xi)}$  is in $U^{-1}(\Lambda)$
since $h^\ast f$ is in $\Lambda$.
\end{proof}

Then, for every \Lamdoctrine~we can define a morphism of \Lamdoctrines~as follows:

\begin{definition}\label{arrowws}
 Let $\doctrine{\mC}{P}$ be  a \Lamdoctrine~and let us consider the $U^{-1}(\Lambda)$-weak subobjects on the Grothendieck category $\mG_{P}$, $\doctrine{\mG_P}{\Psi_{U^{-1}(\Lambda)}}$.
 
 We define the following \Lamdoctrine~morphism
\[\freccia{P}{(L,l)}{\Psi_{U^{-1}(\Lambda)} }\] 
   as the composition of the embedding of $P$ into its free comprehension completion $P_c$ followed by the comprehension doctrine morphism $ \comp{-}$ associating to an object its  comprehension arrow:
 
\[\xymatrix@+2pc{ P\ar@{^{(}->}[r] & P_c \ar[r]^{\comp{-}} &\Psi_{U^{-1}(\Lambda)}.} \]
namely
\begin{itemize}
\item $L(A)=(A,\top_A)$ for each object $A$ of  $\mC$
\item $L(f) =f$ for each morphism $f$ of $\mC$
\item $l(\alpha)=\comp{\alpha}$, i.e. $l(\alpha)$ is the \bemph{comprehension map} $l(\alpha):=\frecciasopra{(A,\alpha)}{\id_A}{(A,\top)}$ for each object $\alpha$ of $P(A)$.
\end{itemize}  
\end{definition}
Observe that:
\begin{lemma}
 For every object $A$ of the base category $\mC$,  and any \Lamdoctrine~$P$,  the functor $\freccia{P(A)}{l}{\Psi_{U^{-1}(\Lambda)}}(A,\top)$ is full and faithful, i.e. its preserves and reflects the order  and preserves finite conjunctions.
\end{lemma}
\begin{proof} It follows from \cite[Thm 2.15]{TECH}. \end{proof}
We now show a key lemma expressing that comprehensions in the image of $l$ are existential free objects:
\begin{lemma}\label{lemma comprehensions map are existential free}
Every element of the doctrine $\doctrine{\mG_P}{\Psi_{U^{-1}(\Lambda)}}$ of the form $\freccia{(A,\alpha)}{\id_A}{(A,\gamma)}$ is a $U^{-1}(\Lambda)$-existential-free element.  In particular, comprehension maps are $U^{-1}(\Lambda)$-existential-free elements.
\end{lemma}
\begin{proof}
It is direct to check that every element of the form $\frecciasopra{(A,\alpha)}{\id_A}{(A,\gamma)}$ is a $U^{-1}(\Lambda)$-existential splitting because left adjoints $\Einv_g$ in
$\Psi_{U^{-1}(\Lambda)}$
with $g\in U^{-1}(\Lambda)$ are given by the post-composition, and the order in the fibres of $\Psi_{U^{-1}(\Lambda)}$ is given by the factorization. Moreover, it is also direct to check that $\frecciasopra{(A,\alpha)}{\id_A}{(A,\gamma)}$ is a $U^{-1}(\Lambda)$-existential free element because for every arrow $\freccia{(B,\beta)}{h}{(A,\gamma)}$ of $\mG_P$ we have that, by definition, $\Psi_{U^{-1}(\Lambda)}(h)(\frecciasopra{(A,\alpha)}{\id_A}{(A,\gamma)})=\frecciasopra{(B,\beta \wedge  P_h(\alpha))}{\;\;\;\;\;\;\id_B}{(B,\beta)}$.

\end{proof}
Notice also that one can check that every element of the fibre $\Psi_{U^{-1}(\Lambda)}(A,\top)$ can be presented just using left adjoints along the class $U^{-1}(\Lambda)$ and  comprehension maps as follows:

\begin{lemma}\label{lemma every element is an exists of a comprensive maps}
For every element $\frecciasopra{(B,\beta)}{f}{(A,\top)}$ of the fibre $ \Psi_{U^{-1}(\Lambda)}(A,\top)$ we have that 
\[\frecciasopra{(B,\beta)}{f}{(A,\top)}=\exists_f(\frecciasopra{(B,\beta)}{\id_B}{(B,\top)})\]
with $\freccia{(B,\top)}{f}{(A,\top)}$ arrow of $U^{-1}(\Lambda).$
\end{lemma}
\begin{proof}
It directly follows by the fact that left adjoints act as the post composition.
\end{proof}

Finally, observe that full and faithful doctrine morphisms preserve and reflects existential-free objects as follows:
\begin{lemma}\label{lemma full and faithful morphism}
Let us consider a $\Lambda$-existential doctrine $P$, a $\Lambda'$-existential doctrine $R$ and a morphism of generalized existential doctrines $\freccia{P}{(F,f)}{R}$ such that both  $F$  and $f$ are full and faithful. Then we have that an element $\alpha\in P(A)$ is $\Lambda$-existential-free if and only if $f(\alpha)\in R(FA)$ is $F(\Lambda)$-existential-free.

\end{lemma}
\begin{proof} Just recall that $F(\Lambda)\subset \Lambda'$ by definition of morphisms between generalized existential doctrines.
\end{proof}

Now we are ready to prove the main result of this section.
\begin{theorem}\label{theorem caract. gen. existential completion}
Let $\doctrine{\mC}{P}$ be a $\Lambda$-existential doctrine. Then the following are equivalent:  
\begin{enumerate}

\item $P$ is a  $\Lambda$-existential completion.

\item There exists a fibred conjunctive sub-doctrine $\doctrine{\mC}{P'}$ of $P$ 
and a morphism $(\ovln{L},\ovln{l})$ of $\Lambda$-existential doctrines such that the diagram
\[\xymatrix@+2pc{P' \ar[rd]_{(L,l)}\ar[r]^{(\id,\iota)}& P \ar@{-->}[d]^{(\ovln{L},\ovln{l})} \\ 
&\Psi_{U^{-1}(\Lambda)} 
}\]
commutes,  where  $\doctrine{\mG_{P'}}{\Psi_{U^{-1}(\Lambda)}}$ is the $U^{-1}(\Lambda)$-weak  subobjects doctrine on the Grothendieck category $\mG_{P'}$ of $P'$. Moreover for every object $A$ in $\mC$ the natural transformation
 $\freccia{PA}{\ovln{l}_A}\Psi_{U^{-1}(\Lambda)}(A,\top)$ is  an isomorphism.

\item $P$ satisfies the following points:
\begin{enumerate}
\item $P$ satisfies $\Lambda$-(RC), i.e. each top element $\top_A$ of the fibre $P(A)$ is  a \qflamb~object.
\item for every \qflamb~object $\alpha$ and $\beta$ of $P(A)$, then $\alpha\wedge \beta$ is a \qflamb~object.
\item $P$ has \eqflamb.
\end{enumerate}

\end{enumerate}
\end{theorem}
\begin{proof}
$\mathrm{(1 \Rightarrow 2)}$ This follows from the universal property of the generalized existential completion. In particular, if $P$ is the $\Lambda$-existential completion of a doctrine $P'$, then the arrow $\freccia{P}{(\ovln{L},\ovln{l})}{\Psi_{U^{-1}(\Lambda)} }$ in definition~\ref{arrowws} exists by the universal properties of the generalized existential completion. Furthermore,  $\freccia{P(A)}{\ovln{l}}{\Psi_{U^{-1}(\Lambda)}(A,\top)}$ is surjective since as observed in Lemma~\ref{lemma every element is an exists of a comprensive maps}, every element $\freccia{(B,\beta)}{g}{(A,\top)}$ can be written as 
$$\exists_g(\comp{\beta})=\exists_g(l_B(\beta))=\exists_g(\ovln{l}_B(\iota_B(\beta)))=\ovln{l}_A\exists_g (\iota_B(\beta)).$$ 
It remains to  prove that $\ovln{l}_A$ reflects the order. Now suppose that $\ovln{l}_A(\Einv_f\iota_B(\beta ))\leq \ovln{l}_A(\Einv_g\iota_C(\xi)$. Then this means that $$\frecciasopra{(B,\iota_B(\beta))}{f}{(A,\top)}\leq \frecciasopra{(C,\iota_C(\xi))}{g}{(A,\top)}$$
in the fibre $\Psi_{\Lambda} (A,\top)$, i.e. that there exists an arrow $\freccia{(B,\iota_B(\beta))}{h}{(C,\iota_C(\xi))}$ such that $f=gh$. Moreover, we have that $\iota_B(\beta)\leq P_h(\iota_C(\xi))$ by definition of morphisms in  $\mG_{P'}$. Therefore we have that
\[\Einv_{f}\iota_B(\beta)\leq \Einv_f P_h(\iota_C(\gamma))=\Einv_g \Einv_hP_h(\iota_C(\gamma))\leq \Einv_g(\iota_C(\gamma)).\]
This ends the proof that for every object $A$, $\freccia{P(A)}{\ovln{l}_A}{\Psi_{U^{-1}(\Lambda)} (A,\top)}$ is an isomorphism of inf-semilattices.\\

\noindent
$\mathrm{(2\Rightarrow 3)}$ By Lemma \ref{lemma comprehensions map are existential free} and Lemma \ref{lemma full and faithful morphism} we have that every element of the form $\iota_A(\alpha)$ where $\alpha \in P'(A)$ is $\Lambda$-existential-free in $P(A)$, and by Lemma \ref{lemma every element is an exists of a comprensive maps} and the fact that $\ovln{l}$ is an isomorphism we  conclude that the doctrine $P$ has enough-$\Lambda$-existential-free elements. Finally, observe that the top element and the binary conjunction of $\Lambda$-existential free elements are $\Lambda$-existential-free elements because $P'$ is a full conjunctive sub-doctrine of $P$, i.e. $\iota$ preserves and reflects the order.
\\

\noindent
$\mathrm{(3\Rightarrow 1)}$ Let $(P,\Lambda)$ be a $\Lambda$-existential doctrine satisfying $ (a), (b)$ and $ (c). $ The we can define the \conjdoctrine
\[\doctrine{\mC}{\Pqf}\]
as the functor which sends an object $A$ to the poset $\Pqf (A)$ whose elements are the \qflamb~objects of $P(A)$ with the order induced from that of $P(A)$, and such that $\Pqf_f=P_f$ for every arrow $f$ of $\mC$. Notice that this functor is a \conjdoctrine~ because \qflamb~objects are stable under re-indexing by definition, and they are closed under the top element and binary conjunctions by the assumptions $(a)$ and $(c).$  Therefore, by the universal property of the generalized existential completion, there exists a 1-cell of $\LamEX$

\[\duemorfismo{\mC}{{\Pqf^{\Lambda}}}{id}{\unvpropmap}{\mC}{P}\]
where the map $\freccia{\gencompex{\Pqf}{\Lambda}(A)}{\unvpropmap_A}{P(A)}$  sends 
$$(\frecciasopra{B}{f}{A},\alpha)\mapsto \Einv_f (\alpha).$$
In particular, $\unvpropmap$ is a morphism of inf-semilattices and it is natural with respects to left adjoints along the class $\Lambda$. Moreover, notice that for every object $A$, we have that $\freccia{\gencompex{\Pqf}{\Lambda}(A)}{\unvpropmap_A}{P(A)}$ is surjective on the objects since $P$ has \eqflamb, i.e. every object $\alpha$ of $P(A)$ is of the form $\Einv_g(\beta)$ for some $\freccia{B}{g}{A}$ and $\beta\in P(B)$.
Now we want to show that $\unvpropmap$ reflects the order, and hence that it is an isomorphism. Suppose that 
\begin{equation}\label{eq Ea leq Eb}
\Einv_f (\alpha) \leq \Einv_g (\beta)
\end{equation}
where $\freccia{B}{f}{A}$ and $\freccia{C}{g}{A}$ are arrows of $\mC$, and $\alpha\in P(B)$ and $\beta\in P(C)$ are \qflamb~objets. Then we have to prove that 
\[(\frecciasopra{B}{f}{A},\alpha)\leq (\frecciasopra{C}{g}{A},\beta). \]
By \eqref{eq Ea leq Eb} we have that $\alpha\leq P_f \Einv_g (\beta)$. Now we can consider the pullback  in $\mC$
\[\pullback{D}{B}{C}{A}{\pbmorph{f}{g}}{\pbmorph{g}{f}}{f}{g}\]
and, after applying BCC,  we obtain
\[\alpha\leq  \Einv_{\pbmorph{f}{g}}P_{\pbmorph{g}{f}}(\beta).\]
Therefore, since $\alpha$ is a \qflamb~object, there exists an arrow $\freccia{B}{h}{D}$ such that
\[ \xymatrix@+1.5pc{
 B\ar[r]^{h}\ar[d]_{\id_B}^{\  }& D \ar[dl]^{\pbmorph{(f)}{g}} \\
B& 
}\]
and \[\alpha\leq P_h P_{\pbmorph{g}{f}}(\beta)\]
From this, it follows that  the diagram
\[\xymatrix@+1.5pc{   B\ar[dr]_f \ar[r]^{(\pbmorph{g}{f})h} &C \ar[d]^g\\
&A
}\]

commutes because $g(\pbmorph{g}{f})h=f(\pbmorph{f}{g})h=f$, and $\alpha\leq P_{(\pbmorph{g}{f})h}(\beta)$. Thus, we have proved that $(\frecciasopra{B}{f}{A},\alpha)\leq (\frecciasopra{C}{g}{A},\beta).$ 
Therefore, we can conclude that $(\id,\unvpropmap)$ is an invertible 1-cell of $\LamEX$, and then $\gencompex{\Pqf}{\Lambda}\cong P$.
\end{proof}
Observe that   when the existential completion   of a conjunctive doctrine $\doctrine{\mC}{P}$ is {\bf full}, i.e. 
 the  existential completion of $P$ is performed with respect to the whole class of morphisms of $\mC$, then such an existential completion comes equipped with comprehensive diagonals.
  From a logical perspective this does not surprise, since the full existential completion freely adds left adjoints, and in particular equality predicates as left adjoint along diagonals.
  \begin{corollary}\label{corollary full ex completion has comprehensive diagonals}
Let $\doctrine{\mC}{P}$ be the full existential completion of a conjunctive subdoctrine $P'$. Then $P$ has comprehensive diagonals. 
\end{corollary}
\begin{proof}
By Theorem \ref{theorem caract. gen. existential completion} we have that every top element is full-existential-free. Now, if we consider two arrows $\freccia{A}{f,g}{B}$, if $\top_A\leq P_{\angbr{f}{g}}(\delta_B)$ then we have that $\exists_{\angbr{f}{g}}(\top_A)\leq \delta_B=\exists_{\Delta_B}(\top_B)$. The by Lemma \ref{lemma E_f (a)<E_g(b)} we can conclude that there exists an arrow $m$ such that $\angbr{f}{g}=\Delta_B m$, and this means that $f=g$. So by proposition 2.2 in \cite{TECH}  we conclude that
 $P$ has comprehensive diagonals.
\end{proof}
Finally we conclude this section by showing another equivalent characterization of $\Lambda$-existential completions via  the comprehension completion construction.

\begin{proposition}\label{proposizione car. P_cI iso weak sub}\label{isopweak}
Let $\doctrine{\mC}{P}$ be a $\Lambda$-existential doctrine, and let us consider a fibred conjunctive subdoctrine $\doctrine{\mC}{P'}$  of $P$. Then the following conditions are equivalent:
\begin{enumerate}
\item  $P=\gencompex{P'}{\Lambda}$ is the $\Lambda$-existential completion of $P'$.
\item There exists an isomorphism $\freccia{P_cI^{\op}}{(\id_{\mG_{P'}},\hat{l})}{\Psi_{U^{-1}(\Lambda)}}$ of $\Lambda$-existential doctrines such that the diagram
\[\xymatrix@+2pc{P' \ar[rd]_{(L,l)}\ar[r]^{(\hat{I},\hat{\iota})}& P_cI^{\op} \ar@{-->}[d]^{(\id_{\mG_{P'}},\hat{l})} \\ 
&\Psi_{U^{-1}(\Lambda)} 
}\]
commutes.
\end{enumerate}
\end{proposition}
 \begin{proof}
 Notice that given a $\Lambda$-existential doctrine $\doctrine{\mC}{P}$ and a fibred subdoctrine $\doctrine{\mC}{P'}$ we always have a full and faithful injection of categories
$$\frecciainj{\mG_{P'}}{I}{\mG_P}$$
given by $I(A,\alpha)=(A,\alpha)$ and $I(f)=f$. Thus, we can consider the $\Lambda$-existential doctrine given by the functor $P_cI^{\op}$
\[\xymatrix{
\mG_{P'}^{\op}\ar@{^{(}->}[r]^{I^{\op}} & \mG_P^{\op} \ar[r]^{P_c} & \infsl
}\] 
which is essentially the restriction of the comprehension completion doctrine $P_c$  to $\mG_{P'}$.
Moreover, notice that we have an injection of doctrines $\frecciainj{P'}{(\hat{I},\hat{\iota})}{P_cI^{\op}}$, where $\hat{I}(A)=(A,\top)$ and $\hat{\iota}(\alpha)=\frecciasopra{(A,\alpha)}{\id_A}{(A,\top)}$.

Therefore, $\it (2)\Rightarrow (1)$ follows because the commutativity of the diagram implies the second point (2) of Theorem \ref{theorem caract. gen. existential completion} as a particular case.

Instead  $\it (1)\Rightarrow (2)$ follows since the fibres $P_c (A,\alpha)$ are contained in $P(A)$, which is isomorphic to $ \Psi_{U^{-1}(\Lambda)}(A,\top)$ by second point (2) of Theorem \ref{theorem caract. gen. existential completion}.
 \end{proof}

In the particular case of the full existential completion we obtain the following corollary relating full existential completions and weak subobjects doctrines.
\begin{corollary}\label{corollary P_cI iso weaksub}
Let $\doctrine{\mC}{P'}$ be a fibred conjunctive subdoctrine of a full existential doctrine $\doctrine{\mC}{P}$. Then the following conditions are equivalent:
\begin{enumerate}
\item  $P=\fullcompex{P'}$ is the full existential completion of $P'$.
\item There exists an isomorphism $\freccia{P_cI^{\op}}{(\id_{\mG_{P'}},\hat{l})}{\Psi_{\mG_{P'}}}$ commuting with the canonical injections of $P'$ in $P$ and of $P'$ in the weak subobjects doctrine  $\Psi_{\mG_{P'}}$.
\end{enumerate}
\end{corollary}

\subsection{A characterization of  pure existential completions}
Here we focus our attention on  \bemph{pure existential completions} $\purecompex{P}$ , i.e.  to existential completions with respect to  the class of product projections in the base category $\mC$ of a primary doctrine $\doctrine{\mC}{P}$ as originally introduced in \cite{ECRT}.
 Contrary to other existential completions, pure existential ones $\purecompex{P}$  inherit the elementary structure from any primary doctrine $P$ generating them and they also induce it on $P$
when they are elementary.

Our aim is to produce  a specific characterization for elementary pure existential completions $\purecompex{P'}$ of a primary doctrine $P'$
 analogous to Proposition \ref{proposizione car. P_cI iso weak sub}, but considering the weak subobjects of the category  of predicates $\Pred{P'}$ instead of the Grothendieck category $\mG_{P'}$.


In \cite{ECRT} it is shown that   the assignment $P\mapsto \purecompex{P}$ of a primary doctrine to its pure existential completion
 extends to a lax-idempotent 2-monads $$\freccia{\PD}{\compex{\mT}}{\PD}$$
on the 2-category of primary doctrines, and that the 2-category $\alg{\compex{\mT}}$ of algebras is isomorphic to the 2-category $\ED$ of existential doctrines. Moreover, the 2-monad $\compex{\mT}$ preserves the elementary structure, i.e. it can be restricted to the 2-category of elementary doctrines.

\begin{example}[Regular fragment of Intuitionistic Logic]\label{example regular fragment is existential comp}
Let $\lang_{=,\exists}$ be the Regular fragment of first order intuitionistic logic \cite{SAE}, i.e. the fragment with equality, conjunction and existential quantification of first order intuitionistic logic. Then the elementary existential doctrine 
$$\doctrine{\mC_{\lang_{=,\exists}}}{\syntdoc_{=,\exists}}$$
is the existential completion of the syntactic elementary doctrine
 \[\doctrine{\mC_{\lang_{=}}}{\syntdoc_{=}}\] 
associated to the Horn fragment $\lang_=$, i.e. the fragment of first order intuitionistic logic with conjunctions and equality. This result is a consequence of the fact that extending the language $\lang_=$ with  existential quantifications is a free operation, so by the known equivalence between doctrines and logic, the elementary existential doctrine $\purecompex{\syntdoc_{\lang_=}}$ must coincide with the syntactic doctrine $\syntdoc_{{=,\exists}}$, since both completions are free.
\end{example}

%

From   theorem~\ref{theorem caract. gen. existential completion} we immediately deduce:
\begin{corollary}\label{theorem caract. existential completion}\label{carpure}
Let $\doctrine{\mC}{P}$ be a pure existential doctrine. Then $P$ is an instance of the pure existential completion construction  
if and only if  the following conditions hold:
\begin{enumerate}
\item $P$ satisfies (RC);
\item $P$ has enough-pure-existential-free objects;
\item for every pure-existential-free object $\alpha$ and $\beta$ of $P(A)$, $\alpha\wedge \beta$ is pure-existential-free object.
\end{enumerate}
\end{corollary}
From a algebraic point of view, since the existential doctrines are exactly the $\mathrm{T}^{\ex}$-algebras \cite{ECRT},  corollary~\ref{carpure} characterizes \emph{free algebras} of the monad $\compex{\mT}$. 

Moreover,  in    \cite{ECRT}  it was shown that the pure existential completion preserves the elementary structure. Here we also prove that an elementary pure existential completion must be the
pure existential completion of an elementary doctrine:
 \begin{theorem}\label{theorem elementary existential completion}
Let $\doctrine{\mC}{P'}$ a primary doctrine and $P=\purecompex{P'}$ its  pure existential completion.  
The following conditions are equivalent
\begin{enumerate}
\item $P'$ is elementary
\item $\purecompex{P'}$ is elementary.
\end{enumerate} 
\end{theorem}
\begin{proof} For $\it (1)\Rightarrow (2)$ see  \cite[Prop. 6.1]{ECRT}.

For $\it (2)\Rightarrow (1)$ we proceed as follows.
First of all recall that left adjoints along the functors $P_{\Delta\times \id}$ are of the form
\[ \Einv_{\Delta\times id}(\alpha)= P_{\angbr{\pr_1}{\pr_2}}(\alpha)\wedge P_{\angbr{\pr_2}{\pr_2}}(\delta_A)\]
Since $P'$ is the subdoctrine of pure-existential-free objects of $P$,  in order to show that $P'$ is elementary it is enough to show that 
the equality predicates $\delta_A$ are pure-existential-free objects so that $ \Einv_{\Delta\times id}$  restricts to $P'$.

We start by showing that $\delta_A$  is a pure-existential-splitting object.
Suppose $\delta_A=\Einv_{\Delta_A}(\top_A) \leq\Einv_{ \pr_{A\times A}}(\beta)$.  Then by the equality adjunction
\begin{equation}\label{eq delta<P_f(a)}
\top_A \leq P_{\Delta_A} \Einv_{ \pr_{A\times A }}(\beta).
\end{equation}
Hence, by BCC, we have that
\begin{equation}\label{eq delta<P_f(a)}
\top_A \leq  \Einv_{ \pr_{A}}P_{\Delta_A\times \id_B}(\beta).
\end{equation}
Since $P$ satisfies the rule of choice by corollary~\ref{carpure}   then there exists an arrow $\freccia{A}{f}{B}$ such that $\top_A\leq P_{\angbr{\id_A}{f}}P_{\Delta_A\times \id_B}(\beta).$ 
 Now since $(\Delta_A\times \id_B)\angbr{\id_A}{f}=(\Delta_A\times f)\Delta_A$ we obtain $\top_A\leq P_{\Delta_A}P_{\Delta_A\times f}(\alpha)$ and by the equality adjunction we conclude
\begin{equation}\label{eq delta<P_f(a)}
\delta_A=\Einv_{\Delta_A}(\top)\leq P_{\Delta_A\times f}(\beta).
\end{equation}
This ends the proof  that $\delta_A$ is pure-existential-splitting object.

Now since pure-existential-free objects are enough for $P$ we have  that $\delta_A =\Einv_{\pr_{A\times A}}(\alpha)$ with $\alpha$ a pure-existential-free object but since $\delta_A$ is a pure existential splitting object  then there exists an arrow $\freccia{A\times A }{g}{B}$ such that $\delta_A\leq P_{\angbr{ id_{A\times A}}{ g}}(\alpha).$ 

Moreover from $\Einv_{\pr_{A\times A}}(\alpha) \leq \delta_A$ we obtain  $\alpha \leq P_{\pr_{A\times A} }(\delta_A)$ and hence we have $P_{ \angbr{id_{A\times A}}{g}}(\alpha) \leq P_{ \angbr{id_{A\times A}}{g}} (P_{\pr_{A\times A}} (\delta_A ))= \delta_A.$ We then conclude that $P_{ \angbr{id_{A\times A}}{g}}(\alpha)= \delta_A$, namely
that $\delta_A$ is a pure-existential-free object.
\end{proof}

\begin{remark}
Theorem~\ref{theorem elementary existential completion} does not hold for all generalized existential completions. In sections~\ref{section realizability hyperdoct} and \ref{supercom} we provide
examples of  full existential  completions which are elementary but their  generating primary doctrines are not.
 
Moreover, a $\Lambda$-weak subobjects doctrine $\doctrine{\mC}{\Psi_{\Lambda}}$ on a finite product category $\mC$  is not necessarily elementary if $\Lambda$ does not contain the diagonal arrow, whilst  its generating  trivial primary doctrine $\doctrine{\mC}{\trdc}$  is elementary, instead.
 \end{remark}
\begin{remark}
Recall from the Remark \ref{rem left adjoint elementary existential doc} that any  elementary existential doctrine $\doctrine{\mC}{P}$ has left adjoint of  every re-indexing functor $P_g$  given by the assignment $$\Einv_g(\alpha)=\Einv_{\pr_1}(P_{\pr_2}(\alpha)\wedge P_{\angbr{\pr_1}{g\pr_2}}(\delta_A)).$$
However, as observed in \cite{TECH},  these left adjoints do not necessarily satisfies BCC conditions unless $P$ has full comprehensions and   comprehensive diagonals.
In this case $P$  not only satisfies the {\em Rule of Choice} by \ref{carpure} but also the Extended Rule of Choice 
as stated in lemma 5.8 of \cite{TECH}.
\end{remark}

%
%
Generalized existential completions  with an elementary structure inherited by their generating subdoctrine are essentially pure existential completions:

\begin{proposition}\label{proposition P generalize ex.com of elementary implies pure existential}\label{genexpure}
Let $\doctrine{\mC}{P}$ be an elementary existential doctrine and  $\Lambda$ is a \lclass\  in $\mC$ containing the projections. Assume also that $P=\gencompex{P'}{\Lambda}$
 is the generalized existential completion of an elementary doctrine $P'$ and that the canonical injection of $P'$ in $P$ preserves the elementary structure. Then $P=\purecompex{P}$ is the pure existential completion of $P'$.
\end{proposition}
\begin{proof}
Clearly the pure existential completion, which call $Q$,  of $P'$ is a subdoctrine of $P$ since  $\Lambda$  contains projections.
The statement of the proposition amounts to show that $Q$ is actually isomorphic to $P$.

To meet our purpose it is enough to show  that every element $\alpha\in P(A)$ is of the form $\alpha=\exists_{\pr_A}(\gamma)$ with $\gamma$ pure-existential-free element and hence it is a fibre object of $Q$. Now, since $P$ is the generalized existential completion of $P'$,  there exists an arrow $\freccia{B}{g}{A}$ and a $\Lambda$-existential-free element $\beta$ such that $\alpha=\exists_g(\beta)$. Moreover, since $P$ is elementary and existential we have that
$$\alpha=\exists_g(\beta)=\Einv_{\pr_A}(P_{\pr_B}(\beta)\wedge P_{\angbr{\pr_A}{g\pr_B}}(\delta_A))$$
which actually gives the claimed representation because we can  show  that 
 $P_{\pr_B}(\beta)\wedge P_{\angbr{\pr_}{g\pr_2}}(\delta_B))$ is a pure-existential-free element   of $P$.
Indeed,  $P_{\angbr{\pr_1}{g\pr_2}}(\delta_B)$ is $\Lambda$-existential-free because $P'$ is elementary and the canonical injection of $P'$ in $P$ is elementary,  i.e. $\delta_B$ is in $P'$ and hence $\Lambda$-existential-free.
Moreover since $\beta$ is also  $\Lambda$-existential-free, by closure of   $\Lambda$-existential-free objects under reindexing and conjunctions
we conclude that $P_{\pr_B}(\beta)\wedge P_{\angbr{\pr_}{g\pr_2}}(\delta_B))$
is also $\Lambda$-existential-free  and finally also  pure-existential-free
 since $\Lambda$ contains the projections.
  This concludes the proof  that $P$  is isomorphic to $Q$, i.e. $P$ is the pure existential completion of $P'$.
\end{proof}
\begin{remark}
Observe that by  Theorem \ref{theorem elementary existential completion} and  (2) of   Theorem \ref{theorem caract. gen. existential completion} under the hypothesis of
Theorem \ref{theorem elementary existential completion} 
we obtain an arrow $\freccia{P}{(\ovln{L},\ovln{l})}{\Psi_{U^{-1}(\Lambda)}}$ which is both a morphism of elementary and existential doctrines. This is because the generalized weak subobjects relative to the class of projections is elementary and existential, since $P'$ is elementary. 
\end{remark}
We can prove a stronger result, which will be fundamental in the sequel of our work and shows  the analogous of Proposition \ref{proposizione car. P_cI iso weak sub}, but considering the weak subobjects of the category  of predicates $\Pred{P'}$ of $P'$ instead of the Grothendieck category $\mG_{P'}$.

%
%
To this purpose we first need to show a lemma.  From here on  we denote the injection of $\Pred{P'}$ into $\Pred{P}$ by $\freccia{\Pred{P'}}{\hat{I}}{\Pred{P}}$.

\begin{lemma}\label{lemma P_cxI= weak su pred}\label{cheldo}
Let $P=\purecompex{P'}$ be the pure existential completion of an elementary doctrine $P'$.
Then, the doctrine $\doctrine{\Pred{P'}}{P_{cx}\hat{I}^{\op}}$ is elementary and existential with comprehensive diagonals and full weak comprehensions. Furthermore, it satisfies the rule of choice (RC). 
\end{lemma}
\begin{proof}
First, $P_{cx}\hat{I}^{\op}$ is elementary and pure existential because both the comprehension completion and the diagonal comprehensive completion preserve the elementary and existential structures, and $\hat{I}^{\op}$ preserves projections.
Now we show that $P_{cx}\hat{I}^{\op}$ has weak comprehensions. Let us consider an element $\beta\in P_{cx}(A,\alpha)$, where $\alpha\in P'(A)$. Now, since $P$ has enough-pure-existential-free elements, we have that $\beta=\exists_{\pr_A}(\gamma)$ where $\gamma\in P'(B\times A)$ and $\freccia{B\times A}{\pr_A}{A}$. Then we define $\comp{\beta}:=\frecciasopra{(B\times A,\gamma)}{\pr_A}{(A,\alpha)}$, that is a morphism of $\Pred{P'}$. Notice that
$\top_{(B,\gamma)} \leq  (P_{cx})_{\comp{\beta}}(\beta)$ since
 $\top_{(B,\gamma)}=\gamma\leq (P_{cx})_{\pr_A}(\exists_{\pr_A}(\gamma))=(P_{cx})_{\pr_A}(\beta)$. Moreover, if we consider an arrow $\freccia{(C,\sigma)}{g}{(A,\alpha)}$ of $\Pred{P'}$ such that $\top_{(C,\sigma)}=\sigma \leq(P_{cx})_g(\beta)$, i.e. $\sigma\leq P_g(\beta)$, in particular we have that $\sigma\leq P_g\exists_{\pr_A}(\gamma)$ then, by BCC we have $\sigma\leq \exists_{\pr_C}P_{\id_B\times g}(\gamma)$ where $\freccia{B\times C}{\pr_C}{C}$. Now, since $\sigma$ is an element of $P'(C)$, it is a pure-existential-free element of $P$, hence there exists an arrow $\freccia{C}{f}{B}$ such that $\sigma\leq P_{\angbr{f}{g}}(\gamma)$ and
 $\freccia{(C,\sigma)}{\angbr{f}{g}}{(B\times A,  \gamma)}$ is such that $g= \comp{\beta} \angbr{f}{g}$.

Thus, we have proved that $\comp{\beta}:=\frecciasopra{(B\times A,\gamma)}{\pr_A}{(A,\alpha)}$ is a weak comprehension of $\beta$. It is direct to check weak comprehensions are full and that $P_{cx}\hat{I}^{\op}$ has comprehensive diagonals. Finally, $P_{cx}\hat{I}^{\op}$ satisfies (RC) because the base category of this doctrine is $\Pred{P'}$, and hence every object of this category is of the form $(A,\alpha)$ with $\alpha$  an element of $P'(A)$  which is in particular, a pure-existential-free elements of $P$. Hence, if $\top_{(A,\alpha)}\leq \exists_{\pr_A}(\beta)$ then, since $\top_{(A,\alpha)}=\alpha$, we can use the universal property of pure-existential-free elements to construct the witness, and concluding that $P_{cx}\hat{I}^{\op}$ satisfies the rule of choice.
\end{proof}

\begin{remark}
Observe that proposition~\ref{cheldo} holds also for elementary generalized existential completions $P$ of  a primary elementary subdoctrine $P'$ with respect to a class $\Lambda$ of morphisms
containing the projections and 
whose  embedding of  $P'$ in $P$  preserves the elementary structure. However, by proposition~\ref{genexpure} such existential completions are  actually pure existential completions.
So the above statement is the most general one.
\end{remark}

Therefore,  we conclude:

\begin{theorem}\label{theorem char pure existential compl of elementary doct}
Let $\doctrine{\mC}{P'}$ be an elementary fibred subdoctrine of an existential elementary doctrine $\doctrine{\mC}{P}$. The following conditions are equivalent:
\begin{enumerate}
\item
$P=\purecompex{P'}$ is the pure existential completion of $P'$.

\item There exists an isomorphism $\freccia{P_{cx}\hat{I}^{\op}}{(\id_{\mG_{P'}},\hat{l})}{\Psi_{\Pred{P'}}}$ of  existential elementary doctrines such that the diagram
\[\xymatrix@+2pc{P' \ar[rd]_{(L,l)}\ar[r]^{(\hat{I},\hat{\iota})}& P_{cx}\hat{I}^{\op} \ar@{-->}[d]^{(\id_{\mG_{P'}},\hat{l})} \\ 
&\Psi_{\Pred{P'}} 
}\]
commutes.
\end{enumerate}
\end{theorem}
\begin{proof}
$\it (1)\Rightarrow (2)$ It follows from the fact that, if $P=\purecompex{P'}$ 
by Lemma \ref{lemma P_cxI= weak su pred} we have that the doctrine $\doctrine{\Pred{P'}}{P_{cx}\hat{I}^{\op}}$ is an existential variational doctrine satisfying the rule of choice, and then, by \cite[Thm. 5.9]{TECH}, $\doctrine{\Pred{P'}}{P_{cx}\hat{I}^{\op}}$ has to be  isomorphic to the weak subobjects doctrine on $\Pred{P'}$. \\

\noindent
$\it (2)\Rightarrow (1)$
It follows by employing the same arguments  in  Theorem \ref{theorem caract. gen. existential completion} and Theorem \ref{theorem char pure existential compl of elementary doct}.
Observe that while in general $[f][g]=[\id]$ in $\Pred{P'}$ does not imply that $fg=\id$ in $\mC$ (recall the notation $[f]$ is introduced in Definition \ref{def extentiona reflection}),  for the specific case of projections if $[\pr][ g]=\id $ in $\Pred{P'}$ then we have that there exists an arrow $g'$ in $\mC$ such that $\pr g'=\id $. Using this, the hypothesis that $P'$ is an elementary subdoctrine of $P$ and the fact that  left adjoint can be written as $\Einv_g(\alpha)=\Einv_{\pr_1}(P_{\pr_2}(\alpha)\wedge P_{\angbr{\pr_1}{g\pr_2}}(\delta_A))$, we can employ the same arguments used for the general case and conclude.
\end{proof}
Combining the previous result with Proposition \ref{proposition P generalize ex.com of elementary implies pure existential} we conclude:
\begin{corollary}
Let $\doctrine{\mC}{P}$ be an elementary $\Lambda$-existential doctrine, where $\Lambda$ is \lclass~containing the projections. Assume also that  $P=\gencompex{P'}{\Lambda}$ is the generalized existential completion of an elementary doctrine $P'$ with respect to $\Lambda$ and that  the canonical injection of $P'$ in $P$  preserves the elementary structure.
Then,  there exists an isomorphism $\freccia{P_{cx}\hat{I}^{\op}}{(\id_{\mG_{P'}},\hat{l})}{\Psi_{\Pred{P'}}}$ of pure existential, elementary doctrines such that the diagram
\[\xymatrix@+2pc{P' \ar[rd]_{(L,l)}\ar[r]^{(\hat{I},\hat{\iota})}& P_{cx}\hat{I}^{\op} \ar@{-->}[d]^{(\id_{\mG_{P'}},\hat{l})} \\ 
&\Psi_{\Pred{P'}} 
}\]
commutes.
\end{corollary}


%


\begin{example}[Regular fragment of Intuitionistic Logic]\label{example regular fragment is existential comp 2}
Observe that the pure-existential-free objects of the existential doctrine
\[\doctrine{\mC_{\lang_{=,\exists}}}{\syntdoc_{=,\exists}}\]
are exactly the formulae which are free from the existential quantifier, and this doctrine has enough-pure-existential-free objects since every formula in the Regular fragment of intuitionistic logic is provably equivalent to a formula in prenex normal form.

Then, we have that the Regular  fragment of Intuitionistic logic $\lang_{=,\exists}$ satisfies the Rule of Choice by Theorem \ref{theorem caract. existential completion}, and more generally, for every formula $[a:A]\; |\; \alpha (a)$ which is quantifier-free, we have that
\[ [a:A]\; |\; \alpha (a)\vdash \exists b:B \; \beta (a,b)\]
then there exists a term $[a:A]\; |\; t(a):B$ such that
\[ [a:A]\; |\; \alpha (a)\vdash \beta (a,t(a)).\]
This property is called \emph{Existence Property} in \cite{vandalen_logic}.
\end{example}

\subsection{A characterization of idempotent $\compex{\mT}$-algebras }

Now we turn our attention to the  \bemph{idempotent} $\compex{\mT}$-\bemph{algebras}, and we show that they coincide with those doctrines which are equipped  with Hilbert's epsilon operator.


The logical relevance of elementary existential doctrines with Hilbert's epsilon operator is provided in \cite{TECH}, where the authors show that these doctrines are exactly those
elementary existential doctrines $P$  such that  their exact completions $ \tripostotopos{P}$ is equivalent to the exact completion $(\Pred{P})_{\mathrm{ex/lex}}$ of the category of predicates
$\Pred{P}$  via an equivalence induced by the canonical embedding
of  $\Pred{P}$  in   $ \tripostotopos{P}$.

Actually  we will derive such a result as a consequence of a more general characterization.
\begin{theorem}\label{theorem P iso P^ex if and only if}
Let $\doctrine{\mC}{P}$ be an existential doctrine. Then $P$ is isomorphic to $\purecompex{P}$ if and only if $P$ is equipped with an $\epsilon$-operators. 
\end{theorem}
\begin{proof}
We need to show that $P$ is isomorphic to $\purecompex{P}$ if and only if for every $A$ and $B$ of $\mC$ and every $\alpha\in P(A\times B)$ there exists arrow $\freccia{A}{f}{B}$ such that
\begin{equation}\label{eq E is less or equal to P_<id,f>}
\Einv_{\pr_1}(\alpha)\leq P_{\angbr{\id_A}{f}}(\alpha)
\end{equation}
where $\freccia{A\times B}{\pr_1}{A}$ is the first projection.
Recall that by \cite{ECRT} we have the following adjunction 
\[\xymatrix@+1pc{
\mC^{\op}\ar[dd]^{\id} \ar[rrd]^{P}_{}="a"  &\\
& \;\dashv& \infsl\\
\mC^{\op} \ar[urr]_{\compex{P}}^{}="b"\\
\ar@/^1pc/^{\eta_P}"a";"b"
\ar@/^1pc/^{\epsilon_P} "b";"a"
}\]
with  
\[
\id_{P}=\epsilon_P\eta_P \;\mbox{  and  } \; \id_{\compex{P}}\leq \eta_P\epsilon_P.
\]
In particular the doctrine $P$ is isomorphic to $\compex{P}$ if an only if $\eta_P\epsilon_P\leq \id_{\compex{P}}$. By definition of the pre-order in the doctrine $\compex{P}$, this inequality holds if and only if for every object $A$ of $\mC$ and every $(\frecciasopra{A\times B}{\pr_1}{A},\alpha\in P(A\times B))\in \compex{P}(A)$, there exists an arrow $\freccia{A}{f}{B}$ such that $\Einv_{\pr_1}(\alpha)\leq P_{\angbr{\id_A}{f}}(\alpha)$.
\end{proof}
\begin{example}
 A notable example of doctrine equipped with Hilbert's $\epsilon$-operator, and hence of $\compex{\mT}$-idempotent algebra thanks to Theorem \ref{theorem P iso P^ex if and only if}, is provided by the well-known hyperdoctrine of subsets $\doctrine{\set}{\mathrm{P}}$.
 
  In particular, thanks to the Axiom of Choice, one can easily check the doctrine $\doctrine{\set}{\mathrm{P}}$ is equipped with $\doctrine{\set}{\mathrm{P}}$.

A second example of $\compex{\mT}$-idempotent algebra is provided by the localic hyperdoctrine presented in Example \ref{example localic doctrine ass. to ordinal}.
\end{example}
\subsection{Examples of generalized existential completions with full comprehensions}\label{section doctrine with full comprehension}
In this section we show that relevant examples of generalized existential doctrines are given by doctrines with full comprehensions closed under compositions.

\begin{definition}
Given  a conjunctive doctrine \[\doctrine{\mC}{P}\]
with full comprehensions, we say that $P$ has \bemph{composable comprehensions}  if its comprehensions are closed under compositions, namely  if  for every  objects $\alpha$ in $P(A)$ and $\beta$ in $P(B)$ with $\comp{\alpha}: B \rightarrow A$
then we have $\comp{\alpha}\comp{\beta}=\comp{\gamma}: C\rightarrow A$  for $\gamma$ in $P(C)$.
\end{definition}

\noindent
Note that $\gamma$ is unique by Lemma~\ref{uniqueness}.

Our aim now is to show  that a doctrince $\doctrine{\mC}{P}$ with full composable comprehensions  is an instance of the generalized existential completion construction with respect to the class $\Lambda_{\cm}$ of comprehensions. 

To this purpose, 
we first show that such a $P$ is closed under  left adjoints along comprehensions:

\begin{proposition}\label{proposition compre compose iff left adjoint}
Let $\doctrine{\mC}{P}$ be a conjunctive doctrine with full comprehensions. Then the following conditions are equivalent:
\begin{enumerate}
\item  $P$ has composable comprehensions.
\item   $P$ has left adjoints along the comprehensions satisfying BCC and FR.
\end{enumerate}
\end{proposition} 
\begin{proof}
We start by showing $\it (2)\Rightarrow (1)$.

Suppose that the doctrine $P$ has left adjoints along comprehensions, and consider two comprehensions $\freccia{C}{\comp{\beta}}{B}$ and $\freccia{B}{\comp{\alpha}}{A}$. We claim that the comprehension 
\[\freccia{D}{\comp{\Einv_{\comp{\alpha}}(\beta)}}{A}\]
is the composition $\comp{\alpha}\comp{\beta}$.
First of all from the unit of the adjunction $\beta\leq P_{\comp{\alpha}}\Einv_{\comp{\alpha}}(\beta)$  it is direct to verify that  
\[\top_C\leq P_{\comp{\beta}}P_{\comp{\alpha}}\Einv_{\comp{\alpha}}(\beta)= P_{\comp{\alpha}\comp{\beta}}\Einv_{\comp{\alpha}}(\beta) \]
Therefore, by comprehension, there exists a unique $\freccia{C}{g}{D}$ such that the following diagram commutes 
\begin{equation}\label{diagramma Eg=comp (a) comp (b) }
\begin{aligned}
\xymatrix{
C\ar[d]_{\comp{\beta}}\ar[rr]^g &&D\ar[ddll]^{\comp{\Einv_{\comp{\alpha}}(\beta)}}\\
B\ar[d]_{\comp{\alpha}}\\
A. }
\end{aligned}
\end{equation}
Now we are going to prove that $g$ is invertible.
Observe that by Lemma \ref{remark tops cover the doctrine with full comp} (or see \cite{TECH}), we have
$\alpha= \Einv_{\comp{\alpha}} (\top_B)$ from which 
\[\Einv_{\comp{\alpha}}(\beta)\leq \Einv_{\comp{\alpha}} (\top_B)\leq \alpha.\]
%
Then 
\[ \top_D\leq     P_{\comp{\Einv_{\comp{\alpha}}(\beta)}} (\,  \Einv_{\comp{\alpha}}(\beta)\, ) \leq    P_{\comp{\Einv_{\comp{\alpha}}(\beta)}}(\alpha).\]
Hence by comprehension there exists a unique $\freccia{D}{h}{B}$ such that the following diagram commutes
\[\xymatrix@+1pc{
D \ar[d]_{\comp{\Einv_{\comp{\alpha}}(\beta)}} \ar[r]^h & B\ar[dl]^{\comp{\alpha}} \\
A.
}\]
Hence
\[ \top_D\leq   P_{\comp{\Einv_{\comp{\alpha}}(\beta)}} (\,  \Einv_{\comp{\alpha}}(\beta)\, ) = P_h(P_{\comp{\alpha}}(\Einv_{\comp{\alpha}}(\beta)) \leq P_h  \beta\]
since $P_{\comp{\alpha}}\Einv_{\comp{\alpha}}(\beta)=\beta$  by BCC and the fact that $\comp{\alpha}$ is monic.

Therefore by comprehension from 
 $\top_D\leq P_h(\beta)$  there is a unique $\freccia{D}{l}{C}$ such that the diagram 
\begin{equation}\label{diagramma comp Exist= comp a comp b l}
\begin{aligned}
\xymatrix@+1pc{
C \ar[r]^{\comp{\beta}}& B \ar[r]^{\comp{\alpha}} & A\\
& D \ar[u]^h \ar[ul]^l \ar[ur]_{\comp{\Einv_{\comp{\alpha}}(\beta)}}
}
\end{aligned}
\end{equation}
commutes. Then, since $ \comp{\Einv_{\comp{\alpha}}(\beta)}g=\comp{\alpha}\comp{\beta}$ by \eqref{diagramma Eg=comp (a) comp (b) }, putting together the known diagrams
 \[\xymatrix@+1.5pc{
D\ar[r]^l \ar[rd]_{\comp{\Einv_{\comp{\alpha}}}}& C\ar[d]^{\comp{\alpha}\comp{\beta}} \ar[r]^g & D\ar[dl]^{\comp{\Einv_{\comp{\alpha}}}}\\
& A.
}\]
we get $\comp{\Einv_{\comp{\alpha}}(\beta)}gl=\comp{\Einv_{\comp{\alpha}}(\beta)}$ and since $\comp{\Einv_{\comp{\alpha}}(\beta)}$ is monic  we conclude $gl=\id_D$.
Analogously, we get
$$\comp{\alpha}\comp{\beta} lg = \comp{\alpha}\comp{\beta}$$ and   hence $ lg=\id_C$ since $\comp{\alpha}\comp{\beta}$ is monic.  Therefore $\comp{\Einv_{\comp{\alpha}}(\beta)}=\comp{\alpha}\comp{\beta}$, namely we have shown  that comprehensions of $P$ compose.

Now we show that $\it (1)\Rightarrow (2)$,  namely that if comprehensions compose then there exists the left adjoint along $\comp{\alpha}$ defined as
\[\Einv_{\comp{\alpha}}(\beta):=\gamma\]
where $\gamma$ is the {\it unique} element of $P(A)$ such that $\comp{\alpha}\comp{\beta}=\comp{\gamma}$.

First, observe that,  since comprehensions are full, we have that $\gamma\leq \sigma$, where $\sigma\in P(A)$, if and only if 
\[\top\leq P_{\comp{\gamma}}(\sigma)\]
and then, by definition of $\gamma$, this holds if and only if 
\[\top\leq P_{\comp{\alpha}\comp{\beta}}(\sigma).\] 
Then, again by fullness, the previous inequality holds if and only if 
\[\beta\leq P_{\comp{\alpha}}(\sigma).\]
Therefore, we have that $\Einv_{\comp{\alpha}}$ is left adjoint to $P_{\comp{\alpha}}$, and it is a direct to verify that these left adjoints satisfy (BCC) and (FR).


\end{proof}


\begin{example}\label{example m-cat}
Given an $\mM$-category $(\mC,\mM)$ one can define the doctrine of $\mM$-\emph{subobjects}
\[	\subobjdoctrine{\mC}{\mM} \]
where $\Sub_{\mM} (X)$ if the inf-semilattice of $\mM$-subobjects of $\mM$, and the action of $\Sub_{\mM}$ on a morphism $\freccia{X}{f}{Y}$ of $\mC$ is given by pulling back the $\mM$-subobjects of $Y$ along $f$.
Recall that a $\mM$-\emph{category} is a pair $(\mC,\mM)$ where $\mC$ is a category and $\mM$ is a stable system of monics, i.e. $\mM$ is a collection of monics which includes all isomorphisms and is closed under composition and pullbacks. 
Observe that a stable system of monics is essentially what was called a \emph{dominion} in \cite{CET}, an \emph{admissible system of subobjects} in \cite{CPM},
 a \emph{notion of partial maps} in \cite{DDC} and a \emph{domain structure} in \cite{PMCPCCC}

It is direct to prove that given an $\mM$-category $(\mC,\mM)$, the doctrine of $\mM$-subobjects $\subobjdoctrine{\mC}{\mM}$ has full composable comprehensions.

%
%
\end{example}

\begin{proposition}\label{lemma lamdba_com RC}
Let $\doctrine{\mC}{P}$ be a conjunctive doctrine with full composable comprehensions, and let $\Lambda_{\cm}$ be the class of the comprehensions. Then:
\begin{enumerate}
\item $\Lambda_{\cm}$ is a \lclass~of $\mC$;
\item for every $\alpha\in P(A)$, $\alpha=\exists_{\comp{\alpha}}(\top_{A_{\alpha}})$;
\item an element $\alpha\in P(A)$ is a \quantfree{\Lambda_{\cm}}~object if and only if is the top element. In particular the doctrine $P$ satisfies the $\Lambda_{\cm}$-(RC).
\end{enumerate}
\end{proposition}
\begin{proof}
$1.$ The class of comprehensions contains identities, and comprehensions are stable under pullbacks by Remark \ref{remark compr has pullbacks} while they compose by Proposition \ref{proposition compre compose iff left adjoint}. \\

\noindent
$2.$ The first point follows from Lemma \ref{remark tops cover the doctrine with full comp}.\\

\noindent
$3.$ First we show that every top element is \quantfree{\Lambda_{\cm}}. If $\top_A\leq \Einv_{\comp{\alpha}}(\beta)$, then we have that $\comp{\top_A}=\id_A$ factors on $\comp{\Einv_{\comp{\alpha}}(\beta)}$, which is equal to the arrow $\comp{\alpha}\comp{\beta}$ by Proposition \ref{proposition compre compose iff left adjoint}. Then there exists an arrow $h$ such that the following diagram commutes
\[\xymatrix{
A\ar[dd]_{\id_A} \ar[rr]^h&& C\ar[dl]^{\comp{\beta}}\\
& A_{\alpha}\ar[dl]^{\comp{\alpha}}\\
A.
}\]
Now we define $f:=\comp{\beta}h$. Thus, we have that $\comp{\alpha}f=\id_A$ and 
\[P_{f}(\beta)=P_{h}P_{\comp{\beta}}(\beta)=\top_A.\]

Now we prove the converse. By point $2.$ we have that $\alpha=\Einv_{\comp{\alpha}}(\top_{A_{\alpha}})$, and then if $\alpha$ is a \quantfree{\Lambda_{\cm}}~object, $\alpha\leq \Einv_{\comp{\alpha}}(\top_{A_{\alpha}})$ implies that there exists an arrow $\freccia{A}{f}{A_{\comp{\alpha}}}$ such that $\alpha\leq P_{f}(\top_{A_{\alpha}})=\top_A$ and $\comp{\alpha}f=\id_A$. Since $\id_A=\comp{\top_A}$ by fullness of comprehensions we conclude
$\top_A\leq \alpha$ and hence $\top_A=\alpha$.
\end{proof}

From this theorem  we obtain the following one:
\begin{theorem}\label{theorem P full comp if if ex comp}
Every \conjdoctrine~$\doctrine{\mC}{P}$ with full and composable comprehensions is an instance of the generalized existential completion construction with respect to the class $\Lambda_{\cm}$ of comprehensions. In particular $P$ is the generalized existential completion of the \conjdoctrine
\[\doctrine{\mC}{\trdc}\]
where for every object $A$ the poset $\trdc(A)$ contains only the top element.
\end{theorem}
\begin{proof}
By point $(3)$ of Theorem~\ref{theorem caract. gen. existential completion} and Proposition~\ref{lemma lamdba_com RC}.
\end{proof}
\begin{remark}
When $\doctrine{\mC}{P}$ is a doctrine with full composable comprehensions, then the generalized existential completion of the trivial doctrine $\doctrine{\mC}{\trdc}$ is exactly the doctrine $\subobjdoctrine{\mC}{\Lambda_{\cm}}$ where $\Sub_{\Lambda_{\cm}}(A)$ is the class of $\Lambda_{\cm}$-subobjects of $A$.

In particular, by Theorem \ref{theorem P full comp if if ex comp} we have an isomorphism given by the 1-cell
\[\duemorfismo{\mC}{P}{\id_{\mC}}{\comp{-}}{\mC}{{\Sub_{\Lambda_{\cm}}}}\]
where $\freccia{P(A)}{\comp{-}_A}{\Msubobj{\Lambda_{\cm}}{A}}$ sends  $\alpha$ to $\comp{\alpha}$. 
\end{remark}
In particular, one can directly check that the previous result, together with Example \ref{example m-cat}, extends to an isomorphism of 2-categories. 
\begin{theorem}\label{theorem equiv m-cat comprensioni}
We have an isomorphism of 2-categories
$$\Mcat\cong\CompAdj$$
where $\Mcat$ is the 2-category of $\mM$-categories and the 2-category $\CompAdj$ of doctrines with full composable comprehensions.
\end{theorem}

\begin{corollary}\label{cor m-var doct are lambda existential comp}
Every m-variational existential doctrine is an instance of generalized existential completion construction with respect to the class of comprehensions.
\end{corollary}\label{subgen}
\begin{remark}\label{rem m-var equivalent to fact system}
Recall that in the case of existential m-variational doctrines, the equivalence of Theorem \ref{theorem equiv m-cat comprensioni} restricts to an equivalence between the 2-category of existential m-variational doctrines and the 2-category of proper stable factorization systems \cite{NRRFS}. We refer to \cite{UEC} and \cite{FSF} the proof of this equivalence.
\end{remark}
In the case of the subobjects doctrine, we have that every that the class of comprehensions is exactly those of all the  monomorphims.
\begin{corollary}
The subobjects doctrine $\doctrine{\mC}{\Sub_{\mC}}$ of a category $\mC$ with finite limits is the generalized existential completion of the trivial doctrine $\doctrine{\mC}{\trdc}$, with respect to the class of all the monomorphisms $\Lambda_{\mono}$.
\end{corollary}

The previous characterizations, in particular of Lemma \ref{lemma lamdba_com RC}, provide also another presentation of existential m-variational doctrines satisfying the Rule of Unique Choice.
\begin{proposition}\label{proposition caratterizzazione RUC}
\label{mainsub}
Let $\doctrine{\mC}{P}$ be a primary doctrine. Then the following are equivalent:
\begin{enumerate}
\item $P$ is existential, m-variational, and it satisfies the Rule of Unique Choice;
\item $\mC$ is regular and $P=\Sub_{\mC}$;
\item $P$ is the $\mM$-existential completion of the trivial doctrine $\doctrine{\mC}{\trdc}$, where $\mM$ is a class of morphisms of $\mC$ such that 
\begin{itemize}
\item there exists a class $\mathcal{E}$ of morphisms of $\mC$ such that $(\mathcal{E}, \mM)$ is a proper, stable, factorization system on $\mC$;
\item for every arrow $\freccia{A\times B}{\pr_A}{A}$ of $\mC$, if $\pr_Af$ is a monomorphism and $f\in \mM$ then $\pr_Af\in \mM$.
\end{itemize}

\end{enumerate}
\end{proposition}
\begin{proof}
$\mathrm{(1 \Rightarrow 2)}$ It follows from \cite[Prop. 5.3]{TECH} and  \cite[Thm. 4.4.4 and Thm. 4.9.4]{CLTT}.\\

\noindent
$\mathrm{(2 \Rightarrow 3)}$ It follows from Corollary~\ref{subgen}.\\

\noindent
$\mathrm{(3 \Rightarrow 1)}$ 
Under these assumptions, we have that $P$ is an existential m-variational doctrine by Remark \ref{rem m-var equivalent to fact system} and Corollary \ref{cor m-var doct are lambda existential comp}, and the arrows of $\mM$ are exactly the comprehensions of $P$.
Now, let us consider a functional, entire relation $\rho\in P(A\times B)$. First, notice that $\rho$ functional implies $\pr_A\comp{\rho}$ monic by Lemma \ref{lemma rel funzionale iff pr rho mono}. Hence, by our assumption, $\pr_A\comp{\rho}$  is a comprehension. Moreover, we have that
\begin{equation}\label{eq rho is entire}
\top_A\leq \Einv_{\pr_A}(\rho)=\Einv_{\pr_A\comp{\rho}}(\top)
\end{equation}
because $\rho$ is entire and $\rho=\Einv_{\comp{\rho}}(\top)$ by Lemma \ref{remark tops cover the doctrine with full comp}. Therefore we can apply Lemma \ref{lemma lamdba_com RC}, because $\pr_A\comp{\rho}$ is a comprehension. Hence,  there exists an arrow $\freccia{A}{f}{(A\times B)_{\rho}}$ such that 
\[\top_A\leq P_f(\top)\]
and $\pr_A\comp{\rho}f=\id_A$, namely $\comp{\rho} f= \angbr{\id_A}{\pr_B\comp{\rho} f}$.
In particular, we have that 
\[\top_A\leq P_f(\top)=P_{\comp{\rho}f}(\rho)=P_{\angbr{\id}{\pr_B\comp{\rho} f}}(\rho).\]
Therefore, $P$ satisfies the (RUC).
\end{proof}


%

\subsection{The realizability hyperdoctrine}\label{section realizability hyperdoct}
In this section we are going to prove that  all the realizability triposes  \cite{TT,TTT,van_Oosten_realizability} are full generalized existential completions as shown independently in \cite{Frey2020} and presented in the talk by the second author \url{https://www.youtube.com/watch?v=sAMVU_5RQJQ&t=688s} in 2020. A different presentation of realizability triposes via free constructions (essentially acting on the representing object of a given representable indexed preorder) can be found in \cite{Hofstra06}. 
 
We start by recalling some related fundamental notions.
 
A \bemph{partial combinatory algebra} (pca) is specified by a set $\pca{A}$ together with a partial binary operation $(-)\cdot (-): \pca{A}\times \pca{A} \rightharpoonup \pca{A}$ for which there exist elements $k,s\in \pca{A}$ satisfying for all $a,a',a''\in \pca{A}$ that
\[k\cdot a \downarrow \mbox{ and } (k\cdot a)\cdot a'\equiv a\] 
and
\[s\cdot a \downarrow \mbox{, } (s\cdot a)\cdot a'\downarrow \mbox{, and }((s\cdot a)\cdot a')\cdot a''\equiv (a \cdot a'')\cdot (a'\cdot a'')\] 
where $e\downarrow$ means "$e$ is defined" and $e\equiv e'$  means "$e$ is defined if and only $e'$ is, and in that case they are equal". For more details we refer to \cite{van_Oosten_realizability}.

Given a pca $\pca{A}$, we can consider the \bemph{realizability hyperdoctrine} $$\doctrine{\set}{\mP}$$ 
over $\set$ introduced in \cite{TT,TTT}. For each set $X$, the partial ordered set $(\mP(X),\leq)$ is defined as follow:
\begin{itemize}
\item $\mP(X)$ is the set of functions $\pwset{\pca{A}}^X$ from $X$ to the powerset $\pwset{\pca{A}}$ of $\pca{A}$;
\item given two elements $\alpha$ and $\beta$ of $\mP(X)$, we say that $\alpha\leq \beta$ if there exists an element $\ovln{a}\in \pca{A}$ such that for all $x\in X$ and all $a\in \alpha (x)$, $\ovln{a}\cdot a$ is defined and it is an element of $\beta (x)$.  By standard properties of pca this relation is reflexive and transitive, i.e. it is a preorder.
Then $\mP (X)$ is defined as the quotient of $\pwset{\pca{A}}^X$ by the equivalence relation generated by $\leq$.  The partial order on the equivalence classes $\eqc{\alpha}$ is that induced by $\leq$.
\end{itemize}
Given a function $\freccia{X}{f}{Y}$ of $\set$, the functor $\freccia{\mP (Y)}{\mP_f}{\mP (X)}$ sends an element $\eqc{\alpha}\in \mP (Y)$ to the element $\eqc{\alpha\circ f }\in \mP (X)$ given by the composition of the two functions. With the these assignments $\doctrine{\set}{\mP}$ is an hyperdoctrine, see \cite[Example 2.3]{TTT} or \cite{van_Oosten_realizability}. 

In particular we recall from some Heyting operations which give the structure of Heyting algebra to the  fibres $\mP(X)$. First, we recall from \cite{van_Oosten_realizability} that from the elements $k$ and $s$ one can construct elements $\bf{p}$, $\bf{p_1}$ and $\bf{p_2}$ which are called \emph{pairing} and \emph{projections} operators. Hence, we can consider the following assignments:
\begin{itemize}
\item $\top_X:=\eqc{\lambda x\in X.\pca{A}}$;
\item $\eqc{\alpha}\wedge \eqc{\beta}:=\eqc{\lambda x\in X.\{({\bf p}\cdot a)\cdot b \;|\; a\in \alpha (x) \mbox{\;and \;} b\in \beta (x)\} }$
\item $\delta_X:=\eqc{\lambda (x_1,x_2)\in X\times X.\pca{A}\mbox{ if } x_1=x_2, \emptyset\mbox{ otherwise}}$;
\item for every projection $\freccia{X \times Y}{\pr_X}{X}$, the functor $\Einv_{\pr_X}$ sends an element $\eqc{\gamma}\in \mP (X\times Y)$ to the following:
$$\Einv_{\pr_X}(\eqc{\gamma})=\eqc{\lambda x\in X. \bigcup_{y\in Y} \gamma (x,y)}. $$
\end{itemize} 
In particular we have that $\mP$ is elementary and existential and then, in particular, it has left adjoints along arbitrary functions, i.e. it is a full existential doctrine.
 
Notice also these left adjoints along any function satisfy BCC and FR (see \cite[Lem. 5.2]{realizability} or \cite{van_Oosten_realizability}).

\noindent
Now we show that the realizability hyperdoctrine is an instance of generalized existential completion. Thus, we fix the class $\Lambda_{\set}$ to be the class of all functions of $\set$. The realizability hyperdoctrine is, in particular, existential and elementary, and then it has left adjoints along all the morphisms of $\set$, so it is a full existential doctrine. The rest of this section is devoted to prove that $\mP$ is a full existential completion.

Hence, we need to understand what are the full-existential-free  objects of the realizability hyperdoctrine. 
\begin{definition}
Let $\pca{A}$ be a pca, and let $\doctrine{\Set}{\mP}$ be the realizability tripos associated to the pca $\pca{A}$. We call an element $\freccia{X}{\eqc{\gamma}}{\pwset{\pca{A}}}$ of the fibre $\mP(X)$ a \bemph{singleton predicate} if there exists a singleton function $\freccia{X}{\alpha}{\pwset{\pca{A}}}$, i.e. $\alpha(x)=\{a\}$ for some $a$ in $\pca{A}$, such that  $\gamma\sim \alpha$.
\end{definition}
We claim that the singleton predicates are exactly the full-existential-free objects provided that we assume {\it the axiom of choice} in our meta-theory as we do.

\begin{lemma}\label{lemma singleton a free-quantifier}
Every singleton predicate is a full-existential-free object.
\end{lemma}
\begin{proof}
Assume that $\freccia{X}{\eqc{\gamma}}{\pwset{\pca{A}}}$ is a singleton predicate, i.e. $\gamma$ assigns to every element $x$ of $X$ a singleton $\gamma (x)=\{ a\}$ for some $a$ in $\pca{A}$.
In order to show that this is a full-existential-free object is enough to show that it is a full-existential splitting since the action of $\mP_f$ preserves singletons, i.e.  for every function $\freccia{Z}{m}{X}$ of $\set$, we have that $\mP_m(\eqc{\gamma})=\eqc{\gamma\circ m}$ is again a singleton predicates.
To this purpose, 
observe that if  $\eqc{\gamma}\leq \Einv_{g}(\eqc{\beta})$ for some $\eqc{\beta} \in \mP (Y)$, then there exists an element $\ovln{b}\in \pca{A}$ such that for every $x\in X$ and every $a\in \gamma (x)$, $\ovln{b}\cdot a\in \Einv_g \beta (x)$. By Remark \ref{rem left adjoint elementary existential doc} we have 
$$\Einv_g(\beta)= \Einv_{\pr_X}(\mP_{\pr_Y}(\beta)\wedge \mP_{g\times \id_X}(\delta_X))$$
 and since $\gamma (x)$ is a singleton $\{a\}$ for every $x\in X$, we have that $\ovln{b}\cdot a\in \bigcup_{y\in Y}(\mP_{\pr_Y}(\beta)\wedge \mP_{g\times \id}(\delta_Y))(x,y)$. Hence we have that $\ovln{b}\cdot a= ({\bf p}\cdot c_1)\cdot c_2$ for some $c_1\in \beta (x,y_a)$, $c_2\in \pca{A}$, and for some $y_a\in Y$ such that $g(y_a)=x$ . By the Axiom of Choice, we can define a function $\freccia{X}{f}{Y}$ such that $f(x)=y_a$. In particular, we have that $\lambda y. {\bf p_1}(\ovln
{b}y)$ realizes $\gamma \leq \mP_{f}(\beta)$, and we have that $gf=\id$. This concludes the proof that  singletons predicates are full-existential splitting.
\end{proof}

\begin{corollary}\label{lemma realizability hyperdoc top is a free quant}
For every set $X$, $\eqc{\top_X}\in \mP(X)$ is a full-existential-free object.
\end{corollary}
\begin{proof}
Let us consider an element $c\in \pca{A}$ and the singleton function $\freccia{X}{\iota_c}{\pwset{\pca{A}}}$ given by the constant assignment $x\mapsto \{ c\}$. It is straightforward to check that $\top_X\sim \iota_c$, i.e. that $\eqc{\top_X}$ is a singleton predicate, and then a full-existential-free object by Lemma \ref{lemma singleton a free-quantifier}. 
\end{proof}

\begin{remark}
Notice that from Lemma \ref{lemma realizability hyperdoc top is a free quant}  it follows that the realizability hyperdoctrine $\doctrine{\set}{\mP}$ satisfies the Extended Rule of Choice.
\end{remark}

Now we ready to show the main result of this section.
\begin{theorem}\label{theorem realizability hyper. is gen ex comp}
The realizability hyperdoctrine $\doctrine{\set}{\mP}$ is the full existential completion of the primary doctrine $\doctrine{\set}{\mP^{\singleton}}$ whose elements of the fibre $\mP^{\singleton}(X)$ are only singleton predicates of $\mP(X)$.
\end{theorem}
\begin{proof}
By Lemma \ref{lemma singleton a free-quantifier} we have that every singleton predicate is a full-existential-free element, and in particular,  by corollary \ref{lemma realizability hyperdoc top is a free quant} we have that also every top element is a full-existential-free element. Notice that singleton predicates are closed under binary meet since if $\eqc{\alpha}$ and $\eqc{\beta}$ are singletons of $\mP (X)$, by definition, we have that
$$(\eqc{\alpha}\wedge \eqc{\beta})=\eqc{\lambda x \in X.\{(p\cdot a)\cdot b \;|\; a\in \alpha (x) \mbox{\;and \;} b\in \beta (x)\}}.$$
is again a singleton function. Moreover, one can directly check that the singleton predicates \emph{cover} the realizability hyperdoctrine, i.e. for every $\eqc{\beta}\in \mP(X)$ there exists a singleton predicate $\eqc{\alpha}\in \mP (Y)$ and a function $\freccia{Y}{g}{X}$ such that $\eqc{\beta}=\Einv_{g}(\eqc{\alpha})$. Then, by Proposition \ref{proposition doctrine P' cover P} we have that singleton predicates are exactly the full-existential-free elements of $\mP$. Therefore the realizability hyperdoctrine satisfies all the conditions of Theorem \ref{theorem caract. gen. existential completion} and then we can conclude that  it is the full existential completion of the primary doctrine $\mP^{\singleton}$ of singletons.
\end{proof}


\begin{remark}
Observe that the doctrine $\doctrine{\set}{\mP^{\singleton}}$ of singletons is closed under universal quantifiers. Then, by combining Theorem \ref{theorem realizability hyper. is gen ex comp} with the results presented in \cite{TSPGF}, we can conclude that the realizability doctrine satisfies the so called \bemph{principle of Skolemization}:
\[\forall u \exists x \;\alpha(u,x,i)=\exists f \forall u\; \alpha(u,fu,i)\]
being the existential completion of a universal doctrine whose base is cartesian closed. This form of choice principle  
plays a central role in the dialectica interpretation \cite{goedel1986} and its categorical presentation, see \cite{hofstra2011,TSPGF} for more details.
\end{remark}

\subsection{Localic doctrines}\label{supercom}

Let us consider a locale $\loc$, i.e. $\loc$ is a poset with finite meets and arbitrary joins, satisfying the \emph{infinite distributive law} $x \wedge (\bigvee_i y_i)=\bigvee_i(x\wedge y_i)$. Recall from \cite{TT,TTT} that given a locale $\loc$ we can define the \emph{canonical localic doctrine} of $\loc$:
\[\doctrine{\set}{\loc^{(-)}}\]
by assign $I\mapsto \loc^I$, and the partial order is provided by the pointwise partial order on functions $\freccia{I}{f}{\loc}$. Propositional connectives are defined pointwise. The existential quantifier along a given function $\freccia{I}{f}{J}$ maps a function $\phi\in \loc^I$ to $\Einv_f(\phi)$ given by $j\mapsto \bigvee_{\{ i\in I|f(i)=j\}}\phi (i)$ and these are called existential since they satisfy
the FR and BCC conditions.

Now we are going to consider localic triposes \cite{TT,TTT} whose locale is  \bemph{supercoherent} as defined in \cite{BANASCHEWSKI199145}.
For the reader convenience we just recall few related basic notions from \cite{BANASCHEWSKI199145}.

\begin{definition}
An element $c$ of a locale $\loc$ is said \bemph{supercompact} if whenever $c\leq \bigvee_{k\in K} b_k$, there exists $\ovln{k}\in K$ such that $c\leq b_{\ovln{k}}$.
\end{definition}

\begin{remark}
Note that a supercompact element of a non trivial locale  must be different from the bottom of the locale.
\end{remark}

\begin{definition}
A local $\loc$ is called \bemph{supercoherent} if:
\begin{itemize}
\item each element $d\in \loc$ is a join $d=\bigvee_I c_i$ of supercompact elements $c_i$;
\item supercompact elements are closed under finite meets.
\end{itemize}
\end{definition}

\begin{remark}
The category $\supercohframe$ of supercoherent frames is a coreflexive subcategory of the category of frames. This result together with the general notion of supercoherent frame was introduced in \cite{BANASCHEWSKI199145} where the authors show that every regular projective complete lattice is exactly a retract of a supercoherent frame \cite[Cor. 6]{BANASCHEWSKI199145}.
\end{remark}
Let $\meetcat$ be the category of meet-semilattices, and let $\framecat$ be the category of frames. Recall from \cite{johnstone1982stonespaces} that we can define a functor 
$$\freccia{\meetcat}{D}{\framecat}$$ 
sending a meet-semilattice $\mM$ to the \emph{down-set lattice} $D(\mM)$, i.e. the lattice of all the $X\subseteq \mM$ such that $a\in X$ implies  that  for all $b\leq a$ we have $b\in X$ and the order is provided by the set-theoretical inclusion. This functor is left adjoint to the inclusion functor, see \cite[Lem. 1]{BANASCHEWSKI199145} or \cite[Thm. 1.2]{johnstone1982stonespaces}. In particular, we have a natural injection $\freccia{\mM}{\eta}{D(\mM)}$ sending $a\mapsto \downarrow (a)$. It is direct to see that sets of the form $\downarrow (a)$ are supercompact elements in $D(\mM)$.
\begin{remark}\label{remark equivalecen supercompframe }
The functor $\freccia{\meetcat}{D}{\framecat}$ induces an equivalence $\meetcat\equiv\supercohframe.$ Essentially this means that supercoherent frames are the frame completion of inf-semilattices.
\end{remark}

\begin{definition}
Let $\loc$ be an arbitrary locale, and let $\doctrine{\Set}{\loc^{(-)}}$ be the localic doctrine. We call \bemph{supercompact predicate} an element $\phi\in \loc^J$ such that $\phi(j)$ is a supercompact element for every $j\in J$ .
\end{definition}
\begin{lemma}\label{lemma supercompact function}
Every supercompact predicate of the localic doctrine $\doctrine{\set}{\loc^{(-)}}$ is a full-existential-free element.
\end{lemma}
\begin{proof}
Let  $\phi\in \loc^J$ be a supercompact predicate. If we have
\[\phi \leq \Einv_f (\psi)\]
for some $\freccia{I}{f}{J}$ and $\psi\in \loc^I$, then in particular
\[\phi(j)\leq \bigvee_{\{i\in I| f(i)=j\}}\psi (i)\]
and, since $\phi(j)$ is supercompact, there exists $\ovln{i}^j\in \{i\in I| f(i)=j\}$ such that $\phi(j)\leq\psi(\ovln{i}^j)$. Hence, we can define a function $\freccia{J}{g}{I}$ sending $j\mapsto \ovln{i}^j$. Hence, by definition, we have that $fg=\id$ and $\phi\leq \loc^{(-)}_g(\psi)$, i.e. $\phi$ is full-existential-splitting. Moreover it is direct to see that supercompact predicate are stable under re-indexing, and hence supercompact predicates are full-existential-free.

\end{proof}

\begin{theorem}\label{theorem caract localic doct}
Let $\loc$ be a locale. 
\begin{enumerate}
\item If $ \loc$ is a supercoherent locale then the localic doctrine $\doctrine{\set}{\loc^{(-)}}$ is a full existential completion  of the primary doctrine $\doctrine{\set}{\supercomp^{(-)}}$ of supercompact predicates of $\loc$.
\item
If the localic doctrine $\doctrine{\set}{\loc^{(-)}}$ is a full existential completion then  its locale $\loc$ is supercoherent.

\end{enumerate}

\end{theorem}
\begin{proof}

(1)
Let $\loc$ be a supercoherent locale. By Lemma \ref{lemma supercompact function} supercompact predicates are full-existential-free. Moreover since $\loc$ is supercoherent, the supercompact predicates are closed under finite meets, and since every element of $a\in\loc$ is a join of supercompact elements, then every element of every fibre of $\loc^{(-)}$ is covered by a supercompact predicates. In particular, if we consider an element $\freccia{I}{\phi}{\loc}$, we have that for every $i\in I$, $\phi(i)=\bigvee_{j\in J_i}c_j^i$ with $c^i_j$ supercompact elements. So, we can define the disjoint sum $J=\oplus_{i\in I}J_i$ and a supercompact predicate $\freccia{J}{\psi}{\loc}$ given by $\psi(i,j)=c^i_j$. Then it is direct to see that $\phi=\exists_f(\psi)$ where $\freccia{J}{f}{I}$ is the function mapping $f(i,j)=i$.

Hence,  by Proposition \ref{proposition doctrine P' cover P} we have that supercompact elements are exactly the full-existential-free elements of $\loc^{(-)}$  and by Theorem \ref{theorem caract. gen. existential completion} we can conclude that $\loc^{(-)}$ is the full existential completion of the primary doctrine $\doctrine{\set}{\supercomp^{(-)}}$ such that $\supercomp (I)$ is the inf-semilattice whose objects are the supercompact predicates of $\loc^I$. 

(2)
Suppose that the localic doctrine $\doctrine{\set}{\loc^{(-)}}$ is a full existential completion. By Lemma \ref{lemma supercompact function} we know that supercompact elements of $\loc$ are values  of full-existential-free elements, hence supercompact elements are closed under finite meets because  full-existential-free elements are closed under them. Now suppose that $\phi$ is a full-existential-free object. Let  $\tilde{j}$ any index in $J$ and suppose that $\phi(\ltil{j})\leq \bigvee_{i\in I} b_i$. Then we can define an element $\freccia{J\times I}{\psi}{\loc}$ such that $\psi(\ltil{j},i)=b_i$ and $\psi(j,i)=\top$ for every $j\neq \ltil{j}$. Hence  it follows  that $\phi(j)\leq \bigvee_{i\in I}\psi (j,i)$ for each $j$ in $J$ which means $\phi \leq \Einv_{\pr_I}(\psi)$.  Then, since $\phi$ is a full-existential-free objec,  there exists a function $\freccia{J}{g}{I}$ such $\phi\leq \loc^{(-)}_{\angbr{\id_J}{g}}(\psi)$,  and then, in particular $\phi(\ltil{j})\leq \psi(\ltil{j},g(\ltil{j}))=b_{g(\ltil{j})}$. Hence every element $\phi(j)$ is supercompact. This concludes the proof that the full-existential-free elements of $\loc^{(-)}$ are exactly the supercompact predicates.

Finally, every element $a\in \loc$ is the join of supercompact elements because for such an element we can define the function $\freccia{\{\ast\}}{\alpha}{\loc}$ as $\ast \mapsto a$, and  since the doctrine has enough-full-existential-free elements, we have $\alpha= \Einv_f \phi$ for a full-existential-free object $\phi$, i.e. $a$ is the join of supercompact elements.
\end{proof}


%

\section{Regular completion via generalized existential completion}
The notion of elementary existential doctrine contains the logical data which allow to describe relational composition as well as functionality and entirety. Given an elementary and existential doctrine $\doctrine{\mC}{P}$, an element $\alpha\in P(A\times B)$ is said \bemph{entire} from $A$ to $B$ if 
$$\top_A\leq \Einv_{\pr_A}(\alpha).$$ 
Moreover it is said \bemph{functional} if 
\[ P_{\angbr{\pr_1}{\pr_2}}(\alpha)\wedge P_{\angbr{\pr_1}{\pr_3}}(\alpha)\leq P_{\angbr{\pr_2}{\pr_3}}(\delta_B)\]
in $P(A\times B\times B)$. Notice that for every relation $\alpha\in P(A\times B)$ and $\beta\in P(B\times C)$, the \bemph{relational composition of} $\alpha$ and $\beta$ is given by the relation 
\[ \Einv_{\angbr{\pr_1}{\pr_3}}(P_{\angbr{\pr_1}{\pr_2}}(\alpha)\wedge  P_{\angbr{\pr_2}{\pr_3}}(\beta)\]
in $P(A\times B)$, where $\pr_i$ are the projections from $A\times B\times C$.

Moreover, as pointed out in \cite{NRRFS,TECH},  for elementary existential m-variational doctrines, i.e. with full comprehensions and comprehensive diagonals, we can define their {\em regular completion} as follows:
\begin{definition}
Let $P$ be an elementary and m-variational doctrine. Then
we define {\it its  m-variational regular completion} as 
 the category $\Ef{P}$,
 whose objects are those of $\mC$, and whose morphisms are entire functional relations.
 \end{definition}

Recall from Section  \ref{section notions of doctrines} that the doctrine $P_{c}$ denotes the free-completion with full comprehensions of the doctrine $P$. 
\begin{definition}
Let $P$ be a an elementary existential doctrine. 
We call the category $\regularcompdoctrine{P}:=\Ef{P_{c}}$  the  {\it the  regular completion}  of $P$.
\end{definition}
 Then, we recall the following result from \cite[Thm. 3.3]{TECH}:
\begin{theorem}\label{theorem maietti rosolini pasquali regular comp}\label{mainreg}
The assignment $P\mapsto \regularcompdoctrine{P}$ extends to a 2-functor
\[\freccia{\EED}{\regularcompdoctrine{-}}{\Reg}\]
which is left biadjoint to the inclusion of the 2-category $\Reg$ of regular categories in the 2-category $\EED$ of elementary and existential doctrines.
\end{theorem}
Notice that the regular completion of an elementary existential doctrine depends only on the objects of the base category and on the fibres. In particular we have the following lemma.
\begin{lemma}\label{lemma P ha fibre isomorfe a R}
Let $P$ and $R$ be two elementary existential doctrines. If $\freccia{P}{(F,f)}{R}$ is a morphism of elementary existential doctrine such that $\freccia{P(A)}{f_A}{R(FA)}$ is an isomorphism for every object $A$ of the base, then $\freccia{\regularcompdoctrine{P}}{\regularcompdoctrine{F,f}}{\regularcompdoctrine{R}}$ is  full and faithful functor. 

\end{lemma}
\begin{proof}
By hypothesis we have that $\freccia{P(A)}{f_A}{R(FA)}$ is an isomorphism, and then when we consider the comprehension completion we have again that each component $\freccia{P_c(A,\alpha)}{(f_c)_{(A,\alpha)}}{R_c(FA,f_A(\alpha))}$ an isomorphism. Since $(F,f)$ is a morphism of elementary and existential doctrines and $f$ is a natural transformation whose components are isomorphisms we have that the $\phi$ is an entire functional relation of $P_c(A,\alpha)$ if and only if $f_c(\phi)$ is an entire functional relation of $R_c(FA,f_A(\alpha))$. Therefore, by definition of regular completion, we have that the functor  $\freccia{\regularcompdoctrine{P}}{\regularcompdoctrine{F,f}}{\regularcompdoctrine{R}}$ is a full and faithful.

\end{proof}

The regular completion construction  $\regularcompdoctrine{P}$  of an elementary existential doctrine $P$ generalizes in the setting of doctrines  the regular completion  construction
$\regularcomp{\mD}$ 
 of a category $\mD$  with finite limits introduced in \cite{REC,SFEC}. In particular, 
such a regular completion  is an instance of  the $\regularcomp{-}$ completion. 

We briefly recall from \cite{REC} the construction of the category $\regularcomp{\mD}$ associate to a category $\mD$ with finite limits.

\begin{definition}
Let $\mD$ be a category with finite limits. The \bemph{regular completion} of $\mD$ is the category $\regularcomp{\mD}$ defined as follows:
\begin{itemize}
\item its objects are arrows $\frecciasopra{X}{g}{Y}$ of $\mD$;
\item an arrow \[\xymatrix{
X \ar[rr]^g &\ar[d]^{[m]}& Y\\
U \ar[rr]_f&& V
}\]
in $\regularcomp{\mD}$ is an equivalence class of arrows $\freccia{X}{m}{U}$ such that $fmg_1=fmg_2$ where $g_1$ and $g_2$ is a kernel pair of $g$. Two such arrows $m$ and $u$ are considered equivalent if $fn=fu$.
\end{itemize} 
\end{definition}


\begin{corollary}\label{corollario regular come  ex lex}
The regular  completion  $\regularcomp{\mD}$ of category with finite limits $\mD$ is equivalent to the regular completion $\regularcompdoctrine{\Psi_{\mD}}$ of the doctrine $\doctrine{\mD}{\Psi_{\mD}}$ of weak subobjects of $\mD$.
\end{corollary}
\begin{proof}
This follows from Theorem~\ref{mainreg} by assigning to a lex category $\mD$  the regular completion  $\regularcompdoctrine{\Psi_{\mD}}$ of its weak subobjects doctrine after observing that the embedding of $\mD$ into $\regularcompdoctrine{\Psi_{\mD}}$ is a finite limit functor.  Moreover, any  finite limit functor $L$ from $\mD$ into a regular category $\mR$ can be completed {\it uniquely} to an elementary existential doctrine
morphism $(L,l)$ from $\Psi_{\mD}$ to the subobjects doctrine of $\mR$ by defining $l(f)\equiv\ \exists_{L(f)}(\top)$, namely the image of $L(f)$ for any $[f]$ in
  the fibre of  $\Psi_D$, since
$f=\exists_f(\top)$ and $l$ must preserve existential quantifiers. \end{proof}

\begin{lemma}\label{lemma tutti gli oggeti della Reg(Gr) sono nella forma (A,a)---->(B,T)}
Let $\doctrine{\mD}{P}$ be a conjunctive doctrine such that the base category has finite limits. The category $\mG_P$ has finite limits, and every object of $\regularcomp{\mG_P}$ is isomorphic to one of the form $\frecciasopra{(A,\alpha)}{f}{(B,\top_B)}$.
\end{lemma}
\begin{proof}
The fact that $\mG_P$ has finite limits follows because the base category $\mD$ has finite limits and every fibre of $P$ is an inf-semilattice.
Moreover, if we consider an object $\frecciasopra{(A,\alpha)}{f}{(B,\beta)}$ of the category $\regularcomp{\mG_P}$, then it is direct to check that 
\[\xymatrix{
(A,\alpha) \ar[rr]^f &\ar[d]^{[\id_A]}& (B,\beta)\\
(A,\alpha) \ar[rr]_f&& (B,\top_B)
}\]
is an arrow of $\regularcomp{\mG_P}$ and that is an isomorphism.
\end{proof}
\begin{remark}\label{remark tutti gli oggetti della reg sono ((A,T),E(T)}
Let us consider the weak subobjects doctrine $\doctrine{\mG_P}{\Psi_{\mG_P}}$, where $P$ is a conjunctive doctrine. Then every object of $\regularcompdoctrine{\Psi_{\mG_P}}$ is of the form $((B,\beta),\Einv_f(\top))$, with $\freccia{(A,\alpha)}{f}{(B,\beta)}$, since $\Psi_{\mG_P}$ is the full existential completion of the trivial doctrine $\trdc_{\mG_P}$. Notice that it is direct to check that $((B,\beta),\Einv_f(\top))\cong ((B,\top),\Einv_f(\top))$ via the isomorphism $\Einv_{\angbr{f}{f}}(\top)$. This is the correspondent result of Lemma \ref{lemma tutti gli oggeti della Reg(Gr) sono nella forma (A,a)---->(B,T)}, since $\regularcomp{\mG_P}\equiv \regularcompdoctrine{\Psi_{\mG_P}}$.
\end{remark} 
From  what said in example \ref{theorem weak sub is a existential comp}, we conclude:
\begin{corollary}\label{corollary regular comp of weak sub 2}
The regular  completion  of  $\doctrine{\mD}{\Psi_{\mD}}$ of a category with finite limits is the regular completion of a full generalized existential completion.
\end{corollary}

Notice that for any regular theory $\theory$ according to Johnstone's definition \cite{SAE}~\footnote{Such a notion of regular logic is indeed the internal language
of existential elementary doctrines in the sense of \cite{MCBDTTCIPT}. In   \cite{MCBDTTCIPT} the internal language of regular categories is presented as a dependent type theory.}
%
 the construction of its syntactic category $\mathfrak{C}_{\theory}$  in the sense of \cite{SAE} is equivalent to the regular completion $\regularcompdoctrine{LT_{\theory}}$ of the syntactic elementary existential doctrine $\doctrine{\mC_{\theory}}{LT}$  associated to the theory $\theory$, defined in Example \ref{example doctrines}. Therefore it is natural to present and consider the notion of \emph{regular Morita-equivalent} for elementary existential doctrines defined as follows:
  
\begin{definition}\label{def reg. morita equivalent}
Two elementary existential doctrines $P$ and $P'$ are said \bemph{regular Morita-equivalent} if their regular completions are equivalent, i.e. when $\regularcompdoctrine{P}\equiv \regularcompdoctrine{P'}$.
\end{definition}
 
Hence, by using  Definition \ref{def reg. morita equivalent}  we can state:
\begin{theorem}\label{theorem nec. regular morita equiv}
Let $\doctrine{\mC}{P}$ be a elementary and existential doctrine. If $\regularcompdoctrine{P}\equiv\regularcomp{\mD} $, then $P$ is regular-Morita equivalent to a full existential completion whose base is $\mD$.
\end{theorem}
\begin{proof} It follows from  Corollaries \ref{corollary regular comp of weak sub 2} and \ref{corollario regular come  ex lex}. \end{proof}
Moreover, the notion of regular Morita-equivalence among weak subobjects doctrines coincides with their equivalence as doctrines:
\begin{theorem}\label{theorema unicita weak subobject doctrine}
Let $\doctrine{\mC}{\Psi_{\mC}}$ and $\doctrine{\mD}{\Psi_{\mD}}$ be two weak subobjects doctrines whose base categories $\mC$ and $\mD$ have finite limits. Then $\Psi_{\mC}$ and $\Psi_{\mD}$ are regular Morita-equivalent if and only if $\Psi_{\mC}\cong \Psi_{\mD}$.
\end{theorem}
\begin{proof}
It follows by the notion of regular Morita-equivalence,   Corollary \ref{corollario regular come  ex lex} and the characterization of regular completions in terms of regular projectives due to Carboni \cite[Lem. 5.1]{SFEC}, 
\end{proof}
The main purpose of this section it to show that, for a given elementary existential doctrine $P$, the condition of being regular-Morita equivalent to a full existential completion is not
only necessary but also sufficient for obtaining  $\regularcompdoctrine{P}$ equivalent to the regular completion of a finite limit category.

In particular, we are going to show that, if we consider a conjunctive doctrine $\doctrine{\mC}{P}$ where $\mC$ is a category with finite limits and we consider the \lclass\ of all the morphisms of $\mC$,
the regular completion of  the full existential completion $\fullcompex{P}$ of $P$ is equivalent to the  regular category associated to the Grothendieck category $\mG_{P}$ of to the doctrine $P$
$$\regularcompdoctrine{\fullcompex{P}}\equiv \regularcomp{\mG_{P}}.$$

\noindent
To this purpose we define the following canonical morphism:
\begin{definition}
Given a full existential doctrine $\doctrine{\mC}{P}$ on a category $\mC$ with finite limits and a fibred subdoctrine $P'$ of $P$, we define a morphism of full existential doctrines
$$\freccia{(\Psi_{\mG_{P'}})}{(N,n)}{P}$$ 
given by 
\begin{itemize}
\item $N(A,\alpha)=A$ for every object $(A,\alpha)$ of $\mG_{P'}$;
\item $N(h)=h$ for every arrow $\freccia{(B,\beta)}{h}{(A,\alpha)}$;
\item $n(g)=\exists_g(\gamma)$ for every element $\freccia{(C,\gamma)}{g}{(A,\alpha)}$ of  $(\Psi_{\mG_{P'}})((A,\alpha))$.
\end{itemize}
This induces a morphism 
$$\freccia{(\Psi_{\mG_{P'}})_c}{(N,n)_c}{P_c}$$ 
defined  as follows:

\begin{itemize}
\item $N_c((A,\alpha),f)=(A,\Einv_f(\beta))$, where $\freccia{(B,\beta)}{f}{(A,\alpha)}$;
\item $\freccia{(A,\exists_f(\beta))}{N_c(h)=h}{(C,\exists_l(\sigma))}$ for any  arrow $\freccia{((A,\alpha),f)}{h}{((C,\gamma),l)}$;
\item $n_c(g)=\Einv_g(\gamma)$ for any given  element $\freccia{(C,\gamma)}{g}{(A,\alpha)}$ of $(\Psi_{\mG_{P'}})_c((A,\alpha),f)$.
\end{itemize}
\end{definition}
We also recall that the action of the functor $\freccia{\regularcompdoctrine{\Psi_{\mG_{P}} }}{\regularcompdoctrine{N,n}}{\regularcompdoctrine{\fullcompex{P}}}$ on the objects of $\regularcompdoctrine{\Psi_{\mG_{P}} }$ is given essentially by the action of the functor $N_c$, i.e. we have that $\regularcompdoctrine{N,n}((A,\alpha),f)=N_c((A,\alpha),f)$, while the action of $\regularcompdoctrine{N,n}$ on morphisms is given by the action of $n_c$, i.e. $\regularcompdoctrine{N,n}(g)=n_c(g)$.

Combining the previous results, we obtain the following theorem.
\begin{theorem}\label{theorem car. regular comp via ex. comp.}
Let $\doctrine{\mC}{P}$ be a conjunctive doctrine, whose base category $\mC$ has finite limits. Then the regular functor $\freccia{\regularcompdoctrine{\Psi_{\mG_{P}} }}{\regularcompdoctrine{N,n}}{\regularcompdoctrine{\fullcompex{P}}}$ is an equivalence, so we have
\[\regularcompdoctrine{\fullcompex{P}} \equiv \regularcompdoctrine{\Psi_{\mG_P}}\equiv\regularcomp{\mG_{P}}.\]
\end{theorem}
\begin{proof}
By Theorem \ref{theorem caract. gen. existential completion} there exists a morphism $\freccia{\fullcompex{P}}{(\ovln{L},\ovln{l})}{\Psi_{\mG_P}}$ of full existential doctrines where $\freccia{\mC}{\ovln{L}}{\mG_P}$ is full and faithful and for every object $A$ of $\mC$, $\ovln{l}_A$ is an isomorphism between the fibres $\fullcompex{P}(A)\cong \Psi_{\mG_{P}}(A,\top)$. Therefore, by Lemma \ref{lemma P ha fibre isomorfe a R}, we have that $\freccia{\regularcompdoctrine{\fullcompex{P}} }{\regularcompdoctrine{\ovln{L},\ovln{l}}}{\regularcompdoctrine{\Psi_{\mG_P}}}$ is an inclusion of categories, i.e. the functor $\regularcompdoctrine{\ovln{L},\ovln{l}}$ is full and faithful. Now, since by Remark \ref{remark tutti gli oggetti della reg sono ((A,T),E(T)} every object of $\regularcompdoctrine{\Psi_{\mG_P}}$ is isomorphic to one of the form $((B,\top),\Einv_f(\top))=\regularcompdoctrine{\ovln{L},\ovln{l}}(B,\exists_f(\alpha))$ we can conclude that $$\regularcompdoctrine{\fullcompex{P}} \equiv \regularcompdoctrine{\Psi_{\mG_P}}\equiv\regularcomp{\mG_{P}}.$$
Finally, since comprehensions are full, we have that $\regularcompdoctrine{N,n}\regularcompdoctrine{\ovln{L},\ovln{l}}= \id$  and also $\regularcompdoctrine{\ovln{L},\ovln{l}}\regularcompdoctrine{N,n}\cong \id$.

\end{proof}
\begin{remark}
Observe that for any conjunctive doctrine $\doctrine{\mC}{P}$, whose base category $\mC$ has finite limits,
we have that $(N,n) (\ovln{L},\ovln{l})=\id$ by fullness of comprehensions but not in general that  $(\ovln{L},\ovln{l})(N,n) =\id$. The composition $(\ovln{L},\ovln{l})(N,n)$
 becomes an identity after applying
the regular completion construction for what observed in Remark~\ref{remark tutti gli oggetti della reg sono ((A,T),E(T)}.
\end{remark}


%

\begin{theorem}\label{tchar1}
Let $\doctrine{\mC}{P}$ be a full existential doctrine, whose base category $\mC$ has finite limits. 
Let $ \doctrine{\mC}{P'}$ be a conjunctive, fibred subdoctrine of $P$.
If the canonical arrow $\freccia{\regularcompdoctrine{\Psi_{\mG_{P'}} }}{\regularcompdoctrine{N,n}}{\regularcompdoctrine{P}}$ is an equivalence preserving the canonical embeddings of $P'$ into $P$ and $\Psi_{\mG_{P'}}$, namely  it forms an equivalence with an arrow
 $\freccia{\regularcompdoctrine{P}}{(H,h)} {\regularcompdoctrine{\Psi_{\mG_{P'}} }}$ making the diagram


\[\xymatrix@+2pc{
P'\ar[r] \ar[rd]& \regularcompdoctrine{P}\ar[d]^{(H,h)}\\
&\regularcompdoctrine{\Psi_{\mG_{P'}} }
}\]
commute, then 
$P$ is the full existential completion of $P'$, i.e. $P=\fullcompex{P'}$.
\end{theorem}
\begin{proof}
To prove the result we employ the second point of Theorem \ref{theorem caract. gen. existential completion}. Hence, we first show that $P(A)\cong \Psi_{\mG_{P'}}(A,\top)$ for every object $A$ of $\mC$. Therefore, we have that for every elementary and existential doctrine the following isomorphisms hold

\begin{equation}\label{eq PA= Sub Reg(A,T)}
P(A)\cong P_c(A,\top)\cong \Sub_{\regularcompdoctrine{P}}(A,\top)\end{equation}
and, in particular we have that 
\begin{equation}\label{eq Psi(A,T)=Sub Reg ((A,T),id)}
\Psi_{\mG_{P'}}(A,\top)\cong (\Psi_{\mG_{P'}})_c((A,\top),\id_{(A,\top)})\cong \Sub_{\regularcompdoctrine{\Psi_{\mG_{P'}} }}((A,\top),\id_{(A.\alpha)}).
\end{equation}
Now, since $\freccia{\regularcompdoctrine{\Psi_{\mG_{P'}} }}{\regularcompdoctrine{N,n}}{\regularcompdoctrine{P}}$ is an equivalence, and by definition of $N$, we have that 

\begin{equation}\label{eq sub reg P =sub reg psi}
\Sub_{\regularcompdoctrine{\Psi_{\mG_{P'}} }}((A,\top),\id_{(A.\top)})\cong \Sub_{\regularcompdoctrine{P}}(A,\top)
\end{equation}
we  conclude that 
\begin{equation}\label{equation P(A) cong Psi(A,top)}
P(A)\cong \Psi_{\mG_{P'}}(A,\top).
\end{equation}

Observe that all the previous mentioned isomorphisms preserve elements of $P'$ because they are components of canonical morphisms and  $\regularcompdoctrine{N,n}$ preserves the elements of $P'$ by hypothesis. Hence we conclude that the isomorphism \eqref{equation P(A) cong Psi(A,top)} sends every element $\alpha \in P'(A)$ to the comprehension map $\frecciasopra{(A,\alpha)}{\id_A}{(A,\top)}$.

The naturality of this family of isomorphisms follows because the  canonical injections of a weak subobjects doctrine into the completions involved, is full and faithful on the base category. In particular the crucial point is  that the functor between the base categories of the canonical injection
$$\frecciasopra{\Psi_{\mG_{P'}}}{}{\Sub_{\regularcompdoctrine{\Psi_{\mG_{P'}}}}}$$ 
is full and faithful. This holds because the weak subobjects doctrines has comprehensive diagonals.  
Therefore, we can apply the second point of Theorem \ref{theorem caract. gen. existential completion} and conclude that $P$ is the full existential completion of $P'$.

\end{proof}
\begin{remark}\label{example regula comp category with finite limits}
A first obvious example of application of theorem \ref{theorem car. regular comp  via ex. comp.}  is  the construction of the reg/lex completion 
of a finite limit category $\mD$ itself. Indeed, after recalling that $\regularcomp{\mD}\equiv\regularcompdoctrine{\Psi_{\mD}}$   from corollary~\ref{corollario regular come  ex lex}
and that  the weak subobjects doctrine is a full existential completions of the  trivial doctrine $\doctrine{\mD}{\trdc}$  from \ref{theorem weak sub is a existential comp},
we can obviously apply Theorem \ref{theorem car. regular comp via ex. comp.} to  $\Psi_{\mD}$ 
by getting
$\regularcomp{\mD}\equiv \regularcompdoctrine{\Psi_{\mD}}\equiv \regularcompdoctrine{\fullcompex{\trdc}}\equiv \regularcomp{\mG_{\trdc}}$.
This chain of equivalences does not yield more information since by Theorem \ref{theorema unicita weak subobject doctrine}  we conclude that $\mD\equiv \mG_{\trdc}$ as one can immediately check independently.
\end{remark}

Finally, we can apply Theorem \ref{theorem car. regular comp via ex. comp.}, together with Theorem \ref{theorem nec. regular morita equiv} obtaining our main result: 

\begin{theorem}\label{theorema caratterizzazione generale regual comp}
Let $\doctrine{\mC}{P}$ be an full existential doctrine, whose base category has finite limits. Then the following are equivalent:

\begin{enumerate}
\item the regular completion $\regularcompdoctrine{P}$ is equivalent to regular completion $\regularcomp{\mD}$ of a lex category $\mD$;
\item $P$ is regular Morita-equivalent to the weak-subojects doctrine $\doctrine{\mD}{\Psi_{\mD}}$ on  a lex category $\mD$;
\item $P$ is regular Morita-equivalent to a full existential completion $P'$.
\end{enumerate}
\end{theorem}
\begin{proof}
$(1\Rightarrow 2)$ It follows by Corollary \ref{corollario regular come  ex lex} and by definition of regular Morita-equivalence, Definition \ref{def reg. morita equivalent}.

$(2\Rightarrow 3)$ It follows again by definition of regular Morita-equivalence and by the fact that every weak subobjects doctrine is a full existential completion, see Theorem \ref{example lambda-weak-sub} and Example \ref{theorem weak sub is a existential comp}.

$(3\Rightarrow 1)$ It follows by Theorem \ref{theorem car. regular comp via ex. comp.}.
\end{proof}
Employing the same arguments used in the previous proofs, together with Theorem \ref{theorem char pure existential compl of elementary doct}, we can prove the analogous results for the case of the pure existential completion of elementary doctrines. 
\begin{theorem}\label{theorem car. regular comp via pure ex. comp.}
Let $\doctrine{\mC}{P}$ be a pure existential, elementary doctrine.
 Then we have the equivalence
\[\regularcompdoctrine{\purecompex{P}} \equiv \regularcompdoctrine{\Psi_{\Pred{P}}}\equiv\regularcomp{\Pred{P}}.\]
\end{theorem} 
\begin{proof}
Notice that in this case we have to employ Theorem \ref{theorem char pure existential compl of elementary doct} to obtain the equivalence between the fibres $\purecompex{P}(A)\cong \Psi_{\Pred{P}}(A,\top)$. After this observation, we can use exactly the same argument used in the proof of Theorem \ref{theorem car. regular comp via ex. comp.} to conclude.
\end{proof}
Similarly we have the following theorem.
\begin{theorem}\label{purextchar1}
Let $\doctrine{\mC}{P}$ be an pure existential, elementary  doctrine, and 
let $ \doctrine{\mC}{P'}$ be an elementary fibred subdoctrine of $P$.
If the canonical morphism given by the pair $\freccia{ \regularcompdoctrine{\Psi_{\Pred{P'}} }}{\regularcompdoctrine{N,n}}{\regularcompdoctrine{P}}$ is an equivalence then 
$P$ is a pure  existential completion of $P'$.
\end{theorem} 
\begin{proof}
Notice that by Theorem \ref{theorem char pure existential compl of elementary doct} we obtain the equivalence between the fibres $\purecompex{P}(A)\cong \Psi_{\Pred{P}}(A,\top)$.
Then, this result can be proved exactly using the same argument of Theorem  \ref{tchar1}. 
\end{proof}
\begin{example}
As application of Theorem \ref{theorem car. regular comp via pure ex. comp.} together with Theorem \ref{theorem P iso P^ex if and only if} we obtain that if an elementary existential doctrine $\doctrine{\mC}{P}$ is equipped with Hilbert's $\epsilon$-operators, then we have the equivalence $\regularcompdoctrine{P}\equiv \regularcomp{\Pred{P}}$. Therefore, as a particular case of our result, we obtain the result proved in \cite[Thm. 6.2(ii)]{TECH}.
\end{example}

\subsection{The example of the regular category of assemblies of a partial combinatory algebra}
From \cite{RosoliniRobinson90,HYLAND1982165,SFEC}  we know that the category of assemblies is the reg/lex completion of  the category of partitioned assemblies.
We show here that  such a result is a consequence of Theorem \ref{theorem car. regular comp via ex. comp.}  and the fact that  the realizability hyperdoctrines on a pca  \cite{TT,TTT,van_Oosten_realizability} are full existential completions as proved in theorem~\ref{theorem realizability hyper. is gen ex comp}.

Recall that an \bemph{assembly} on a pca $\pca{A}$ is a pair $(X,E)$ where $X$ is a set and $E$ is a function $\freccia{X}{E}{\pwsetnonempty{\pca{A}}}$ assigning to every element $x\in X$ a non-empty subset $E(x)\subseteq \pca{A}$.

A \bemph{morphism of assemblies}  $\freccia{(X,E)}{f}{(Y,F)}$ is a function $\freccia{X}{F}{Y}$ with the property that there exists an element $a\in \pca{A}$ such that for every $x\in X$, for every $b\in E(x)$, $a\cdot b\downarrow$ and $a\cdot b\in F(f(x))$. This element $a$ is said to \bemph{track} the function $f$. Assemblies and morphisms of assemblies form a category denoted by $\assembly{\pca{A}}$.

An assembly $(X,E)$ is said \bemph{partitioned} if $E$ is single-valued, i.e. $E$ is a function from $X$ to $\pca{A}$. The full subcategory of $\assembly{\pca{A}}$ on partitioned assemblies is written $\parassembly{\pca{A}}$. 

We refer to \cite{van_Oosten_realizability,RosoliniRobinson90} for an extensive analysis of these categories.

First of all observe that:
\begin{theorem}\label{theorem assemnlies and regularization realiz. tripos}
The category of assemblies $\assembly{\pca{A}}$  on a pca $\pca{A}$ is equivalent to the regular completion $\regularcompdoctrine{\mP}$ of the realizability doctrine $\doctrine{\set}{\mP}$ associated to $\pca{A}$.
\end{theorem}
\begin{proof}
The equivalence   $\assembly{\pca{A}}\equiv \regularcompdoctrine{\mP}$ can easily established after observing the  following two facts whose proofs are straightforward.
First, given an object $(A,\alpha)$ of $\regularcompdoctrine{\mP}$, then  if $A$ is not empty, then $(A,\alpha)$ is isomorphic to an object $(A',\alpha')$ such that $\alpha'(a)\neq \emptyset$ for every $a\in A'$, otherwise it is isomorphic to $(\emptyset,!_{\pwset{\pca{A}}})$, where $\freccia{0}{!_{\pwset{\pca{A}}}}{\pwset{\pca{A}}}$ is the empty function.

Moreover, if $\freccia{(A,\alpha)}{R}{(B,\beta)}$ is an arrow of $\regularcompdoctrine{\mP}$ where $\alpha$ and $\beta$ are always non-empty, then there exists a unique function $\freccia{A}{f}{B}$ such that 
\[R(a,b)= \begin{cases}
\alpha (a) & \text{if } b=f(a)\\
\emptyset & \text{otherwise.}
\end{cases}\]
Second, notice   that $f$ has the following property: there exists an element $n\in \pca{A}$ such that for every $a\in A$, for every $x\in \alpha(a)$, $n\cdot x\downarrow$ and $n\cdot x\in  \beta (f(a))$.
\end{proof}

It is immediate to observe that:
\begin{lemma}\label{par}
The Grothendieck category of $\mP^{\singleton}$ defined in section~\ref{section realizability hyperdoct} is  equivalent to that of partitioned assemblies 
$$\mG_{\mP^{\singleton}}\equiv \parassembly{\pca{A}}.$$
\end{lemma}

Hence, from these facts and our previous results we conclude in an alternative way Robinson and Rosolini's result:
\begin{corollary}
The category of assemblies
 is equivalent
to reg/lex completion of the category of partitioned assemblies.
\end{corollary}
\begin{proof}
From Theorem \ref{theorem realizability hyper. is gen ex comp} we know that the realizability hyperdoctrine $\doctrine{\set}{\mP}$ associated to a  partial combinatory algebra $\pca{A}$ is the full existential completion of the primary doctrine $\doctrine{\set}{\mP^{\singleton}}$ whose elements of the fibre $\mP^{\singleton}(X)$ are only singleton predicates of $\mP(X)$.  Therefore, by Theorem \ref{theorem car. regular comp via ex. comp.}, we have that 
\[\regularcompdoctrine{\mP}\equiv \regularcomp{\mG_{\mP^{\singleton}}}.\] 
and by theorem~\ref{theorem assemnlies and regularization realiz. tripos} and by lemma~\ref{par}  we conclude
$$\assembly{\pca{A}}\  \equiv\ \regularcomp{\parassembly{\pca{A}}}$$
\end{proof}

%

\begin{remark}
Observe also that in \cite{EQCCP}  the category of assemblies can be obtained also as a different kind of completion which instead is  not an instance of a full existential completion because the generating doctrine does not satisfy the choice principles (RC) in Remark \ref{RC}.
\end{remark}

\section{Exact completions via generalized existential completion}

The tripos-to-topos construction orginally  introduced in \cite{TTT,TT} is a fundamental construction of categorical logic that relates the notion of tripos with that of topos. 
This construction is an instance of the exact completion $\tripostotopos{P}$ of an elementary existential doctrine $\doctrine{\mC}{P}$  studied  in \cite{UEC}.

Given an elementary existential doctrine $\doctrine{\mC}{P}$, the category $\tripostotopos{P}$ consists of:
 
\medskip
\noindent 
\textbf{objects:} pairs $(A,\rho)$ such that $\rho \in P(A\times A)$  satisfies

\begin{itemize}

\item \emph{symmetry:} $\rho\leq P_{\angbr{\pr_2}{\pr_1}}(\rho)$;
\item \emph{transitivity:} $P_{\angbr{\pr_1}{\pr_2}}(\rho)\wedge P_{\angbr{\pr_2}{\pr_3}}(\rho)\leq P_{\angbr{\pr_1}{\pr_3}}(\rho)$, where $\pr_i$ are the projections from $A\times A\times A$;
\end{itemize} 
\textbf{arrows:} $\freccia{(A,\rho)}{\phi}{(B,\sigma)}$ are objects $\phi\in P(A\times B)$ such that
\begin{enumerate}[label=(\roman*)]
\item $\phi\leq P_{\angbr{\pr_1}{\pr_1}}(\rho)\wedge P_{\angbr{\pr_2}{\pr_2}}(\sigma)$;
\item $P_{\angbr{\pr_1}{\pr_2}}(\rho)\wedge P_{\angbr{\pr_1}{\pr_3}}(\phi)\leq P_{\angbr{\pr_2}{\pr_3}}(\phi)$ where $\pr_i$ are projections from $A\times A\times B$; 
\item $P_{\angbr{\pr_2}{\pr_3}}(\sigma)\wedge P_{\angbr{\pr_1}{\pr_2}}(\phi)\leq P_{\angbr{\pr_1}{\pr_3}}(\phi)$ where $\pr_i$ are projections from $A\times B\times B$; 
\item $P_{\angbr{\pr_1}{\pr_2}}(\phi)\wedge P_{\angbr{\pr_1}{\pr_3}}(\phi)\leq P_{\angbr{\pr_2}{\pr_3}}(\sigma)$ where $\pr_i$ are projections from $A\times B\times B$; 
\item $P_{\angbr{\pr_1}{\pr_1}}(\rho)\leq \Einv_{\pr_1}(\phi)$ where $\pr_i$ are projections from $A\times B$.
\end{enumerate}
\medskip
\noindent

The construction $\tripostotopos{P}$ is called the \bemph{exact completion} of the elementary existential doctrine $P$. In particular we recall the following result from \cite[Cor. 3.4]{UEC}:
\begin{theorem}\label{theorem maietti rosolini pasquali exact comp}
The assignment $P\mapsto \tripostotopos{P}$ extends to a 2-functor 
\[\frecciasopra{\EED}{}{\excat}\]
from the 2-category $\EED$ of elementary, existential doctrines to the 2-category $\excat$ of exact categories, and it is left adjoint to the functor sending an exact category $\mC$ to doctrine $\Sub_{\mC}$ of its subobjects.
\end{theorem}
As done for regular completions,  we introduce the following definition of \bemph{exact Morita-equivalence} between elementary existential doctrines:
\begin{definition}\label{def exact. morita equivalent}
Two elementary existential doctrines $P$ and $P'$ are said \bemph{exact Morita-equivalent} if their exact completions are equivalent, i.e. when $\tripostotopos{P}\equiv \tripostotopos{P'}$. 
\end{definition}

\begin{remark}\label{exreg}
Recall from \cite[Ex. 3.2]{UEC} and \cite[Ex. 4.8]{TECH} that if we consider a regular category $\mA$, the exact completion $\tripostotopos{\Sub_{\mA}}$ of the subobjects doctrine $\doctrine{\mA}{\Sub_{\mA}}$  coincides with the exact completion $\exactcomp{\mA}$, i.e. with the exact completion $\exactcomp{\mA}$ of the regular category $\mA$ introduced by Freyd in \cite{CA}.  Therefore, combining Proposition \ref{proposition caratterizzazione RUC} with \cite[Ex. 4.8]{TECH} we can provide a description of the exact completion of a regular category in terms of existential completions.

In particular, if a doctrine $\doctrine{\mC}{P}$  is the $\mM$-existential completion of the trivial doctrine $\doctrine{\mC}{\trdc}$, where $\mM$ is a class of morphisms of $\mC$ such that 
\begin{itemize}
\item there exists a class $\mathcal{E}$ of morphisms of $\mC$ such that $(\mathcal{E}, \mM)$ is a proper, stable, factorization system on $\mC$;
\item for every arrow $\freccia{A\times B}{\pr_A}{A}$ of $\mC$, if $\pr_Af$ is a monomorphism and $f\in \mM$ then $\pr_Af\in \mM$;
\end{itemize}
then $\tripostotopos{P}\equiv\exactcomp{\mC}$.
\end{remark}

Then, from the results shown in the previous sections we also obtain:
\begin{corollary}
 The exact completion  ${\cal C}_{\mathrm{ex/reg}}$ of a regular category $\cal C$ is the 
 exact completion of a full existential completion.
 \end{corollary}
 \begin{proof}
 As observed in remark ~\ref{exreg} the ex/reg completion ${\cal C}_{\mathrm{ex/reg}}$  of  a regular category $\mC$
  coincides  with  the exact completion of  the elementary existential doctrine of the subobjects doctrine $\Sub_{\cal C}$ and this is the full existential completion with respect to all the morphisms
  of its base category as shown in Proposition \ref{mainsub}. 
\end{proof}

%

Thanks to the result presented in  \cite[Sec. 2]{UEC} and  \ref{theorem maietti rosolini pasquali exact comp} we have that:
\begin{theorem}\label{teorema T_P=reg(P)_exreg}
 For any  elementary existential doctrine $\doctrine{\mC}{P}$ the following equivalence holds
\[\tripostotopos{P}\equiv \exactcomp{\regularcompdoctrine{P}}.\]
\end{theorem}

This is indeed a generalization of the well known result from \cite{SFEC,REC,FECLEO} that the exact completion $\exactcomplex{\mD}$ of a category $\mD$ with finite limits is equivalent to the construction $\exactcomp{\regularcomp{\mD}}$.  

Furthermore, as first observed in \cite{UEC} and analogously to what proved for the reg/lex completion in Corollary \ref{corollario regular come  ex lex}, the following holds:
\begin{corollary}\label{corollary exact comp of weak sub 1}  
The exact completion  $\exactcomplex{\mD}$ of a category $\mD$  with finite limits happens to be equivalent to the exact completion $\tripostotopos{\Psi_{\mD}}$ of the doctrine $\doctrine{\mD}{\Psi_{\mD}}$ of weak subobjects of $\mD$.
\end{corollary}
Now, by combining these results we obtain the following
analogous result to that  of Theorem \ref{theorem nec. regular morita equiv}:
\begin{theorem}\label{theorem nec. exact morita equiv}
Let $\doctrine{\mC}{P}$ be a elementary and existential doctrine. If $\tripostotopos{P}\equiv\exactcomplex{\mD} $, then $P$ is exact-Morita equivalent to a full existential completion whose base is $\mD$.
\end{theorem}  
\begin{proof}
It follows from Example \ref{theorem weak sub is a existential comp} and Corollary \ref{corollary exact comp of weak sub 1}.
\end{proof}
As in  the case of  the regular completion in Theorem \ref{theorema unicita weak subobject doctrine}, exact-Morita equivalence between weak subobjects doctrines
coincides with their equivalence:
\begin{theorem}\label{theorema unicita weak subobject doctrine 2}
Let $\doctrine{\mC}{\Psi_{\mC}}$ and $\doctrine{\mD}{\Psi_{\mD}}$ be two weak subobjects doctrines whose base categories $\mC$ and $\mD$ have finite limits. Then $\Psi_{\mC}$ and $\Psi_{\mD}$ are exact Morita-equivalent if and only if $\Psi_{\mC}\cong \Psi_{\mD}$.
\end{theorem}
\begin{proof}
It follows by the notion of exact Morita-equivalence, Corollary \ref{corollary exact comp of weak sub 1}  and the characterization of exact completions in terms of regular projectives due to Carboni \cite[Lem. 2.1]{SFEC}.
\end{proof}
Furthermore, from these facts and Theorem \ref{theorem car. regular comp via ex. comp.}, we directly conclude the following:

\begin{theorem}\label{theorem exact comp generalized existential comp}
Let $\doctrine{\mC}{P}$ be an elementary and existential doctrine, whose base category $\mC$ has finite limits. 
Let $ \doctrine{\mC}{P'}$ be a conjunctive, fibred subdoctrine of $P$.
Then the arrow $\freccia{\tripostotopos{\Psi_{\mG_{P'}} }}{\mathbf{Ex}{(N,n)}}{\tripostotopos{P}}$ is an equivalence, preserving the canonical embeddings of $P'$ into $P$ and $\Psi_{\mG_{P'}}$, if and only if $P$ is the full existential completion of $P'$.
\end{theorem}
\begin{proof}
It follows from Theorem  \ref{teorema T_P=reg(P)_exreg}, Theorem \ref{theorem car. regular comp via ex. comp.} and Theorem \ref{tchar1}, together with the observation that the \emph{only if} part has a straightforward proof following the same line and technique used in Theorem \ref{tchar1} for the case of the regular completion. 
\end{proof} 

\begin{corollary}\label{theorem exact comp generalized existential comp}
Let $\doctrine{\mC}{P}$ be a \conjdoctrine~whose base category $\mC$ has finite limits. Then we have the equivalence
\[\tripostotopos{\fullcompex{P}}\equiv \exactcomplex{\mG_{P}}.\]
\end{corollary}

\begin{theorem}\label{main}

Let $\doctrine{\mC}{P}$ be an full existential doctrine, whose base category has finite limits. Then the following are equivalent:

\begin{enumerate}
\item the exact completion $\tripostotopos{P}$ is equivalent to regular completion $\exactcomplex{\mD}$ of a lex category $\mD$;
\item $P$ is exact Morita-equivalent to the weak-subojects doctrine $\doctrine{\mD}{\Psi_{\mD}}$ on  a lex category $\mD$;
\item $P$ is exact Morita-equivalent to a full existential completion $P'$.
\end{enumerate}
\end{theorem}
\begin{proof}
$(1\Rightarrow 2)$ It follows by Corollary \ref{corollary exact comp of weak sub 1} and by definition of exact Morita-equivalence.

$(2\Rightarrow 3)$ It follows again by definition of exact Morita-equivalence and by the fact that every weak subobjects doctrine is a full existential completion, see Theorem \ref{example lambda-weak-sub} and Example \ref{theorem weak sub is a existential comp}.

$(3\Rightarrow 1)$ It follows by Corollary \ref{theorem exact comp generalized existential comp}.
\end{proof}
Furthermore,  we obtain analogous results for  the notion of pure existential completion of an elementary doctrine.
 \begin{theorem}\label{theorem exact comp pure existential comp}
Let $\doctrine{\mC}{P}$ be an pure existential, elementary  doctrine, and 
let $ \doctrine{\mC}{P'}$ be an elementary fibred subdoctrine of $P$.
Then,  the canonical morphism given by the pair $\freccia{ \tripostotopos{\Psi_{\Pred{P'}} }}{\mathbf{Ex}(N,n)}{\tripostotopos{P}}$ is an equivalence, preserving the canonical embeddings of $P'$ into $P$ and $\Psi_{\Pred{P'}}$, if and only if 
$P$ is a pure  existential completion of $P'$.
\end{theorem} 
\begin{proof}
It follows from Theorem \ref{teorema T_P=reg(P)_exreg}, Theorem \ref{theorem car. regular comp via pure ex. comp.} and Theorem \ref{purextchar1}.
\end{proof}
\begin{corollary}
If $\doctrine{\mC}{P}$ is an elementary doctrine then we have the equivalence
\[\tripostotopos{\purecompex{P}}\equiv \exactcomplex{\Pred{P}}.\]
\end{corollary}
As application of the previous result we obtain as corollary the result regarding doctrines equipped with $\epsilon$-operators and exact completion presented in \cite[Thm. 6.2(iii)]{TECH}.
\begin{corollary}
An elementary existential doctrine $\doctrine{\mC}{P}$ is equipped with Hilbert's $\epsilon$-operators if and only if the arrow $\freccia{ \tripostotopos{\Psi_{\Pred{P}} }}{\mathbf{Ex}(N,n)}{\tripostotopos{P}}$ is an equivalence.
\end{corollary}
\begin{proof}
By Theorem \ref{theorem P iso P^ex if and only if} we have that a doctrine is equipped with Hilbert's Theorem $\epsilon$-operators if and only if it is the pure existential completion of itself. Then, applying Theorem \ref{theorem exact comp pure existential comp} we can conclude.
\end{proof}
\begin{remark}
A first obvious  example of application of theorem \ref{theorem exact comp generalized existential comp}  is  the construction of the ex/lex completion of
 a finite limit category $\mD$ itsefl. Indeed, after recalling that $\exactcomp{\mD}\equiv\tripostotopos{\Psi_{\mD}}$   from corollary~\ref{corollario regular come  ex lex}
and that  the weak subobjects doctrine is a full existential completions of the  trivial doctrine $\doctrine{\mD}{\trdc}$  from \ref{theorem weak sub is a existential comp},
we can obviously apply Theorem \ref{theorem exact comp generalized existential comp}to  $\Psi_{\mD}$ 
by getting
$\exactcomplex{\mD}\equiv \tripostotopos{\Psi_{\mD}}\equiv \tripostotopos{\fullcompex{\trdc}}\equiv \exactcomplex{\mG_{\trdc}}$.
This chain of equivalences  does not yield more information, since by Theorem \ref{theorema unicita weak subobject doctrine 2}  we conclude that $\mD\equiv \mG_{\trdc}$ as one can immediately check independently.

\end{remark}

\begin{example}\label{Example realizability topos as exact com}
A relevant exact category having the presentation of Theorem \ref{theorem exact comp generalized existential comp} is provided by the realizability topos $\Eff (\pca{A})$ associated to a pca $\pca{A}$ \cite{TT,RosoliniRobinson90,HYLAND1982165}. In particular, given a realizability doctrine $\doctrine{\Set}{\mP}$ associated to a pca $\pca{A}$, we have
that combining Theorem \ref{theorem assemnlies and regularization realiz. tripos}, Lemma \ref{par}, and  Theorem \ref{theorem exact comp generalized existential comp} we can conclude that
\[\tripostotopos{\mP}=\tripostotopos{\fullcompex{\mP^{\singleton}}}\equiv \exactcomplex{\mG_{\mP^{\singleton}}}\equiv \exactcomplex{\parassembly{\pca{A}}}\equiv \Eff (\pca{A}).\]
 
\end{example}
\begin{example}
Notice that the result presented in Theorem \ref{theorem exact comp generalized existential comp} provides also a useful instrument to study subcategories of $\mathpzc{Eff}$ which are instances of $\exactcomplex{-}$ completion.
In particular, let us consider the doctrine $\doctrine{\Rec}{\trdc}$, where $\Rec$ denotes the category (with finite limits) whose objects are subsets of $\mathbb{N}$, and whose arrows are total, recursive functions. Following the notation used in Theorem \ref{theorem weak sub is a existential comp}, the functor $\trdc$ denotes the trivial doctrine, i.e. $\trdc (A)=\{\top\}$ for every object $A$ of $\Rec$. It is easy to see that the doctrine $\doctrine{\Rec}{\trdc}$ is a sub-doctrine of the doctrine $\doctrine{\Set}{\mP^{\singleton}}$ of singletons of the realizability tripos, and then we have that $\exactcomplex{\Rec}$ is a sub-category of $\mathpzc{Eff}$.
\end{example}

\begin{example}\label{example topos of sheaves over supercomp locale}
As observed in \cite[Ex. 2.15]{TT}, the topos $\tripostotopos{\loc^{(-)}}$ obtained by the localic doctrine associated to a local $\loc$ has to be equivalent to the category $\sheaves{\loc}$ of canonical sheaves over $\loc$. Now, given a inf-semilattice $\supercomp$, we have a canonical way to construct a topos from $\supercomp$, which is $\sheaves{D(\supercomp)}$. By Theorem \ref{theorem caract localic doct}  the full existential completion $\fullcompex{\supercomp^{(-)}}$ is exactly the localic doctrine $D(\supercomp)^{(-)}$
and hence by applying the tripos-to-topos construction to both from Theorem \ref{theorem exact comp generalized existential comp}  we conclude

\[\sheaves{D(\supercomp)}\equiv \tripostotopos{D(\supercomp)^{(-)}} \equiv \tripostotopos{\fullcompex{\supercomp^{(-)
}}} \equiv \exactcomplex{\mG_{\supercomp^{(-)}} }\]
 Hence, for every supercoherent local $\loc$, we have
\[\tripostotopos{\loc}\equiv \sheaves{\loc}\equiv \exactcomplex{\mG_{\supercomp^{(-)}}}\]
where $\supercomp$ is the inf-semilattice of the supercompact elements of $\loc$.

\end{example}
Recall from \cite{Menni2003} that given a local $\loc$ the category of presheaves  $\presheaves{\loc}$ over $\loc$ is defined as $\exactcomplex{\loc_+}$, where $\loc_+$ is the coproduct completion of $\loc$. Employing our main results we can prove the following theorem, showing that every category of presheaves on a given locale is equivalent to the category of sheaves on the locale provided by the free completion which adds arbitrary sups to a locale.
\begin{theorem}
Let $\loc$ be a locale. Then we have the following equivalence
\[\sheaves{D(\loc)}\equiv \presheaves{\loc}\]
where $D(\loc)$ is the supercoherent locale constructed from $\loc$.
\end{theorem}
\begin{proof}
Observe that the category called $\loc_+$ is exactly the Grothendieck category $\mG_{\loc^{(-)}}$, where $\doctrine{\Set}{\loc^{(-)}}$ is the localic tripos. Therefore, by Theorem \ref{theorem exact comp generalized existential comp} and Theorem \ref{theorem caract localic doct}, we can conclude that 
\[\sheaves{D(\loc)}\equiv \exactcomplex{\mG_{\loc^{(-)}}}\equiv \presheaves{\loc}\]
where $D(\loc)$ is the supercoherent locale constructed from $\loc$.
\end{proof}

\section{Conclusions}
We have employed the tool of generalized existential completions to study properties of regular and exact completions.
We conclude providing  pictures summarizing our principal results and examples.

Let $\doctrine{\mC}{P}$ be a tripos, which is the full existential completion of a fibred subdoctrine $P'$, and let $\doctrine{\mG_{P'}}{\Psi_{\mG_{P'}}}$ be the weak subobjects doctrine on the Grothendieck category $\mG_{P'}$.  Then we have the following morphisms of doctrines and categories
\[\xymatrix@+1pc{
P\cong\fullcompex{P'}\ar[d] \ar[r]& P_c \ar[r] \ar@/_1pc/[d]& \regularcompdoctrine{P}\ar[r] \ar@3{-}[d]&\tripostotopos{P}\ar@3{-}[d]\\
\Psi_{\mG_{P'}}\ar[r]& (\Psi_{\mG_{P'}})_c \ar[r]\ar@/_1pc/[u] &\regularcomp{\mG_{P'}}\ar[r] & \exactcomplex{\mG_{P'}}.
}\]
The previous diagram can be then specialized for realizability triposes  and supercoherent localic triposes as follows.

Let $\pca{A}$ be a pca and let us consider the realizability tripos associated to $\pca{A}$,  which we know it is the full existenttial completion $\fullcompex{\mP^{\singleton}}$ of the
conjuncitve doctrine of singletons. Then let us consider its realizability topos $ \Eff(\pca{A})$
and also the weak subobjects doctrine $\doctrine{\parassembly{\pca{A}}}{\Psi_{\parassembly{\pca{A}}}}$ on the category of partitioned assemblies over $\pca{A}$. In this case, we obtain the following diagram
\[\xymatrix@+1pc{
\mP\cong\fullcompex{\mP^{\singleton}}\ar[d] \ar[r]& \mP_c \ar[r] \ar@/_1pc/[d]& \regularcompdoctrine{\mP}\ar[r] \ar@3{-}[d]&\Eff(\pca{A}) \ar@3{-}[d]\\
\Psi_{\parassembly{\pca{A}}}\ar[r]& (\Psi_{\parassembly{\pca{A}}})_c \ar[r]\ar@/_1pc/[u] &\regularcomp{\parassembly{\pca{A}}}\equiv \assembly{\pca{A}} \ar[r] & \exactcomplex{\parassembly{\pca{A}}}.
}\]

For any  supercoherent locale $\loc$, where $\mN$ the inf-semilattice of its supercompact elements, and the localic tripos $\doctrine{\Set}{\loc^{(-)}}$ associated to $\loc$
we obtain:

\[\xymatrix@+1pc{
\loc^{(-)}\cong\fullcompex{\mN^{(-)}}\ar[d] \ar[r]& \loc^{(-)}_c \ar[r] \ar@/_1pc/[d]& \regularcompdoctrine{\loc^{(-)}}\ar[r] \ar@3{-}[d]&\tripostotopos{\loc^{(-)}}\equiv \sheaves{\loc}\ar@3{-}[d]\\
\Psi_{\mG_{\supercomp^{(-)}}}\ar[r]& (\Psi_{\mG_{\supercomp^{(-)}}})_c \ar[r]\ar@/_1pc/[u] &\regularcomp{\mG_{\supercomp^{(-)}}}\ar[r] & \exactcomplex{\mG_{\supercomp^{(-)}} }
}\]

\section{Future work}


A preliminary version of the characterization of generalized existential completion presented here  has already been fruitfully employed in recent works  \cite{hofstra2011,TSPGF,trottapasquale2020,dialecticalogicalprinciples} to  give a categorical version of G\"odel's dialectica interpretation \cite{goedel1986} in terms of quantifier-completions.
As future work, we intend to further apply  the tools and results developed here to the categorification of  G{\"o}del  dialectica interpretation and study their connections with dialectica triposes
in \cite{Biering_dialecticainterpretations} and modified realizability triposes \cite{Hofstra06,VANOOSTEN1997273}.




\section*{Acknowledgements}

We acknowledge fruitful discussions with Fabio Pasquali and Pino Rosolini on topics presented in this paper.

We also thanks Valeria de Paiva, Martin Hyland and Paul Taylor for useful discussions and comments regarding  applications of our results to G{\"o}del dialectica interpretation.

\bibliography{biblio_tesi_PHD}  

\end{document}

%% file: macros.tex
\newcommand{\alg}[1]{\operatorname{#1-\mathbf{Alg}}}

\newcommand{\bemph}[1]{\textbf{\emph{#1}}}
\newcommand{\theory}[1]{\mathcal{#1}}

\newcommand{\eqc}[1]{[#1]}

\newcommand{\ltil}[1]{\widetilde{#1}}

\newcommand{\ovln}[1]{\overline{#1}}

\newcommand{\Ef}[1]{\mathpzc{Ef}_{#1}}                
 
\newcommand{\tripostotopos}[1]{\mathcal{T}_{#1}}
\newcommand{\compex}[1]{{#1}^{\existential}}
\newcommand{\gencompex}[2]{\mathsf{Ex}^{#2}(#1)}
\newcommand{\purecompex}[1]{\mathsf{pEx}(#1)}
\newcommand{\fullcompex}[1]{\mathsf{fEx}(#1)}

\newcommand{\angbr}[2]{\langle #1,#2 \rangle} 
\newcommand{\comprl}{\ensuremath{ \{ \hspace{-0.6ex} |}} 
\newcommand{\comprr}{  \ensuremath{ | \hspace{-0.6ex}\}}} 
\newcommand{\comp}[1]{\comprl #1 \comprr}
\newcommand{\freccia}[3]{\xymatrix{#2 \colon #1  \ar[r] &  #3}}

\newcommand{\frecciainj}[3]{\xymatrix{#2 \colon #1  \ar@{^{(}->}[r] &  #3}}
\newcommand{\cover}[3]{\xymatrix{#2 \colon #1  \ar@{-|>}[r] &  #3}}
\newcommand{\frecciasopra}[3]{\xymatrix{ #1  \ar[r]^{#2} &  #3}}

\newcommand{\pbmorph}[2]{#1^{\ast}#2} 
\newcommand{\duefreccia}[3]{\xymatrix@C=0.5cm{#2 \colon #1  \ar@{=>}[r] &  #3}}
\newcommand{\modificazione}[3]{\xymatrix@C=0.5cm{#2 \colon #1  \ar@{~>}[r] &  #3}}
\newcommand{\doctrine}[2]{\xymatrix{#2 \colon #1^{\op}  \ar[r] & \infsl }}

\newcommand{\duemorfismo}[6]{\xymatrix{
#1^{\op} \ar[rrd]^#2_{}="a" \ar[dd]_{#3^{\op}}\\
&& \infsl\\
#5^{\op}  \ar[rru]_#6^{}="b"
\ar_{#4}  "a";"b"}}
\newcommand{\comsquare}[8]{ \xymatrix@+1pc{ 
#1 \ar[r]^{#5} \ar[d]_{#6} & #2 \ar[d]^{#7} \\
#3 \ar[r]_{#8} & #4 
}}
\newcommand{\pullback}[8]{ \xymatrix@+2pc{ 
#1 \pullbackcorner \ar[r]^{#5} \ar[d]_{#6} & #2 \ar[d]^{#7} \\
#3 \ar[r]_{#8} & #4 
}}
\newcommand{\quadratocomm}[8]{ \xymatrix@+1pc{ 
#1 \ar[r]^{#5} \ar[d]_{#6} & #2 \ar[d]^{#7} \\
#3 \ar[r]_{#8} & #4 
}}
\newcommand{\comsquarelargo}[8]{ \xymatrix@+1pc{ 
#1 \ar[rr]^{#5} \ar[d]_{#6} && #2 \ar[d]^{#7} \\
#3 \ar[rr]_{#8} && #4 
}}
\newcommand{\parallelmorphisms}[4]{\xymatrix@+1pc{
#1 \ar @<+4pt>[r]^{#2} \ar @<-4pt>[r]_{#3} & #4
}}
\newcommand{\relation}[4]{\xymatrix@+1pc{
\angbr{#2}{#3}\colon #1 \ar @<+4pt>[r] \ar @<-4pt>[r] & #4
}}
\newcommand{\frecceparalleleopposte}[4]{\xymatrix@+1pc{
#1 \ar@<+4pt>[r]^{#2} \ar@<-4pt>@{<-}[r]_{#3} & #4
}}
\newcommand{\equalizer}[6]{\xymatrix@+1pc{
#1 \ar[r]^{#2} & #3 \ar @<+4pt>[r]^{#4} \ar @<-4pt>[r]_{#5} & #6
}}
\newcommand{\coequalizer}[6]{\xymatrix@+1pc{
 #1 \ar @<+4pt>[r]^{#2} \ar @<-4pt>[r]_{#3} & #4 \ar[r]^{#5} & #6
}}
\newcommand{\sottoggetto}[2]{\xymatrix{
#1 \ar@{>->}[r] & #2
}}
\newcommand{\subobject}[3]{\xymatrix{
#1 \ar@{>->}[r]^{#2} & #3
}}

\newcommand{\pullbackcorner}[1][ul]{\save*!/#1+1.2pc/#1:(1,-1)@^{|-}\restore}

\def\ED{\operatorname{\mathbf{ExD}}}            
\def\PD{\operatorname{\mathbf{PD}}}             

\def\EED{\operatorname{\mathbf{EED}}}           
\def\Reg{\operatorname{\mathbf{Reg}}}
\def\excat{\operatorname{\mathbf{Xct}}}        
\def\mR{\mathbb{\mathcal{R}}}

\def\mA{\mathcal{A}}

\def\mC{\mathcal{C}}
\def\mX{\mathcal{X}}
\def\mD{\mathcal{D}}

\def\mM{\mathcal{M}}
\def\mN{\mathcal{N}}
\def\mH{\mathcal{H}}

\def\mG{\mathcal{G}}
\def\cT{\mathbb{T}}    

\def\mT{\mathrm{T}}

\newbox\erove \setbox\erove=\hbox{\reflectbox{E}}
\def\Einv{\exists}
\def\pr{\operatorname{ pr}}

\def\id{\operatorname{ id}}

\def\op{\operatorname{ op}}

\def\mono{\operatorname{ mono}}

\def\existential{\operatorname{ ex}}

\def\theory{\operatorname{ \cT}}
\def\Sub{\operatorname{ Sub}}

\newcommand{\exactcomp}[1]{(#1)_{\ex / \reg}}      
\newcommand{\regularcomp}[1]{(#1)_{\reg / \lex}}   
\newcommand{\exactcomplex}[1]{(#1)_{\ex / \lex}}   
\def\reg{\operatorname{reg}}
\def\ex{\operatorname{ex}}
\def\lex{\operatorname{lex}}
\def\Set{\operatorname{\mathrm{Set}}}
\def\Rec{\operatorname{\mathbf{Rec}}}

\def\signature{\operatorname{\mathbf{Sg}}}

\def\:{\colon}
\def\infsl{\operatorname{\mathbf{InfSL}}}

\def\2cellidentity{i}
\def\1cellidentity{1}
 
\def\Mcat{\operatorname{\mathcal{M}-\mathbf{Cat}}}
\def\CompAdj{\operatorname{\mathbf{CE}_c}}
\newcommand{\subobjdoctrine}[2]{\xymatrix{\Sub_{#2} \colon #1^{\op}  \ar[r] & \infsl }}
\newcommand{\Msubobj}[2]{\Sub_{#1}(#2)}

\def\Pred#1{{\ensuremath{{\mathpzc{Prd}\kern-.4ex_{{#1}}}}}}

\def\lang{\mathcal{L}}
\def\syntdoc{\mathrm{LT}}

\def\singleton{\mathrm{sing}}
\newcommand{\pca}[1]{\mathbb{#1}}

\def\mP{\mathcal{P}}
\def\set{\mathrm{Set}}
\newcommand{\pwset}[1]{\mathrm{P}(#1)}
\newcommand{\pwsetnonempty}[1]{\mathrm{P}^{\ast}(#1)}

\newcommand{\quantfree}[1]{$#1$-existential-free}

\newcommand{\qflamb}{$\Lambda$-existential-free}
\newcommand{\eqflamb}{enough-$\Lambda$-existential-free objects}
\newcommand{\Pqf}{\ovln{P}}
\newcommand{\cm}{\operatorname{comp}}

\newcommand{\trdc}{\Upsilon}

\newcommand{\loc}{\mathcal{A}}

\newcommand{\framecat}{\mathbf{Frm}}
\newcommand{\meetcat}{\mathbf{M}}
\newcommand{\supercohframe}{\mathbf{SCohFrm}}
\newcommand{\supercomp}{\mathcal{N}}
\newcommand{\sheaves}[1]{\mathbf{Sh}(#1)}
\newcommand{\presheaves}[1]{\mathbf{PreSh}(#1)}
\newcommand{\assembly}[1]{\mathpzc{Asm}(#1)}
\newcommand{\parassembly}[1]{\mathpzc{PAsm}(#1)}
\def\Eff{\mathsf{RT}}
\def\unvpropmap{\varrho}

\newcommand{\regularcompdoctrine}[1]{\mathsf{Reg}(#1)}

\newcommand{\lclass}{left class of morphisms}

\def\conjdoctrine{conjunctive doctrine}
\def\Lamdoctrine{left class doctrine}
\def\Lamdoctrines{left class doctrines}
\def\LamPR{\operatorname{\mathbf{Lc}-\mathbf{CD}}}  
\def\LamEX{\operatorname{\mathbf{Lc}-\mathbf{ExD}}}  